\pdfoutput=1

\documentclass[11pt]{article}

\usepackage{amsmath,amsthm,verbatim,amssymb,amsfonts,amscd}
\usepackage{comment}
\usepackage{empheq} 
\usepackage{mathabx}
\usepackage{mathrsfs} 
\usepackage{authblk} 
\usepackage{mathtools} 
\usepackage[shortlabels]{enumitem}
\usepackage[title]{appendix}

\usepackage{accents}

\usepackage[utf8]{inputenc}
\usepackage[T1]{fontenc}

\usepackage{graphicx}
\newcommand\smallO{
  \mathchoice
    {{\scriptstyle\mathcal{O}}}
    {{\scriptstyle\mathcal{O}}}
    {{\scriptscriptstyle\mathcal{O}}}
    {\scalebox{.7}{$\scriptscriptstyle\mathcal{O}$}}
  }

\theoremstyle{plain}
\newtheorem{theorem}{Theorem}[section]
\newtheorem{corollary}[theorem]{Corollary}
\newtheorem{lemma}[theorem]{Lemma}
\newtheorem{proposition}[theorem]{Proposition}

\newtheorem{remark}[theorem]{Remark}

\theoremstyle{definition}

\numberwithin{equation}{section}

\newcommand{\R}{\mathbb{R}}

\newcommand{\N}{\mathbb{N}}
\newcommand{\Z}{\mathbb{Z}}

\newcommand{\mx}{\mathrm{d}x}
\newcommand{\my}{\mathrm{d}y}

\DeclareMathOperator{\dist}{{\rm dist}}

\DeclareMathOperator{\FT}{\mathcal{F}}           
\DeclareMathOperator{\IFT}{\mathcal{F}^{\,-1}}   

\DeclareMathOperator{\Dil}{\mathrm{Dil}}         

\begin{document}

\title{Strichartz estimates for mixed homogeneous surfaces in three dimensions}

\author{Ljudevit Palle
\thanks{The author has been supported by Deutsche Forschungsgemeinschaft,
project number 237750060, \emph{Fragen der Harmonischen Analysis im Zusammenhang mit Hyperfl\"{a}chen.}
}
}

\affil{Christian-Albrechts-Universit\"{a}t zu Kiel}

\date{\today}

\maketitle

\begin{abstract}
We obtain sharp mixed norm Strichartz estimates associated to mixed homogeneous surfaces in $\R^3$.
Both cases with and without a damping factor are considered.
In the case when a damping factor is considered our results yield
a wide generalization of a result of Carbery, Kenig, and Ziesler \cite{CKZ13}.
The approach we use is to first classify all possible singularities locally,
after which one can tackle the problem by appropriately modifying the methods from the paper of Ginibre and Velo \cite{GV92},
and by using the recently developed methods by Ikromov and M\"{u}ller \cite{IM16}.
\end{abstract}

\tableofcontents




\section{Introduction}
\label{section_intro}

Let us fix a pair $\alpha = (\alpha_1, \alpha_2) \in (0, \infty)^2$,
denote $|\alpha| \coloneqq \alpha_1 + \alpha_2$, and introduce
its associated $\alpha$-mixed homogeneous dilations in $\R^2$ by
\begin{align*}
\delta_r(x_1, x_2) = (r^{\alpha_1} x_1, r^{\alpha_2} x_2), \qquad r > 0.
\end{align*}
The main goal of this article is to study Strichartz estimates for a fixed mixed homogeneous surface $S$,
i.e., a surface given as the graph of a fixed smooth function $\phi : \R^2 \setminus \{0\} \to \R$
which is $\alpha$-mixed homogeneous of degree $\rho$:
\begin{align}
\label{section_intro_alpha_dilations}
\phi \circ \delta_r(x_1, x_2) = r^{\rho} \phi(x_1,x_2), \qquad r > 0.
\end{align}
We may and shall assume without loss of generality that $\rho \in \{-1, 0, 1\}$.
Both $\alpha$ and $\rho$ shall be fixed throughout the article.
Note that when $\rho = -1$ the function $\phi$ has a singularity at the origin.

As is well known,
Strichartz estimates are directly related to Fourier restriction estimates
and we shall be in particular interested in the mixed norm estimate
\begin{align}
\label{section_intro_restriction_estimate}
\Vert \widehat{f} \Vert_{L^2(\mathrm{d}\mu)} \leq C \Vert f\Vert_{L^{p_3}_{x_3}(L^{p_1}_{(x_1,x_2)})}, \qquad f \in \mathcal{S}(\R^3),
\end{align}
where $\mu$ is the surface measure
\begin{align}
\label{section_intro_restriction_measure}
\langle \mu, f \rangle = \int_{\R^2\setminus\{0\}} f(x_1,x_2,\phi(x_1,x_2)) \, \mathcal{W}(x_1,x_2) \, \mathrm{d}x
\end{align}
and $p = (p_1,p_3) \in (1,2)^2$.
The weight $\mathcal{W} \geq 0$ is added in order to insure that the measure has a scaling invariance
which will enable us to reduce global estimates to local ones by a Littlewood-Paley argument.
We take $\mathcal{W}$ to be $\alpha$-mixed homogeneous of degree $2 \vartheta$
and consider two particular cases.
The function $\mathcal{W}$ will be either equal to
\begin{align}
\label{section_intro_weight_growth}
|x|_{\alpha}^{2\vartheta} = (|x_1|^{1/\alpha_1} + |x_2|^{1/\alpha_2})^{2\vartheta},
\end{align}
%
or equal to the Hessian determinant of $\phi$ (denoted by $\mathcal{H}_\phi$)
raised to the power $|\cdot|^{\sigma}$, $\sigma \in [0, 1/2)$, i.e.,
\begin{align}
\label{section_intro_weight_Hessian}
|\mathcal{H}_\phi(x)|^{\sigma} =
\Bigg|
\det
\begin{bmatrix}
    \partial^2_{x_1} \phi & \partial_{x_1} \partial_{x_2} \phi \\
    \partial_{x_1} \partial_{x_2} \phi & \partial^2_{x_2} \phi
\end{bmatrix}
\Bigg|^{\sigma}.
\end{align}
%
One can easily show that the Hessian determinant of $\phi$ is $\alpha$-mixed homogeneous of degree $2(\rho-|\alpha|)$,
and so in the case when $\mathcal{W}$ equals \eqref{section_intro_weight_Hessian}
the relation between $\vartheta$ and $\sigma$ is $\vartheta = \sigma(\rho-|\alpha|)$.
We shall later determine $\vartheta$ in Subsection \ref{subsection_prelim_scaling}
(and in particular in Proposition \ref{prop_prelim_scaling})
so that the Fourier restriction estimate for $\mu$ is invariant under scaling.
This choice depends in general on $p = (p_1,p_3)$.

\medskip

Oscillatory integrals, Fourier restriction estimates, and other problems
related to homogeneous and mixed homogeneous surfaces have already been previously studied
and a sample of such works are
\cite{DZ19}, \cite{Schwend}, \cite{Gb18}, \cite{IKM05}, \cite{IS96}, \cite{FGU04}, \cite{FU09}, \cite{CKZ13}.

The case of general $L^p-L^2$ Fourier restriction in $\R^3$ with respect to the Euclidean measure
was recently solved in \cite{IM16} by Ikromov and M\"uller
for a wide class of smooth surfaces in $\R^3$, including all the analytic ones.
Mixed norm estimates were shown in \cite{Pa19} for surfaces given as graphs of functions $\phi$ in adapted coordinates
and also for analytic functions $\phi$ satisfying $h_{\text{lin}}(\phi) < 2$.

In \cite{CKZ13} Carbery, Kenig, and Ziesler considered
the case with the weight \eqref{section_intro_weight_Hessian} for ``isotropically'' homogeneous
(i.e., when $\alpha_1 = \alpha_2$) polynomials $\phi$.
Since the weight \eqref{section_intro_weight_Hessian} has roots at the degenerate points,
the estimate \eqref{section_intro_restriction_estimate} holds for a wider range of exponents
compared to the case when the weight \eqref{section_intro_weight_growth} is used.
The use of this so called mitigating or damping factor goes back to Sj\"{o}lin \cite{Sj74} (see also \cite{CDMM90}, \cite{Dr90}, \cite{KPV91}).
When one uses the above damping factor \eqref{section_intro_weight_Hessian}
one can even obtain estimates for certain classes of flat surfaces \cite{CZ02}, \cite{AS06}, \cite{CKZ07}.
On the other hand, weak type estimates have been obtained by Oberlin in the Tomas-Stein range in \cite{Ob12}
for a wide class of surfaces having a bounded generic multiplicity (see also \cite{Ob04}).
Let us also mention a recent result by Gressman \cite{Gm16}
where he obtained decay estimates for damping oscillatory integrals for a certain class of singularities.

In this article we shall first classify the possible local singularities
for mixed homogeneous surfaces (see Proposition \ref{proposition_normal_forms} below)
and then either apply the Fourier restriction estimates obtained in \cite{IM16} and \cite{Pa19},
or use the techniques from these articles, and also from the article \cite{GV92} (see also \cite{KT98}), to obtain sharp estimates.
In particular, we obtain a wide generalization of the Fourier restriction estimate in \cite{CKZ13}
with methods which are more elementary,
and in particular avoiding any use of results from algebraic topology or algebraic geometry.


\medskip
In order to state the main results of this paper
(namely, Theorem \ref{theorem_Strichartz_mitigating}, Theorem \ref{theorem_Strichartz},
Proposition \ref{proposition_normal_forms}, and Corollary \ref{corollary_Strichartz_equation})
we first recall certain concepts and introduce a few conditions.
Recall that we call a smooth function $\varphi : \Omega \subseteq \R^n \to \R$ a function of \emph{finite type} at $P \in \Omega$
if there exists a multiindex $\tau$ such that $\partial^\tau \varphi (P) \neq 0$.
We shall consider the following two conditions on $\phi$:
\begin{itemize}
\item[{\bf (H1)}]
At any given point $(x_1,x_2) \neq (0,0)$ where the Hessian determinant of $\phi$ vanishes
at least one of the mappings
$t \mapsto \partial_1^2 \phi(t, x_2)$ or
$t \mapsto \partial_2^2 \phi(x_1, t)$ is of finite type at $t = x_1$ (resp. $t = x_2$).
\item[{\bf (H2)}]
The Hessian determinant $\mathcal{H}_\phi$ is not flat at any point $x \neq 0$.
\end{itemize}
It actually suffices to check the conditions only at points $(x_1,x_2)$ in say a unit circle, by homogeneity.
Furthermore, we remark that the condition (H2) is stronger
than the condition (H1) (this follows from the calculations in Subsection \ref{subsection_forms_flat} below).

Let us now introduce a further condition and two new quantities.
For a point $v \in \R^2 \setminus \{0\}$ let us define the function
\begin{align*}
\phi_v(x) \coloneqq \phi(x+v) - \phi(v) - x \cdot \nabla \phi(v).
\end{align*}
Then we shall often consider whether the following condition is satisfied at $v$:
\begin{itemize}
\item[{\bf (LA)}]
There is a linear coordinate change which is adapted to $\phi_v$ at the origin.
\end{itemize}
Compare with the negation of this condition in \cite[Section 1.2]{IM16}.

Let us furthermore denote the \emph{linear height} of $\phi_v$ by $h_{\text{lin}}(\phi,v)$ and its \emph{Newton height} by $h(\phi,v)$.
Both linear height and Newton height are defined in the usual manner (see \cite{IM16}, \cite{Pa19}).
We define the \emph{global linear height} $h_{\text{lin}}(\phi)$ and the \emph{global Newton height} $h(\phi)$
by the respective expressions
\begin{align}
\label{subsection_intro_results_global_heights}
h_{\text{lin}}(\phi) &= \sup_{v \in \mathbb{S}^1} h_{\text{lin}}(\phi,v), & h(\phi) &= \sup_{v \in \mathbb{S}^1} h(\phi,v).
\end{align}
It will be clear from Section \ref{section_forms} that
$h_{\text{lin}}(\phi,v)$ and $h(\phi,v)$ do not change along the \emph{homogeneity curve through} $v$
defined as the curve
\begin{align*}
r \mapsto (r^{\alpha_1} v_1, r^{\alpha_2} v_2), \qquad r > 0,
\end{align*}
and therefore in the above definitions of global linear height and global Newton height
one could have taken the supremum over the set $\R^2 \setminus \{0\}$ too.

\begin{theorem}
\label{theorem_Strichartz_mitigating}
Let $\phi$ be mixed homogeneous satisfying condition (H2).
Let $\mu$ be the measure defined as in \eqref{section_intro_restriction_measure}
with $\mathcal{W}(x) = |\mathcal{H}_\phi(x)|^{\sigma}$ for some fixed $\sigma \geq 0$.
If $\sigma \in [0,\frac{1}{3}]$, then the Fourier restriction estimate \eqref{section_intro_restriction_estimate} holds true for
\begin{align*}
\Big(\frac{1}{p_1'}, \frac{1}{p_3'}\Big) = \Big(\frac{1}{2} - \sigma, \sigma \Big).
\end{align*}
If (LA) is satisfied at all points $v \neq 0$, then the estimate holds true even if $\sigma \in [0,\frac{1}{2})$.
In particular, if $\alpha_1 = \alpha_2$, then (LA) is satisfied at all points $v \neq 0$
and the estimate holds true for any $\sigma \in [0,\frac{1}{2})$.
\end{theorem}
Several comments are in order.
Firstly, precise conditions for when the (LA) condition is satisfied at $v \neq 0$
can be checked by using the normal form tables in Section \ref{section_forms}
(note that in the Proposition \ref{proposition_normal_forms} below, where the normal forms are listed,
only the Normal form (vi) is not in adapted coordinates).
That one is restricted to $0 \leq \sigma \leq 1/3$ in the case when (LA) is not satisfied
is a consequence of a Knapp-type example, as we shall show in Subsubsection \ref{subsubsection_strichartz_mitigating_vi_Knapp}.

Secondly, in the case when $\rho = 1 = |\alpha|$, one can extend the above estimate to the range where
\begin{align*}
\frac{1}{p_1'} + \frac{1}{p_3'} = \frac{1}{2}, \qquad \frac{1}{p_3'} \leq \sigma.
\end{align*}
The reason for this is that $\rho = 1 = |\alpha|$ implies that
the weight $\mathcal{W}$ (and the Hessian determinant) are $\alpha$-mixed homogeneous of degree $0$, and hence bounded on $\R^2$,
and so the estimate for $(p_1,p_3) = (2,1)$ follows trivially by Plancherel.

Finally, let us mention that the most interesting part of the proof of the above theorem
is the proof of Fourier restriction for the Normal form (v) from Proposition \ref{proposition_normal_forms},
which is to be found in Subsection \ref{subsection_strichartz_mitigating_v}.
There we need to estimate the Fourier transform of a certain measure,
and for this we perform a natural decomposition of this measure.
What is remarkable is that at the critical frequencies only $\mathcal{O}(1)$ decomposition pieces contribute
to the size of the Fourier transform.
Interestingly, the same thing already happens in the much easier case of Normal form (iv).

\medskip

In the case of the other weight (which has no roots away from the origin) we have:
\begin{theorem}
\label{theorem_Strichartz}
Let $\phi$ be mixed homogeneous satisfying condition (H1).
Let $\mu$ be the measure defined as in \eqref{section_intro_restriction_measure} with $\mathcal{W}(x) = |x|_{\alpha}^{2\vartheta}$.
If the exponents $(p_1, p_3) \in (1,2)^2$ and $\vartheta \in \R$ satisfy
\begin{align*}
\frac{1}{p_1'} + \frac{h_{\text{lin}}(\phi)}{p_3'} &\leq \frac{1}{2}, &
\frac{1}{p_3'} &\leq \frac{1}{2 h(\phi)}, &
\vartheta &= \frac{|\alpha|}{p_1'}+\frac{\rho}{p_3'} - \frac{|\alpha|}{2},
\end{align*}
then the Fourier restriction estimate \eqref{section_intro_restriction_estimate} holds true.
\end{theorem}
We remark that the quantity $\vartheta$ in the above theorem is allowed to be negative.

As a special case of Theorem \ref{theorem_Strichartz_mitigating} we obtain:
\begin{corollary}
\label{corollary_Strichartz_mitigating}
Let $\phi$ be any mixed homogeneous polynomial in $\R^2$
and let $\mu$ be the measure defined as in \eqref{section_intro_restriction_measure} with $\mathcal{W}(x) = |\mathcal{H}_\phi(x)|^{1/4}$.
Then the Fourier restriction estimate \eqref{section_intro_restriction_estimate} holds true for $p_1' = p_3' = 4$.
\end{corollary}
In the case of the above corollary we note that the Hessian determinant can either vanish identically,
or it does not vanish of infinite order anywhere,
since it is necessarily a nonzero mixed homogeneous polynomial.
But the case when the Hessian determinant vanishes identically is trivial,
so we are indeed within the scope of Theorem \ref{theorem_Strichartz_mitigating}.

When one considers ``isotropically'' homogeneous polynomials (i.e., when $\alpha_1 = \alpha_2$)
Corollary \ref{corollary_Strichartz_mitigating} recovers the main result of \cite{CKZ13}.
The strategy of proof in \cite{CKZ13} was to first perform certain decompositions of the surface measure in order to get
appropriate control over the size of $\nabla \phi$ and the Hessian determinant $\mathcal{H}_\phi$,
after which one applies an $L^4$ argument,
as the $L^{4/3}(\R^3) \to L^2(\mathrm{d}\mu)$ Fourier restriction estimate is
equivalent to the $L^2(\mathrm{d}\mu) \to L^4(\R^3)$ extension estimate.

Our proofs of Theorem \ref{theorem_Strichartz_mitigating} and Theorem \ref{theorem_Strichartz}
are based on the following intermediary result:
\begin{proposition}
\label{proposition_normal_forms}
Let $v \in \R^2 \setminus \{0\}$, let $\phi$ be as above $\alpha$-mixed homogeneous of degree $\rho$, and
let us assume that it satisfies condition (H1) and
that it is degenerate at $v$ (i.e., the Hessian determinant vanishes at $v$).
Then after a linear transformation of coordinates the function
$\phi_v(x) \coloneqq \phi(x+v) - \phi(v) - x \cdot \nabla \phi(v)$
and its Hessian determinant $\mathcal{H}_{\phi_v}$
assume precisely one of the following local normal forms at the origin:
\begin{itemize}
\item[(i)]
$\phi_v(x) = x_2^k r(x) + \varphi(x)$, $k \geq 2$, $\varphi$ is flat, \\
$\mathcal{H}_{\phi_v}(x) = x_2^{\tilde{k}+2k-2} q(x)$, $\tilde{k} \geq 0$, or $\mathcal{H}_{\phi_v}$ is flat, \\
and in case when $\mathcal{H}_{\phi_v}$ is not flat, then $\varphi$ vanishes identically,
\item[(ii)]
$\phi_v(x) = x_1^2 r_1(x_1) + x_2^k r_2(x)$, $k \geq 3$, \\
$\mathcal{H}_{\phi_v}(x) = x_2^{k-2} q(x)$,
\item[(iii)]
$\phi_v(x) = x_1^2 r_1(x) + x_2^k r_2(x)$, $k \geq 3$, \\
$\partial_2^j r_1(0) = c(\phi,v) \, j \, \partial_2^{j-1} r_1(0), \, j = 1, \ldots, k-1$, for some constant $c(\phi,v) \neq 0$, \\
$\mathcal{H}_{\phi_v}(x) = x_2^{k-2} q(x)$,
\item[(iv)]
$\phi_v(x) = x_1^2 r_1(x_1) + (x_2-x_1^2 \psi(x_1))^k r_2(x)$, $k \geq 3$, \\
$\mathcal{H}_{\phi_v}(x) = (x_2-x_1^2 \psi(x_1))^{k-2} q(x)$,
\item[(v)]
$\phi_v(x) = x_1^2 r_1(x) + (x_2 - x_1^2 \psi(x_1))^k r_2(x)$, $k \geq 3$, \\
$\partial_2^j r_1(0) = c(\phi,v) \, j \, \partial_2^{j-1} r_1(0), \, j = 1, \ldots, k-1$, for some constant $c(\phi,v) \neq 0$, \\
$\mathcal{H}_{\phi_v}(x) = (x_2 - x_1^2 \psi(x_1))^{k-2} q(x)$,
\item[(vi)]
$\phi_v(x) = (x_2-x_1^2 \psi(x_1))^k r(x)$, $k \geq 2$, \\
$\mathcal{H}_{\phi_v}(x) = (x_2-x_1^2 \psi(x_1))^{2k-3} q(x)$.
\end{itemize}
In all the above cases the appearing functions are smooth
and do not vanish at the origin, i.e., $r(0), r_1(0), r_2(0), q(0), \psi(0) \neq 0$
(except for the function $\varphi$ which is flat),
and the root of the function $x \mapsto x_2 - x_1^2 \psi(x_1)$ corresponds to the homogeneity curve through $v$.
If condition (H2) is satisfied,
then the function $\varphi$ in case (i) always vanishes identically and the Hessian determinant is never flat.
Finally, if $\alpha_1 = \alpha_2$, then only Normal forms (i) and (ii) can appear.
\end{proposition}
In cases (i) and (ii) one has further subcases (see Subsection \ref{subsection_forms_considerations})
of technical nature, so we left them out in the above proposition.
We also note that only in case (vi) the function $\phi_v$ is not in adapted coordinates
(and the adapted coordinates can be achieved only through
the nonlinear transformation $(x_1,x_2) \mapsto (x_1, x_2 + x_1^2 \psi(x_1))$),
but it is linearly adapted.

\medskip

The idea to apply Fourier restriction estimates to obtain a priori estimate for PDEs goes back to Strichartz \cite{St77}.
In our case one can apply the above results to obtain Strichartz estimates
for the nonhomogeneous initial problem
\begin{align*}
\begin{cases}
(\partial_t - i \phi(D))u(x,t) &= F(x,t), \quad (x,t) \in \R^2 \times (0,\infty),\\
\qquad \qquad \,\,\quad u(x,0) &= G(x), \quad \quad\,x \in \R^2,
\end{cases}
\end{align*}
where $F \in \mathcal{S}(\R^3)$, $G \in \mathcal{S}(\R^2)$.
Namely, by an application of the Christ-Kiselev lemma \cite{CK01} one gets the following result:
\begin{corollary}
\label{corollary_Strichartz_equation}
Let $\phi$, $\mathcal{W}$, and $(p_1,p_3) \in (1,2)^2$ be either as in Theorem \ref{theorem_Strichartz_mitigating} or Theorem \ref{theorem_Strichartz},
and let us furthermore assume that $\rho \in \{0,1\}$.
Then for the above nonhomogeneous PDE one has the a priori estimate
\begin{align*}
\Vert u \Vert_{L^{p_3'}_{t}(L^{p_1'}_{(x_1,x_2)})}
  \leq
  C_1 \Vert \mathcal{W}^{-1/2} \FT G \Vert_{L^2(\R^2)} +
  C_2 \Vert \IFT_{(x_1,x_2)} (\mathcal{W}^{-1} \FT_{(x_1,x_2)} F) \Vert_{L^{p_3}_{t}(L^{p_1}_{(x_1,x_2)})},
\end{align*}
where $\FT_{(x_1,x_2)}$ is the partial Fourier transformation in the $x = (x_1, x_2)$ direction.
\end{corollary}
In the case when $\mathcal{W}$ is the function $|\cdot|_{\alpha}^{2\vartheta}$
the norms on the right hand side are a type of homogeneous anisotropic Sobolev norms \cite[Chapter 5]{Tr06}
(in particular, note that $\Vert \mathcal{W}^{-1/2} \FT G \Vert_{L^2(\R^2)} = \Vert \IFT \mathcal{W}^{-1/2} \FT G \Vert_{L^2(\R^2)}$).

Since the procedure of how to obtain the corresponding Strichartz estimate from a Fourier restriction estimate is mostly standard
we have deferred the sketch of the proof of Corollary \ref{corollary_Strichartz_equation} to Appendix \ref{appendix_CKlemma}.

\medskip

The article is structured in the following way.
In Section \ref{section_prelim} we first perform some elementary reductions.
Since the proofs of Theorem \ref{theorem_Strichartz_mitigating} and Theorem \ref{theorem_Strichartz}
are essentially based on Proposition \ref{proposition_normal_forms},
we first prove this proposition (and even obtain slightly more precise results) in Section \ref{section_forms}.
Subsequently we prove Theorems \ref{theorem_Strichartz_mitigating} and \ref{theorem_Strichartz}
in the respective Sections \ref{section_strichartz_mitigating} and \ref{section_strichartz_mixed}.
In the appendix we then give a sketch of the proof of Corollary \ref{corollary_Strichartz_equation}.

In this paper we use the symbols $\sim, \lesssim, \gtrsim, \ll, \gg$ with the following meaning.
If two nonnegative quantities $A$ and $B$ are given, then
by $A \ll B$ we mean that there exists a sufficiently small positive constant $c$ such that $A \leq cB$,
and by $A \lesssim B$ we mean that there exists a (possibly large) positive constant $C$ such that $A \leq CB$.
The relation $A \sim B$ means that there exist positive constants $C_1 \leq C_2$ such that $C_1 A \leq B \leq C_2 A$ is satisfied.
Relations $A \gg B$ and $A \gtrsim B$ are defined analogously.
Sometimes the implicit constants $c$, $C$, $C_1$, and $C_2$ depend on certain parameters $p$,
and in order to emphasize this dependence we shall write for example $\lesssim_p$, $\sim_p$, and so on.

We also use the symbols $\chi_0$, $\chi_1$, $r$, and $q$ generically in the following way.
We require $\chi_0$ to be supported in a neighbourhood of the origin and identically equal to $1$ near the origin.
On the other hand, we require $\chi_1$ to be supported away from the origin and identically equal to $1$ on an open neighbourhood of $\pm 1 \in \R$.
Sometimes, when several $\chi_0$ or $\chi_1$ appear within the same formula, they may designate different functions.
The functions $r$ and $q$ (also used with subscripts and tildes) shall be used generically as smooth functions which are nonvanishing at the origin.
Occasionally they can also be flat at the origin, in which case we state this explicitly.


\section{Preliminary reductions}
\label{section_prelim}


\subsection{Rescaling and reduction to local estimates}
\label{subsection_prelim_scaling}

As mentioned, the measure we consider is
\begin{align*}
\langle \mu, f \rangle = \int_{\R^2\setminus\{0\}} f(x_1,x_2,\phi(x_1,x_2)) \, \mathcal{W}(x_1,x_2) \, \mathrm{d}x,
\end{align*}
where $\mathcal{W}$ is nonnegative, continuous on $\R^2\setminus\{0\}$, and $\alpha$-mixed homogeneous of degree $2 \vartheta$.
In this subsection we determine the degree of homogeneity $2 \vartheta$ so that
the global Fourier restriction estimate \eqref{section_intro_restriction_estimate}
becomes equivalent to the local one.
By this we mean the following.
Let us take a partition of unity $(\eta_j)_{j \in \Z}$ in $\R^2\setminus\{0\}$:
\begin{align}
\label{section_prelim_partition}
\sum_{j \in \Z} \eta_j(x) = 1, \qquad x \neq 0,
\end{align}
such that $\eta_j = \eta \circ \delta_{2^{-j}}$
for some $\eta = \eta_0 \in C_c^\infty(\R^2)$ supported away from the origin.
Let us consider the measures
\begin{align}
\label{section_prelim_local_measure}
\langle \mu_j, f \rangle \coloneqq \int_{\R^2} f(x,\phi(x)) \, \eta_j(x) \, \mathcal{W}(x) \, \mathrm{d}x,
\end{align}
which now satisfy $\mu = \sum_{j \in \Z} \mu_j$,
and let us furthermore assume that we have the local estimate for some $j_0 \in \Z$:
\begin{align*}
\Vert \widehat{f} \Vert_{L^2(\mathrm{d}\mu_{j_0})} \leq C \Vert f\Vert_{L^{p}},
\end{align*}
where $L^{p} = L^{p_3}_{x_3}(L^{p_1}_{(x_1,x_2)})$.
We want to determine the degree of homogeneity of $\mathcal{W}$ so that
the Fourier restriction estimate is invariant under the dilations $\delta_r$,
i.e., that we have
\begin{align}
\label{section_prelim_local_estimate_j}
\Vert \widehat{f} \Vert_{L^2(\mathrm{d}\mu_{j})} \leq C \Vert f\Vert_{L^{p}}
\end{align}
for all $j \in \Z$ whenever the estimate is true for some $j_0 \in \Z$.
In this case, and if $(p_1,p_3) \in (1,2]^2$,
a standard Littlewood-Paley argument will then yield
\begin{align*}
\Vert \widehat{f} \Vert_{L^2(\mathrm{d}\mu)} \leq C \Vert f\Vert_{L^{p}}.
\end{align*}
To summarize, we have:
\begin{proposition}
\label{prop_prelim_scaling}
Let $\mathcal{W}$ be $\alpha$-mixed homogeneous of degree $2\vartheta$, not identically zero, and continuous on $\R^2 \setminus \{0\}$,
let $\mu$ be defined as in \eqref{section_intro_restriction_measure}, and let $p_1, p_3 \in (1,2]$.
Then the Fourier restriction estimate \eqref{section_intro_restriction_estimate} for $\mu$
is equivalent to the Fourier restriction estimate \eqref{section_prelim_local_estimate_j} for
the measure $\mu_j$ for any $j \in \Z$ (as defined in \eqref{section_prelim_local_measure})
if and only if
\begin{align}
\label{section_prelim_hom_condition}
\vartheta = \frac{|\alpha|}{p_1'}+\frac{\rho}{p_3'} - \frac{|\alpha|}{2}
\end{align}
is satisfied.
\end{proposition}
\begin{proof}
Let us first determine what $2 \vartheta$, the degree of homogeneity of $\mathcal{W}$,
needs to be in order that \eqref{section_prelim_local_estimate_j} holds true
for all $j \in \Z$ whenever the it holds true for some $j_0 \in \Z$.
Recall that $| \delta_r x |_\alpha = r |x|_\alpha$.
Inspecting the definition \eqref{section_prelim_local_measure} of $\mu_j$ one gets:
\begin{align*}
\langle \mu_j, f \rangle = 2^{j |\alpha| + 2j \vartheta} \langle \mu_0, \Dil_{(2^{-j\alpha_1},2^{-j\alpha_2},2^{-j\rho})} f \rangle,
\end{align*}
where $(\Dil_{(\lambda_1, \lambda_2, \lambda_3)} f)(x_1, x_2, x_3) = f(\lambda_1^{-1} x_1, \lambda_2^{-1} x_2, \lambda_3^{-1} x_3)$.
The above relation can be interpreted as
\begin{align*}
\mu_j = 2^{2j \vartheta - j \rho} \Dil_{(2^{j\alpha_1}, 2^{j\alpha_2}, 2^{j \rho})} \mu_0.
\end{align*}
Let us assume that we have for some $j \in \Z$ the estimate
\begin{align*}
\langle \mu_j, |\widehat{f}|^2 \rangle =
\Vert \widehat{f} \Vert^2_{L^2(\mathrm{d}\mu_{j})}
\leq C^2 \Vert f\Vert^2_{L^{p}}.
\end{align*}
Since the Fourier transform behaves well with respect to dilations $\Dil_{(\lambda_1, \lambda_2, \lambda_3)}$,
we may rescale the above estimate and get
\begin{align*}
\Vert \widehat{f} \Vert_{L^2(\mathrm{d}\mu_{0})}
   \leq C 2^{-j|\alpha|/2-j\vartheta+j(\alpha_1/p_1'+\alpha_2/p_1'+\rho/p_3')} \Vert f\Vert_{L^{p}}.
\end{align*}
From this one sees that we need precisely \eqref{section_prelim_hom_condition}
in order for the constant in \eqref{section_prelim_local_estimate_j} to be independent of $j$.
If \eqref{section_prelim_hom_condition} does not hold,
then the constant blows up in one of the cases $j \to \infty$ or $j \to -\infty$,
and in particular, the Fourier restriction estimate \eqref{section_intro_restriction_estimate} for $\mu$ cannot hold
(here we use that the restriction operators for $\mu$ and $\mu_j$'s are nonzero since $\mathcal{W}$ is not identically zero).

\medskip
Let us now assume that we indeed have \eqref{section_prelim_hom_condition}.
It is obvious that the Fourier restriction estimate for $\mu$ implies
the Fourier restriction estimate for $\mu_j$ for any $j$.
Let us therefore assume that the estimate \eqref{section_prelim_local_estimate_j} holds true for any $j \in \Z$,
and thus for all $j \in \Z$.

Before proceeding further let us denote by $(\tilde{\eta}_j)_{j \in \Z}$
a family of $C_c^\infty(\R^2 \setminus \{0\})$ functions such that
\begin{align*}
\tilde{\eta}_j = \tilde{\eta}_0 \circ \delta_{2^{-j}} \qquad \text{for all} \, j \in \Z,
\end{align*}
and such that $\tilde{\eta}_j$ is equal to $1$ on the support of $\eta_j$.
One can for example take $\tilde{\eta}_j = \sum_{|k-j| \leq N} \eta_{k}$
for some sufficiently large $N$.
Let us furthermore denote by $S_j$ the Fourier multiplier operator in $\R^3$
with multiplier $(\tilde{\eta}_j \otimes 1) (\xi_1,\xi_2,\xi_3) = \tilde{\eta}_j(\xi_1,\xi_2)$.

Now \eqref{section_prelim_local_estimate_j} implies
\begin{align*}
\Vert \widehat{S_j f} \Vert_{L^2(\mathrm{d}\mu_j)} = \Vert \widehat{f} \Vert_{L^2(\mathrm{d}\mu_j)}
   \leq C \Vert S_j f\Vert_{L^{p}}.
\end{align*}
Therefore
\begin{align*}
\Vert \widehat{f} \Vert^2_{L^2(\mathrm{d}\mu)}
   &= \langle \mu, |\widehat{f}|^2 \rangle
    = \sum_{j\in\Z} \Big\langle \mu_j, |\widehat{f}|^2 \Big\rangle
    = \sum_{j\in\Z} \Big\langle \mu_j, |\widehat{S_j f}|^2 \Big\rangle \\
   &\leq C^2 \sum_{j\in\Z} \Vert S_j f\Vert^2_{L^{p}}
    = C^2 \Big\Vert \Vert S_j f\Vert_{L^{p_3}_{x_3}(L^{p_1}_{(x_1,x_2)})} \Big\Vert^2_{l^2_j},
\end{align*}
where $l^2_j$ denotes the norm of the Hilbert space of $l^2$ sequences on $\Z$.
Since both $p_1 \leq 2$ and $p_3 \leq 2$ we may use Minkowski's inequality
to interchange the $l^2_j$ norm with the $L^p$ norm, and
subsequently apply Littlewood-Paley theory in the $(x_1,x_2)$ variable
(in particular, we do not need to use mixed norm Littlewood-Paley theory)
to get
\begin{align*}
\Big\Vert \Vert S_j f\Vert_{L^p} \Big\Vert^2_{l^2_j}
 \leq \Big\Vert \Vert S_j f\Vert_{l^2_j} \Big\Vert^2_{L^p}
 = \Big\Vert \Big(\sum_{j \in \Z} |S_j f|^2 \Big)^{1/2} \Big\Vert^2_{L^p} \sim \Vert f \Vert_{L^p}.
\end{align*}
This finishes the proof of Proposition \ref{prop_prelim_scaling}.
\end{proof}

\begin{remark}[Scaling in the case of Hessian determinant]
\label{remark_Hessian}
Using the homogeneity condition of $\phi$ one easily obtains that
the Hessian determinant is also $\alpha$-mixed homogeneous of degree $2\rho - 2|\alpha|$.
Thus, when we take $\mathcal{W} = |\mathcal{H}_\phi|^{\sigma}$, $\mathcal{W}$ is homogeneous of degree $2\vartheta = 2\sigma(\rho-|\alpha|)$.
Recall that in this case (i.e., as in the assumptions of Theorem \ref{theorem_Strichartz_mitigating})
we assume that $1/p_1' = 1/2 - \sigma$, $1/p_3' = \sigma$, and so by \eqref{section_prelim_hom_condition}
the equality $2\vartheta = 2\sigma(\rho-|\alpha|)$ is indeed satisfied,
i.e., the right relation between the exponents if one wants scaling invariance.
\end{remark}

\begin{remark}[A general sufficient condition for local integrability of $\mathcal{W}$]
\label{remark_integrability}
Since $\mathcal{W}$ is mixed homogeneous of degree $2\vartheta$,
$\mathcal{W} |x|_\alpha^{-2\vartheta}$ is mixed homogeneous of degree $0$, and in particular a bounded function.
Thus $|\mathcal{W}| \lesssim |x|_\alpha^{2\vartheta}$, and so
it is sufficient to check when $|x|_\alpha^{2\vartheta}$ is locally integrable in $\R^2$.
By symmetry it is sufficient to integrate over $\{(x_1,x_2) : x_1, x_2 > 0\}$.
We have
\begin{align*}
\int_{x_1, x_2 > 0, |x| \lesssim 1} |x|_\alpha^{2\vartheta} \mathrm{d}x
  &= \int_{x_1, x_2 > 0, |x| \lesssim 1} (x_1^{1/\alpha_1} + x_2^{1/\alpha_2})^{2\vartheta} \mathrm{d}x \\
  &\sim \int_{y_1, y_2 > 0, |y| \lesssim 1} (y_1^2 + y_2^2)^{2\vartheta} y_1^{2\alpha_1-1} y_2^{2\alpha_2-1} \mathrm{d}y \\
  &\sim \int_{0 < r \lesssim 1} \int_0^{\pi/2} r^{4\vartheta+2|\alpha|-1} (\cos \theta)^{2\alpha_1-1} (\sin \theta)^{2\alpha_2-1} \mathrm{d}\theta \mathrm{d}r
\end{align*}
Therefore, we must have $4\vartheta+2|\alpha|-1 > -1$, i.e.,
\begin{align*}
2\vartheta + |\alpha| > 0.
\end{align*}
Note that this holds if $\rho \geq 0$, $p_1 > 1$, and $\vartheta$ is given by \eqref{section_prelim_hom_condition}.
\end{remark}

\begin{remark}
When $\phi$ is smooth at the origin and a nonconstant function, then $\rho = 1$, and
the necessary condition obtained by a Knapp-type example
associated to the principle face of $\mathcal{N}(\phi)$ in the initial coordinate system
(see \cite[Proposition 2.1]{Pa19}) tells us that
\begin{align*}
\frac{|\alpha|}{p_1'}+\frac{1}{p_3'} \leq \frac{|\alpha|}{2}
\end{align*}
is necessary for \eqref{section_intro_restriction_estimate}
if $\mathcal{W} \equiv 1$ (i.e., $\vartheta = 0$).
On the other hand,
if we denote $l_\alpha = \{ (t_1, t_3) \in \R^2 : |\alpha| t_1 + t_3 = |\alpha|/2 \}$,
then the expression \eqref{section_prelim_hom_condition} for $\vartheta$ implies that
\begin{align*}
|\vartheta| = \sqrt{1+|\alpha|^2} \, \dist \Big( (1/p_1',1/p_3'), l_\alpha \Big).
\end{align*}
\end{remark}

\subsection{Some further reductions}
\label{subsection_prelim_further}

According to Proposition \ref{prop_prelim_scaling},
under the conditions of Theorem \ref{theorem_Strichartz_mitigating} or Theorem \ref{theorem_Strichartz},
we have to prove the Fourier restriction estimate for a measure defined by the mapping
\begin{align*}
f \mapsto \int_{\R^2} f(x,\phi(x)) \, \eta(x) \, \mathcal{W}(x) \mathrm{d}x,
\end{align*}
where $\eta \in C_c^\infty(\R^2\setminus\{0\})$ is supported in a compact annulus centered at the origin.
Note that in the case of the weight $\mathcal{W} = |\mathcal{H}_\phi|^{\sigma}$ (the case of Theorem \ref{theorem_Strichartz_mitigating})
the degree of homogeneity $2 \vartheta = 2\sigma(\rho-|\alpha|)$
satisfies the relation \eqref{section_prelim_hom_condition} by Remark \ref{remark_Hessian}.

\medskip
{\bf Reductions for the amplitude $\eta$.}
One can easily show that in the context of the Fourier restriction problem we may make the following reductions.
First, by reordering coordinates and/or changing their sign, and by splitting the amplitude $\eta$ into functions with smaller support,
we may restrict ourselves to amplitudes $\eta$ with support contained in the half-plane $\{ (x_1, x_2) \in \R^2 : x_1 \gtrsim 1 \}$.
Then, by compactness, we may localize to small neighbourhoods of points $v \neq 0$ having $v_1 \gtrsim 1$.
Thus, one may assume that the support of $\eta$ is contained in a small neighbourhood
of some generic point $v$ satisfying $v_1 \sim 1$ and $|v| \lesssim 1$.
In fact, compactness and changing signs if necessary implies that we may further assume that either $v_2 = 0$ or $v_2 \sim 1$.

\medskip
{\bf Changing the affine terms of the phase.}
By the previous discussion it suffices to consider the measure
\begin{align}
\label{subsection_prelim_further_local_measure}
f \mapsto \int_{\R^2} f(x,\phi(x)) \, \eta_v(x) \, \mathcal{W}(x) \mathrm{d}x,
\end{align}
where $\eta_v$ is a smooth function supported in a small neighbourhood of a point $v \neq 0$.
We now recall the fact that we can freely add or remove
linear and constant terms in the expression for $\phi$
in the context of the Fourier restriction problem.
For the constant term this is obvious.
For the linear terms this can be achieved by using a linear transformation of the form
$(x_1,x_2,x_3) \mapsto (x_1,x_2,b_1 x_1 + b_2 x_2 + x_3)$ (for more details see \cite[Subsection 3.1]{Pa19}).
In particular, instead of considering the measure \eqref{subsection_prelim_further_local_measure},
we may consider the measure
\begin{align*}
f \mapsto \int_{\R^2} f(x,\phi_v(x-v)) \, \eta_v(x) \, \mathcal{W}(x) \mathrm{d}x,
\end{align*}
where we recall that
\begin{align*}
\phi_v(x) \coloneqq \phi(x+v) - \phi(v) - x \cdot \nabla \phi(v).
\end{align*}

\medskip

The strategy for the proof of Theorem \ref{theorem_Strichartz_mitigating} and Theorem \ref{theorem_Strichartz} should now be clear.
The above discussion reduces the problem to proving a local Fourier restriction estimate in the vincinity of a point $v$,
and so one needs to determine the local normal form of $\phi$ at $v$,
and in the case $\mathcal{W}(x) = |\mathcal{H}_\phi(x)|^{\sigma}$
one needs to additionally determine the order of vanishing of the Hessian determinant at $v$
in the $x_2$ direction (after which the normal form of $\mathcal{W}$ will be clear by homogeneity).



\section{Local normal forms}
\label{section_forms}

In this section we derive the local normal forms for $\phi$ and
for the Hessian determinant $\mathcal{H}_{\phi}$ at a fixed point $v \neq 0$
(as a consequence we prove Proposition \ref{proposition_normal_forms}).
The discussion in Subsection \ref{subsection_prelim_further} implies that
we may assume that $v_1 \sim 1$, and either $v_2 = 0$ or $v_2 \sim 1$.

The structure of this section is as follows.
In Subsection \ref{subsection_forms_considerations} we fix the notation for this section, introduce relevant quantities,
and define the coordinate systems $y$, $z$, and $w$
(the coordinate systems $z$ and $w$ will not be described precisely until Subsection \ref{subsection_forms_derivations_w} though).
In Subsections \ref{subsection_forms_flat}, \ref{subsection_forms_table_0}, and \ref{subsection_forms_table_1}
tables with normal forms of $\phi_v$ are given.
It turns out that in most cases $y$ coordinates suffice and when we use them one obtains the normal forms easily.
We deal with the case when $y$ coordinates do not suffice in Subsection \ref{subsection_forms_derivations_w}.
In Subsection \ref{subsection_forms_derivations_Hessian} we sketch how to calculate
what is the order of vanishing of the Hessian determinant for the respective normal forms.

We assume that the (H1) condition is satisfied throughout this section.
In fact, in Subsection \ref{subsection_forms_flat} we shall explicitly determine the local normal form of $\phi$
when $t \mapsto \partial_2^2 \phi(v_1, t)$ is flat at $v_2$.
In this case it turns out that the Hessian determinant either does not vanish at $v$, or that it is flat at $v$.
In all the other subsections we shall assume that $t \mapsto \partial_2^2 \phi(v_1, t)$ is of finite type at $v_2$.

\subsection{Notation and some general considerations}
\label{subsection_forms_considerations}

Let us begin by introducing the notation.
It will be useful to denote
\begin{align*}
\gamma \coloneqq \frac{\alpha_2}{\alpha_1} > 0,
\end{align*}
and for the point $v = (v_1,v_2)$ (recall $v_1 \sim 1$) we define
\begin{align*}
t_0 &\coloneqq v_2 v_1^{-\gamma}.
\end{align*}
Let us denote the $\partial_2$ derivatives of $\phi$ at $(1,t_0)$ by
\begin{align*}
b_{j} \coloneqq \partial_2^{j} \phi(1,t_0) = g^{(j)}(t_0), \qquad j \in \N_0,
\end{align*}
where
\begin{align*}
g(t) \coloneqq \phi(1,t).
\end{align*}
We furthermore denote
\begin{align}
\label{definition_k}
k \coloneqq \inf \{ j \geq 2 : b_j \neq 0 \},
\end{align}
where we take $k = \infty$ if $b_j = 0$ for all $j \geq 2$.
The equality $k = \infty$ is equivalent to $g^{(2)}$ being flat at $0$.
What precisely happens when $g^{(2)}$ is flat shall be explained in Subsection \ref{subsection_forms_flat},
and in the rest of the section (including this subsection) we assume that $k < \infty$,
unless explicitly stated otherwise.

\medskip

{\bf General form of mixed homogeneous $\phi$.}
Recall that we denote by $\rho \in \{-1,0,1\}$ the degree of homogeneity of $\phi$.
Then we have for any $x$ satisfying $x_1 > 0$:
\begin{align}
\label{eq_homogeneity_phi}
\phi(x_1,x_2) = x_1^{\rho/\alpha_1} \phi(1,x_2 \, x_1^{-\gamma}).
\end{align}
Let us consider the Taylor expansion of $t \mapsto \phi(1,t)$ at $t_0$:
\begin{align*}
g(t) = \phi(1,t) = b_0 + (t-t_0) b_1 + \frac{1}{k!} (t-t_0)^k g_k(t),
\end{align*}
where $g_k$ is a smooth function such that $b_k = g_k(0)$.
Thus, we get
\begin{align}
\label{eq_phi_basic_form}
\phi(x) &=
  x_1^{\rho/\alpha_1}
  \Bigg(b_0 + (x_2 \, x_1^{-\gamma}-t_0) b_1 + \frac{1}{k!} (x_2 \, x_1^{-\gamma}-t_0)^k g_k(x_2 \, x_1^{-\gamma}) \Bigg) \nonumber \\
  &=
  x_1^{\rho/\alpha_1} (b_0-t_0 b_1) + x_2 \, x_1^{(\rho-\alpha_2)/\alpha_1} b_1
  + \frac{1}{k!} x_1^{(\rho-k\alpha_2)/\alpha_1} (x_2 - t_0 x_1^\gamma)^k g_k(x_2 \, x_1^{-\gamma}).
\end{align}
More generally, we have the formal series expansion:
\begin{align}
\label{eq_phi_basic_form_series}
\begin{split}
\phi(x)
   &\approx \sum_{j = 0}^\infty \frac{b_j}{j!} (x_2 - t_0 x_1^\gamma)^j \, x_1^{\rho/\alpha_1 - j\gamma} \\
   &= b_0 x_1^{\rho/\alpha_1} + b_1 (x_2 - t_0 x_1^\gamma) x_1^{\rho/\alpha_1-\gamma}
    + \sum_{j = k}^\infty \frac{b_j}{j!} (x_2 - t_0 x_1^\gamma)^j \, x_1^{\rho/\alpha_1 - j \gamma}.
\end{split}
\end{align}
If $\gamma = 1$ (i.e., $\alpha_1 = \alpha_2$) it will be usually better to write
\begin{align}
\label{eq_phi_basic_form_gamma_1}
\phi(x) =
  x_1^{\rho/\alpha_1} b_0 + (x_2 - t_0 x_1) x_1^{\rho/\alpha_1-1} b_1
  + \frac{1}{k!} (x_2 - t_0 x_1)^k \, x_1^{\rho/\alpha_1-k} \, g_k(x_2 \, x_1^{-1}).
\end{align}
Since $v_1 \sim 1$, we may assume
\begin{align*}
|x_1^{1/\alpha_1} -v_1^{1/\alpha_1}| \ll 1, \qquad |x_2 \, x_1^{-\gamma} - v_2 \, v_1^{-\gamma}| \ll 1.
\end{align*}
The second condition is equivalent to $|x_2 - t_0 x_1^{\gamma}| \ll 1$.
Note that the points on the homogeneity curve through $v$ satisfy the equation $x_2 = t_0 x_1^{\gamma}$.

\medskip

In order to determine the normal forms it will suffice to introduce three additional coordinate systems,
which we shall denote by $y$, $z$, and $w$ respectively, each having the point $v$ as their origin.
The original coordinate system is denoted by $x$.
The function $\phi$ in the coordinate system $y$ (resp. $z$, $w$) shall be denoted by $\phi^y$ (resp. $\phi^z$, $\phi^w$).
For the original coordinate system $x$ we simply use $\phi$, or $\phi^x$ for emphasis.

The function $\phi$ in the coordinate system $y$ (resp. $z$, $w$) but without the affine terms at $v$
shall be denoted by $\phi_v^y$ (resp. $\phi_v^z$, $\phi_v^w$).
This means
\begin{align*}
\phi_v^y(y) \coloneqq \phi^y(y) - \phi^y(0) - y \cdot \nabla \phi^y(0),
\end{align*}
and similarly for $\phi_v^z$ and $\phi_v^w$.

\medskip

{\bf The coordinate system $y$.}
It is defined through the following affine coordinate change having $v = (v_1,v_2)$ as the origin:
\begin{align*}
\begin{split}
y_1 &= x_1 - v_1,\\
y_2 &= x_2 - v_2 - \gamma v_2 v_1^{-1} (x_1-v_1)\\
    &= x_2 - (1-\gamma)v_2 - \gamma v_2 v_1^{-1} x_1.
\end{split}
\end{align*}
The reverse transformation is
\begin{align}
\label{eq_basic_coordinate_change_reverse}
\begin{split}
x_1 &= y_1 + v_1,\\
x_2 &= y_2 + v_2 + \gamma v_2 v_1^{-1} y_1.
\end{split}
\end{align}
One can easily check that in these coordinates we can write
\begin{align*}
x_2 - t_0 x_1^{\gamma}
   &= y_2 + v_2 + \gamma v_2 v_1^{-1} y_1 - v_2 (1 + v_1^{-1} y_1)^{\gamma} \\
   &= y_2 + v_2 + \gamma v_2 v_1^{-1} y_1 - v_2 \Big( 1 + \gamma v_1^{-1} y_1 + \binom{\gamma}{2} v_1^{-2} y_1^2 + \mathcal{O}(y_1^3) \Big) \\
   &= y_2 - y_1^2 \omega(y_1),
\end{align*}
i.e., the points on the homogeneity curve through $v$
satisfy the equation $y_2 = y_1^2 \omega(y_1)$ in $y$ coordinates.
Above (and in the following) we use the notation $\binom{c}{m} = c (c-1) \cdot \ldots \cdot (c-m+1)/m!$
for $c \in \R$ and $m$ nonnegative integer. Furthermore, we obviously have:
\begin{remark}
\label{remark_root_identically_null}
It holds that $\omega(0) \neq 0$ if and only if
$\omega$ is not identically $0$ if and only if
$v_2 \neq 0$ (i.e., $t_0 \neq 0$) and $\gamma \neq 1$.
\end{remark}
The coordinate system $y$ will be used in most of the normal forms below
which shall follow directly from the expression
\begin{align}
\label{eq_phi_basic_form_y}
\phi^y(y) &=
  (v_1+y_1)^{\rho/\alpha_1} (b_0-t_0 b_1) + (v_2+y_2+\gamma v_2 v_1^{-1}y_1) \, (v_1+y_1)^{(\rho-\alpha_2)/\alpha_1} b_1 \\
  &\quad\,\, + (y_2-y_1^2 \omega(y_1))^k r(y) \nonumber
\end{align}
which one obtains from \eqref{eq_phi_basic_form} and \eqref{eq_basic_coordinate_change_reverse}.
When $\gamma = 1$ one uses \eqref{eq_phi_basic_form_gamma_1} instead and gets
\begin{align}
\label{eq_phi_basic_form_y_gamma_1}
\phi^y(y) =
  (v_1+y_1)^{\rho/\alpha_1} b_0 + y_2 (v_1+y_1)^{\rho/\alpha_1-1} b_1 + y_2^k r(y).
\end{align}
In both \eqref{eq_phi_basic_form_y} and \eqref{eq_phi_basic_form_y_gamma_1}
the function $r$ is smooth and nonvanishing at the origin.
Let us also note that the expansion \eqref{eq_phi_basic_form_series} can be rewritten in $y$ coordinates as
\begin{align}
\label{eq_phi_basic_form_series_y}
\begin{split}
\phi^y(y)
   &\approx b_0 (v_1+y_1)^{\rho/\alpha_1} + b_1 (y_2-y_1^2 \omega(y_1)) (v_1+y_1)^{\rho/\alpha_1-\gamma} \\
   &\quad + \sum_{j = k}^\infty \frac{b_j}{j!} (y_2-y_1^2 \omega(y_1))^j \, (v_1+y_1)^{\rho/\alpha_1 - j \gamma}.
\end{split}
\end{align}

The following simple lemma shall be useful later:
\begin{lemma}
\label{lemma_second_y_derivatives}
From equations \eqref{eq_phi_basic_form_y} and \eqref{eq_phi_basic_form_y_gamma_1} we get the following information
on the second order derivatives of $\phi^y$:
\begin{itemize}
\item[(1)]
It always holds:
\begin{align*}
k = 2 \,\, \Longleftrightarrow \,\, b_2 \neq 0 \,\, \Longleftrightarrow \,\, \partial_2^2 \phi^y (0) \neq 0.
\end{align*}
\item[(2.a)]
If $\rho \neq 1$ or $\alpha_2 \neq 1$ (i.e., $\rho - \alpha_2 \neq 0$), then
\begin{align*}
b_1 \neq 0 \,\, \Longleftrightarrow \,\, \partial_1 \partial_2 \phi^y (0) \neq 0.
\end{align*}
\item[(2.b)]
If $\rho = \alpha_2 = 1$ or if $b_1 = 0$, then $\partial_1 \partial_2 \phi^y (0) = 0$.
\item[(3.a)]
If $\rho = 0$ and $\alpha_1 \neq \alpha_2$ (i.e., $\gamma \neq 1$),
or if $\rho = \alpha_1 = 1$ and $\alpha_2 \neq 1$ (and in particular $\gamma \neq 1$), then
\begin{align*}
b_1 \neq 0, t_0 \neq 0 \,\, \Longleftrightarrow \,\, \partial_1^2 \phi^y (0) \neq 0,
\end{align*}
and we remind that $v_2 \neq 0$ if and only if $t_0 \neq 0$.
\item[(3.b)]
If $\rho = \alpha_2 = 1$ and $\alpha_1 \neq 1$ (and in particular $\gamma \neq 1$), then
\begin{align*}
b_0 - t_0 b_1 \neq 0 \Longleftrightarrow \,\, \partial_1^2 \phi^y (0) \neq 0.
\end{align*}
\item[(3.c)]
If $\gamma = 1$ (i.e., $\alpha_1 = \alpha_2$) or if $b_1 = 0$, then
\begin{align*}
b_0 \neq 0, \frac{\rho}{\alpha_1} \notin \{0,1\} \Longleftrightarrow \,\, \partial_1^2 \phi^y (0) \neq 0.
\end{align*}
Note that $\rho/\alpha_1 = 0$ if and only if $\rho = 0$, and $\rho/\alpha_1 = 1$ if and only if $\rho = \alpha_1 = 1$.
\end{itemize}
\end{lemma}

\begin{proof}
The only not completely trivial case is (3.a).
Since in this case $\rho/\alpha_1 \in \{0,1\}$,
the first term in \eqref{eq_phi_basic_form_y} is an affine term, and so we can ignore it.
Since $k \geq 2$,
the third term also does not contribute to the $y_1^2$ term in the Taylor series of $\phi^y$,
and and so we can ignore it too.
We therefore only need to consider the term:
\begin{align*}
(v_2+y_2+\gamma v_2 v_1^{-1}y_1) \, (v_1+y_1)^{(\rho-\alpha_2)/\alpha_1} b_1,
\end{align*}
and in fact, we may even reduce ourselves to
\begin{align*}
(v_2+\gamma v_2 v_1^{-1}y_1) \, (v_1+y_1)^{(\rho-\alpha_2)/\alpha_1} b_1
= b_1 v_2 (1+\gamma v_1^{-1}y_1) \, (v_1+y_1)^{(\rho-\alpha_2)/\alpha_1}.
\end{align*}

Now if $t_0 = 0$ (i.e. $v_2 = 0$) or if $b_1 = 0$, then $\partial_1^2 \phi^y (0) = 0$ follows.
Let us now assume $v_2 \neq 0$ and $b_1 \neq 0$.
We note that in our case we may rewrite $(\rho-\alpha_2)/\alpha_1 = \rho-\gamma$,
and so it suffices to show that
\begin{align*}
\partial_{y_1}^2 \mid_{y_1 = 0} \Big((1+\gamma v_1^{-1}y_1) \, (1+v_1^{-1} y_1)^{\rho-\gamma}\Big) \neq 0.
\end{align*}
Calculating the second derivative one gets
\begin{align*}
2 \gamma v_1^{-2} (\rho-\gamma) + v_1^{-2} (\rho-\gamma)(\rho-\gamma-1).
\end{align*}
This is not zero since in this case we have $\rho \in \{0,1\}$ and $\gamma \notin \{0,1\}$.
\end{proof}

\medskip

{\bf The coordinate systems $z$ and $w$.}
These are defined through affine coordinate changes of the form
\begin{align}
\label{eq_basic_coordinate_change_wz}
\begin{split}
x_1 &= v_1 + z_1,         \qquad \qquad \qquad \qquad \quad  w_1 = z_1 + \frac{1}{B}z_2, \\
x_2 &= v_2 + z_2 + A z_1, \qquad \qquad \qquad w_2 = z_2,
\end{split}
\end{align}
having $(v_1,v_2)$ as their origin, where we shall have $B \coloneqq A - \gamma v_2 v_1^{-1} \neq 0$
so that the coordinate system $y$ never coincides with the coordinate system $z$,
and the coordinate system $z$ never coincides with the coordinate system $w$.
The constant $A$ shall depend on $v$ and the first few derivatives of $\phi$ at $v$
(note that $A = B \neq 0$ if $v_2 = t_0 = 0$).
These coordinate systems will be described more precisely in Subsection \ref{subsection_forms_derivations_w}.
There we shall also introduce a smooth function $\tilde{\omega}$ such that
\begin{align*}
x_2 - t_0 x_1^{\gamma} = y_2 - y_1^2 \omega(y_1) = (w_1 - w_2^2 \tilde{\omega}(w_2)) r_0(w)
\end{align*}
for some smooth function $r_0$ satisfying $r_0(0) \neq 0$.
Note that we have
\begin{align}
\label{eq_basic_coordinate_change_yzw}
\begin{split}
y_1 &= z_1 
= w_1 - \frac{1}{B} w_2, \\
y_2 &= z_2 + B z_1 = B w_1.
\end{split}
\end{align}

\medskip

{\bf Some general considerations regarding the Hessian determinant $\mathcal{H}_\phi$.}
Recall that
\begin{align*}
\phi(r^{\alpha_1} x_1, r^{\alpha_2} x_2) = r^\rho \phi(x_1,x_2).
\end{align*}
Taking derivatives in $x_1$ and $x_2$ we get
\begin{align*}
(\partial_1^{\tau_1} \partial_2^{\tau_2} \phi)(r^{\alpha_1} x_1, r^{\alpha_2} x_2) =
   r^{\rho - \tau_1 \alpha_1 - \tau_2 \alpha_2} (\partial_1^{\tau_1} \partial_2^{\tau_2} \phi)(x_1,x_2).
\end{align*}
Thus, we have for the Hessian determinant of $\phi$:
\begin{align*}
\mathcal{H}_\phi (r^{\alpha_1} x_1, r^{\alpha_2} x_2)
   = r^{2(\rho-|\alpha|)} \mathcal{H}_\phi (x_1, x_2).
\end{align*}
From this it follows that if $\mathcal{H}_\phi$ vanishes at the point $v$, then
it also vanishes along the homogeneity curve through $v$ which we recall
is parametrized by $r \mapsto (r^{\alpha_1} v_1, r^{\alpha_2} v_2)$.

We are interested in the order of vanishing of $\mathcal{H}_\phi$ in directions transversal to this curve.
In particular, if we have $\partial_2^{\tau_2} \mathcal{H}_\phi(v) = 0$ for $\tau_2 < N$ and $\partial_2^N \mathcal{H}_\phi(v) \neq 0$,
then by using homogeneity and a Taylor expansion (as we did for $\phi$) we get
\begin{align*}
\mathcal{H}_\phi (x) = (x_2 - t_0 x_1^{\gamma})^N q(x),
\end{align*}
for some smooth function $q$ satisfying $q(v) \neq 0$.
Calculating $N$ shall be done in Subsection \ref{subsection_forms_derivations_Hessian}
by using the normal forms of $\phi$.
Recall that the Hessian determinant is equivariant under affine coordinate changes,
and so we can freely change to $y$, $z$, or $w$ coordinates.

\medskip

{\bf Preliminary comments on the normal forms.}
Let us introduce the following notation for the nondegenerate case
(i.e., the case when the Hessian determinant of $\phi$ does not vanish at $v$):
\begin{itemize}
\item[{\bf (ND)}]
The function $\phi_v$ is nondegenerate at the origin.
\end{itemize}
When $\phi_v$ does not satisfy (ND), then we shall show that
we can associate to it one of the following normal forms:
\begin{itemize}
\item[(i.y1)]
$\phi_v^y(y) = y_2^k r(y)$, $k \geq 2$, \\
$\mathcal{H}_{\phi^y}(y) = y_2^{\tilde{k}+2k-2} q(y)$, $0 \leq \tilde{k} \leq \infty$,
\item[(i.y2)]
$\phi_v^y(y) = y_1^{\tilde{k}} r(y_1) + \varphi(y)$, $\tilde{k} \geq 2$, \\
$\varphi$ and $\mathcal{H}_{\phi^y}$ are flat,
\item[(i.w1)]
$\phi_v^w(w) = w_2^2 r(w_2) + \varphi(w)$, \\
$\varphi$ and $\mathcal{H}_{\phi^w}$ are flat,
\item[(i.w2)]
$\phi_v^w(w) = w_2^2 r(w) + \varphi(w)$, \\
$v_1 B \partial_1^{j} r(0) = j A (\gamma-1) \partial_1^{j-1} r(0)$ for all $j \geq 1$,
where $A,B,v_1 \neq 0$ are defined as above, \\
$\varphi$ and $\mathcal{H}_{\phi^w}$ are flat,
\item[(ii.y)]
$\phi_v^y(y) = y_1^2 r_1(y_1) + y_2^k r_2(y)$, $k \geq 3$, \\
$\mathcal{H}_{\phi^y}(y) = y_2^{k-2} q(y)$,
\item[(ii.w)]
$\phi_v^w(w) = w_1^{\tilde{k}} r_1(w) + w_2^2 r_2(w_2)$, $\tilde{k} \geq 3$, \\
$\mathcal{H}_{\phi^w}(w) = w_1^{\tilde{k}-2} q(w)$,
\item[(iii)]
$\phi_v^w(w) = w_1^{\tilde{k}} r_1(w) + w_2^2 r_2(w)$, $\tilde{k} \geq 3$, \\
$v_1 B \partial_1^{j} r_2(0) = j A (\gamma-1) \partial_1^{j-1} r_2(0)$ for $1 \leq j \leq \tilde{k}-1$,
where $A,B,v_1 \neq 0$ are defined as above, \\
$\mathcal{H}_{\phi^w}(w) = w_1^{\tilde{k}-2} q(w)$,
\item[(iv)]
$\phi_v^y(y) = y_1^2 r_1(y_1) + (y_2-y_1^2 \omega(y_1))^k r_2(y)$, $k \geq 3$, \\
$\mathcal{H}_{\phi^y}(y) = (y_2-y_1^2 \omega(y_1))^{k-2} q(y)$,
\item[(v)]
$\phi_v^w(w) = (w_1 - w_2^2 \tilde{\omega}(w_2))^{\tilde{k}} r_1(w) + w_2^2 r_2(w)$, $\tilde{k} \geq 3$, \\
$v_1 B \partial_1^{j} r_2(0) = j A (\gamma-1) \partial_1^{j-1} r_2(0)$ for $1 \leq j \leq \tilde{k}-1$,
where $A,B,v_1 \neq 0$ are defined as above, \\
$\mathcal{H}_{\phi^w}(w) = (w_1 - w_2^2 \tilde{\omega}(w_2))^{\tilde{k}-2} q(w)$,
\item[(vi)]
$\phi_v^y(y) = (y_2-y_1^2 \omega(y_1))^k r(y)$, $k \geq 2$, \\
$\mathcal{H}_{\phi^y}(y) = (y_2-y_1^2 \omega(y_1))^{2k-3} q(y)$.
\end{itemize}
All the appearing functions are smooth and do not vanish at the origin
(except $\varphi$ which is always flat).
The number $k$ is as defined in \eqref{definition_k} and
it is always finite in the above normal forms (when it is infinite it turns out that one is necessarily in case of Normal form (i.y2)).
On the other hand, the definition of the number $\tilde{k}$ changes from case to case,
and we allow $\tilde{k}$ to be infinite only in Normal form (i.y1), in which case 
we consider the Hessian determinant to be flat at the origin.
Let us furthermore remark that Normal forms (i.w1) and (i.w2) stem from Normal forms (ii.w), (iii), and (v),
in the sense that they correspond to $\tilde{k} = \infty$.

\medskip

The first step in deriving the above normal forms is to switch to $y$ coordinates.
In most cases this will suffice and the normal form will be obvious,
and so in the following subsections we shall leave out most of the details for these cases.
In particular, as a consequence of considerations in 
Subsections \ref{subsection_forms_table_0} and \ref{subsection_forms_table_1},
we shall obtain:
\begin{lemma}
\label{lemma_k_greater_than_3}
If $k \geq 3$ and if we are not in the (ND) case,
then the function $\phi^y_v$ is always in one of the following normal forms:
(i.y1), (i.y2), (ii.y), (iv), or (vi).
\end{lemma}

If $k = 2$, $b_1 \neq 0$, $\rho \neq \alpha_2$, and we are not in the (ND) case,
then we shall either need to
\begin{itemize}
\item[{\bf (FP)}]
Flip coordinates (i.e., exchange $x_1$ and $x_2$) and use the $y$ coordinates associated to the flipped coordinates,
\end{itemize}
or we shall need $w$ (and the intermediary $z$) coordinates.
Details are to be found in Subsection \ref{subsection_forms_derivations_w} below.

Note that flipping coordinates makes sense only when $v_2 \neq 0$
(and indeed, we shall flip coordinates only when $A = 0$, which, as it turns out, never happens when $v_2 = 0$).
After flipping coordinates it will always suffice to use the $y$ coordinates
(associated to the flipped $x$, $v$, and $\alpha$),
and in particular, we shall be able to apply Lemma \ref{lemma_k_greater_than_3}.
Note that these $y$ coordinates are not in general equal to flipped $y$ coordinates associated to the original $x$, $v$, and $\alpha$.


\subsection{Normal form when $t \mapsto \partial_2^2 \phi(1,t)$ is flat at $t_0$ (i.e., $k = \infty$)}
\label{subsection_forms_flat}

Let us assume that
\begin{align}
\label{subsection_forms_flat_assumption}
\partial_2^{j} \phi (1,t_0) = 0 \quad \text{for all} \,\, j \geq 2,
\end{align}
and so we have $\partial_2^{j} \phi (v) = 0$ for all $v$ (with $v_1 > 0$) satisfying $v_2 v_1^{-\gamma} = t_0$ by \eqref{eq_homogeneity_phi}.
The Euler equation for $\phi$ is
\begin{align*}
\rho \phi(x) = \alpha_1 x_1 \partial_1 \phi(x) + \alpha_2 x_2 \partial_2 \phi(x).
\end{align*}
Taking the derivative $\partial^\tau = \partial_1^{\tau_1} \partial_2^{\tau_2}$ we get
at $(v_1,v_2)$ that
\begin{align*}
(\rho - \alpha_1 \tau_1 - \alpha_2 \tau_2) \partial^\tau\phi(v)
    = \alpha_1 v_1 \partial^{\tau+(1,0)} \phi(v) + \alpha_2 v_2 \partial^{\tau+(0,1)} \phi(v).
\end{align*}
From this, the fact that $\alpha_1 v_1 \neq 0$, and the flatness assumption \eqref{subsection_forms_flat_assumption}
it follows by induction in $\tau_1$ that $\partial^\tau \phi(v) = 0$
for all $\tau_1 \geq 0$ and $\tau_2 \geq 2$.

If now $\partial_1 \partial_2 \phi(v) \neq 0$,
then the Hessian determinant does not vanish and we are in the (ND) case
(this always happens for example when $\phi(x_1,x_2) = x_1 x_2$).
On the other hand, if $\partial_1 \partial_2 \phi(v) = 0$,
then we get in the same way as above that $\partial^\tau \phi(v) = 0$
for all $\tau_1 \geq 1$ and $\tau_2 = 1$.
Thus, by using a Taylor expansion at $v$ and by switching to $y$ coordinates (recall $x_1 = y_1 + v_1$)
we may write
\begin{align*}
\phi_v^y(y) = y_1^2 r(y_1) + \varphi(y),
\end{align*}
where $r$ is a smooth function and $\varphi$ is a flat smooth function.
In particular, in this case the Hessian determinant vanishes of infinite order at $x = v$ and therefore the condition (H2) cannot hold.
This also shows that (H2) is a stronger condition than (H1).
Since we assume that at least (H1) holds, then we necessarily have that $t \mapsto \partial_1^2 \phi(t,v_2)$ is not flat at $v_1$,
and so $r$ cannot be flat either, i.e., we can write
\begin{align*}
\phi_v^y(y) = y_1^{\tilde{k}} \tilde{r}(y_1) + \varphi(y),
\end{align*}
for some smooth function $\tilde{r}$ satisfying $\tilde{r}(0) \neq 0$ and $\tilde{k} \geq 2$.
This is precisely the Normal form (i.y2).


\subsection{Normal form tables for $\phi$ mixed homogeneous of degree $\rho = 0$}
\label{subsection_forms_table_0}

Recall that we assume $k < \infty$ in this and the following subsections.
In this case \eqref{eq_phi_basic_form_y} becomes
\begin{align*}
\phi^y(y) - (b_0-t_0 b_1) =
  (v_2+y_2+\gamma v_2 v_1^{-1}y_1) \, (v_1+y_1)^{-\gamma} b_1 + (y_2-y_1^2 \omega(y_1))^k r(y)
\end{align*}
if $\gamma \neq 1$,
and in the case $\gamma = 1$ we have by \eqref{eq_phi_basic_form_y_gamma_1} that
\begin{align}
\label{subsection_forms_table_0_gamma_1_phi}
\phi^y(y) - b_0 =
   y_2 (v_1+y_1)^{-1} b_1 + y_2^k r(y).
\end{align}
We have put the constant terms on the left hand side
since we may freely ignore them.
Note that in the case $\gamma = 1$ we have $\partial_1^2 \phi^y (0) = 0$.

\medskip
\noindent
{\bf Case: $\gamma = 1$.}

\begin{center}
\begin{tabular}{|c|c|c|c|}
\hline
Conditions & Case \\
\hline
$b_1 = 0$ & Normal form (i.y1) \\
\hline
$b_1 \neq 0$ & (ND) \\
\hline
\end{tabular}
\end{center}

Here we actually have in the case when $b_1 = 0$ a precise order of vanishing of the Hessian determinant:
it is always $2k-2$.
This follows from Subsection \ref{subsection_forms_derivations_Hessian} (see in particular \eqref{subsection_forms_derivations_conditions_1}).

If $b_1 \neq 0$, then from \eqref{subsection_forms_table_0_gamma_1_phi}
we obviously have $\partial_1 \partial_2 \phi^y(0) \neq 0$,
and it follows that the Hessian determinant at $0$ is nonzero.

\medskip
\noindent
{\bf Case: $\gamma \neq 1$.}

\begin{center}
\begin{tabular}{|c|c|c|c|}
\hline
Conditions & Case \\
\hline
$t_0 = 0, b_1 = 0$ & Normal form (i.y1) \\
\hline
$t_0 = 0, b_1 \neq 0$ & (ND) \\
\hline
$t_0 \neq 0, b_1 = 0$ & Normal form (vi) \\
\hline
$t_0 \neq 0, b_1 \neq 0, k \geq 3$ & (ND) \\
\hline
$t_0 \neq 0, b_1 \neq 0, k = 2$ & (ND), or (FP), or Normal form (v), or Normal form (i.w2) \\
\hline
\end{tabular}
\end{center}

In the case $t_0 = 0, b_1 \neq 0$ we apply Lemma \ref{lemma_second_y_derivatives}, (2.a) and (3.a),
and get respectively that $\partial_1 \partial_2 \phi^y(0) \neq 0$ and $\partial_1^2 \phi^y(0) = 0$,
from which it indeed follows that we are in the (ND) case.
Similarly, in the case $t_0 \neq 0, b_1 \neq 0, k \geq 3$ we use
Lemma \ref{lemma_second_y_derivatives}, (1) and (2.a),
and obtain that $\partial_2^2 \phi^y(0) = 0$ and $\partial_1 \partial_2 \phi^y(0) \neq 0$,
from which we again get that the Hessian determinant of $\phi^y$ does not vanish.

As the case $t_0 \neq 0, b_1 \neq 0, k = 2$ shall be treated in the same way as certain
other cases which appear later and where $w$ coordinates may be needed,
we have postponed its discussion to Subsection \ref{subsection_forms_derivations_w}.


\subsection{Normal form tables for $\phi$ mixed homogeneous of degree $\rho = \pm 1$}
\label{subsection_forms_table_1}

Recall that here we have
\begin{align}
\label{subsection_forms_table_1_phi}
\phi(x) =
  x_1^{\rho/\alpha_1} (b_0-t_0 b_1) + x_2 \, x_1^{(\rho-\alpha_2)/\alpha_1} b_1
  + \frac{1}{k!} x_1^{(\rho-k\alpha_2)/\alpha_1} (x_2 - t_0 x_1^\gamma)^k g_k(x_2 \, x_1^{-\gamma})
\end{align}
and that in $y$ coordinates this becomes
\begin{align}
\label{subsection_forms_table_1_phi_y}
\begin{split}
\phi^y(y) &=
  (v_1+y_1)^{\rho/\alpha_1} (b_0-t_0 b_1) + (v_2+y_2+\gamma v_2 v_1^{-1}y_1) \, (v_1+y_1)^{(\rho-\alpha_2)/\alpha_1} b_1 \\
  &\quad\,\, + (y_2-y_1^2 \omega(y_1))^k r(y).
\end{split}
\end{align}
In this subsection (where $\rho = \pm 1$) we need to consider five possible subcases.
The cases we first consider are when $\rho = \alpha_1$, or $\rho = \alpha_2$, or both.
Since $\alpha_1$ and $\alpha_2$ are strictly positive, these cases are only possible for $\rho = 1$.
The penultimate case is when $\alpha_1 = \alpha_2 \neq \rho$,
and the last case is when all of $\alpha_1$, $\alpha_2$, and $\rho$ are different from each other.

\medskip
\noindent
{\bf Case: $\rho = 1$, $\alpha_1 = 1$, $\alpha_2 = 1$.}

In this case the first two terms in \eqref{subsection_forms_table_1_phi_y} become affine,
and by Remark \ref{remark_root_identically_null} we have $\omega \equiv 0$.
As a consequence we have only one case:

\begin{center}
\begin{tabular}{|c|c|c|c|}
\hline
Conditions & Case \\
\hline
- & Normal form (i.y1) \\
\hline
\end{tabular}
\end{center}

Furthermore, we note that initially we know that the order of vanishing of the Hessian determinant is at least $2k-2$,
which is always greater than or equal to $2$.
Since this is true at every point, the Hessian determinant vanishes identically in this case.

\medskip
\noindent
{\bf Case: $\rho = 1$, $\alpha_1 \neq 1$, $\alpha_2 = 1$.}

Here we first note that by Lemma \ref{lemma_second_y_derivatives}, (2.b),
we always have $\partial_1 \partial_2 \phi^y(0) = 0$.
This is a simple consequence of the fact that
in this case the second term in \eqref{subsection_forms_table_1_phi_y} is linear.

\begin{center}
\begin{tabular}{|c|c|c|c|}
\hline
Conditions & Case \\
\hline
$b_0 - t_0 b_1 = 0, t_0 = 0$ & Normal form (i.y1) \\
\hline
$b_0 - t_0 b_1= 0, t_0 \neq 0$ & Normal form (vi) \\
\hline
$b_0 - t_0 b_1 \neq 0, k = 2$ & (ND) \\
\hline
$b_0 - t_0 b_1 \neq 0, k \geq 3, t_0 = 0$ & Normal form (ii.y) \\
\hline
$b_0 - t_0 b_1 \neq 0, k \geq 3, t_0 \neq 0$ & Normal form (iv) \\
\hline
\end{tabular}
\end{center}

The (ND) case follows from Lemma \ref{lemma_second_y_derivatives}, (1) and (3.b).

\medskip
\noindent
{\bf Case: $\rho = 1$, $\alpha_1 = 1$, $\alpha_2 \neq 1$.}

Here we note that the first term in \eqref{subsection_forms_table_1_phi_y} becomes linear,
and therefore does not influence the normal form of $\phi^y_v$.

\begin{center}
\begin{tabular}{|c|c|c|c|}
\hline
Conditions & Case \\
\hline
$t_0 = 0, b_1 = 0$ & Normal form (i.y1) \\
\hline
$t_0 = 0, b_1 \neq 0$ & (ND) \\
\hline
$t_0 \neq 0, b_1 = 0$ & Normal form (vi) \\
\hline
$t_0 \neq 0, b_1 \neq 0, k \geq 3$ & (ND) \\
\hline
$t_0 \neq 0, b_1 \neq 0, k = 2$ & (ND) or (FP) \\
\hline
\end{tabular}
\end{center}

The cases $t_0 = 0, b_1 \neq 0$ and $t_0 \neq 0, b_1 \neq 0, k \geq 3$ are (ND)
by the same argumentation as in the table above for $\rho = 0$, $\gamma \neq 1$
(namely, by applying Lemma \ref{lemma_second_y_derivatives}, (2.a) and (3.a),
in the case $t_0 = 0, b_1 \neq 0$, and by applying Lemma \ref{lemma_second_y_derivatives}, (1) and (2.a),
in the case $t_0 \neq 0, b_1 \neq 0, k \geq 3$).

Let us note the following for the last case where $t_0 \neq 0$, $b_1 \neq 0$, and $k = 2$.
The expression in \eqref{subsection_forms_table_1_phi} can be rewritten as
(after ignoring the first term, which is linear in this case):
\begin{align*}
b_1 x_2 \, &x_1^{1-\gamma}
  + \frac{b_2}{2} x_1^{1-2\gamma} (x_2 - t_0 x_1^\gamma)^2
  + \mathcal{O}\Big((x_2 - t_0 x_1^\gamma)^3 \Big)
\end{align*}
We want to calculate what the Hessian determinant of $\phi^x_v = \phi_v$ at $v$ is
(or equivalently, the Hessian determinant of $\phi$ at $v$).
For this we only need the second derivatives of $\phi$ at $v$,
and so we can freely ignore the last term of size $(x_2 - t_0 x_1^\gamma)^3$.
After expanding the second term in the above expression and
ignoring the linear terms and the term $\mathcal{O}((x_2 - t_0 x_1^\gamma)^3 )$ we get
\begin{align*}
(b_1 - t_0 b_2) x_1^{1-\gamma} x_2 + \frac{b_2}{2} x_1^{1-2\gamma} x_2^2.
\end{align*}
From this it follows by a direct calculation that
\begin{align*}
\partial_1^2 \phi(v) = -\gamma \frac{v_2}{v_1} \partial_1 \partial_2 \phi(v),
\end{align*}
and so
\begin{align*}
\mathcal{H}_\phi(v)
  = - \partial_1 \partial_2 \phi(v) \Big( \partial_1 \partial_2 \phi(v) + \gamma \frac{v_2}{v_1} \partial_2^2 \phi(v)\Big),
\end{align*}
which we note can be rewritten as
\begin{align*}
\mathcal{H}_\phi(v)
  = - \partial_1 \partial_2 \phi^x (v) \, \partial_1 \partial_2 \phi^y(0),
\end{align*}
by \eqref{eq_basic_coordinate_change_reverse}.
This implies in particular that $\mathcal{H}_\phi(v) = 0$
if and only if $\partial_1 \partial_2 \phi(v) = 0$
if and only if $\partial_1^2 \phi(v) = 0$
since by Lemma \ref{lemma_second_y_derivatives}, (2.a), we know that $\partial_1 \partial_2 \phi^y(0) \neq 0$.

Thus, in the last case where $t_0 \neq 0$, $b_1 \neq 0$, and $k = 2$,
we are either in the (ND) case, and otherwise we have $\partial_1^2 \phi(v) = 0$.
This means precisely that the ``$k$'' associated to the flipped coordinates
(and we can flip coordinates since $t_0 \neq 0$, i.e., $v_2 \neq 0$) is necessarily $\geq 3$.
For the flipped coordinates we may now use the previous table
where we have $\rho = 1$, $\alpha_1 \neq 1$, $\alpha_2 = 1$
(or apply Lemma \ref{lemma_k_greater_than_3}).

\medskip
\noindent
{\bf Case: $\rho = \pm 1$, $\alpha_1 = \alpha_2 \neq \rho$.}

Here one uses \eqref{eq_phi_basic_form_y_gamma_1}:
\begin{align*}
\phi^y(y) =
  (v_1+y_1)^{\rho/\alpha_1} b_0 + y_2 (v_1+y_1)^{\rho/\alpha_1-1} b_1 + y_2^k r(y).
\end{align*}

\begin{center}
\begin{tabular}{|c|c|c|c|}
\hline
Conditions & Case \\
\hline
$b_0 = 0, b_1 = 0$ & Normal form (i.y1) \\
\hline
$b_0 = 0, b_1 \neq 0$ & (ND) \\
\hline
$b_0 \neq 0, b_1 = 0, k \geq 3$ & Normal form (ii.y) \\
\hline
$b_0 \neq 0, b_1 = 0, k = 2$ & (ND) \\
\hline
$b_0 \neq 0, b_1 \neq 0, k \geq 3$ & (ND) \\
\hline
$b_0 \neq 0, b_1 \neq 0, k = 2$ & (ND), or (FP), or Normal form (ii.w), or Normal form (i.w1) \\
\hline
\end{tabular}
\end{center}

The first (ND) case $b_0 = 0, b_1 \neq 0$ follows from Lemma \ref{lemma_second_y_derivatives}, (2.a) and (3.c),
the second (ND) case $b_0 \neq 0, b_1 = 0, k = 2$ follows from Lemma \ref{lemma_second_y_derivatives}, (2.a), (3.c), and (1),
and the third (ND) case $b_0 \neq 0, b_1 \neq 0, k \geq 3$ follows from Lemma \ref{lemma_second_y_derivatives}, (1) and (2.a).
For the last case $b_0 \neq 0, b_1 \neq 0, k = 2$ we again refer the reader to Subsection \ref{subsection_forms_derivations_w}.

We give two further remarks.
Firstly, one can show that in the case $b_0 = 0$, $b_1 = 0$
the order of vanishing of the Hessian determinant is precisely equal to $2k-2$
if and only if we additionally have
\begin{align*}
\frac{\rho}{\alpha_1} \notin \{ 1, k \},
\end{align*}
as is shown in Subsection \ref{subsection_forms_derivations_Hessian}.
Note that here we cannot have $\rho/\alpha_1 = 1$,
and when $\rho/\alpha_1 = k$ from Subsection \ref{subsection_forms_derivations_Hessian}
we see that the Hessian determinant vanishes of order $2k+\tilde{k}-2$
where $\tilde{k}$ is the smallest positive integer such that $b_{k+\tilde{k}} \neq 0$
(it is also possible $\tilde{k} = \infty$ with the obvious interpretation).

Secondly, here we can calculate explicitly from the derivatives $b_{\tau_2} = g^{(\tau_2)}(t_0)$
the number $\tilde{k}$ in the Normal form (ii.w)
(see \eqref{section_forms_special_case_coef} in Subsection \ref{subsection_forms_derivations_w}).
This is already known for homogeneous polynomials \cite{FGU04}.

\medskip
\noindent
{\bf Case: $\rho = \pm 1$, $\alpha_1 \neq \rho$, $\alpha_2 \neq \rho$, $\alpha_1 \neq \alpha_2$.}

\begin{center}
\begin{tabular}{|c|c|c|c|}
\hline
Conditions & Case \\
\hline
$b_1 = 0, b_0 = 0, t_0 = 0$ & Normal form (i.y1) \\
\hline
$b_1 = 0, b_0 = 0, t_0 \neq 0$ & Normal form (vi) \\
\hline
$b_1 = 0, b_0 \neq 0, k = 2$ & (ND) \\
\hline
$b_1 = 0, b_0 \neq 0, k \geq 3, t_0 = 0$ & Normal form (ii.y) \\
\hline
$b_1 = 0, b_0 \neq 0, k \geq 3, t_0 \neq 0$ & Normal form (iv) \\
\hline
$b_1 \neq 0, k \geq 3$ & (ND) \\
\hline
$b_1 \neq 0, k = 2, t_0 = 0$ & (ND), or Normal form (iii), or Normal form (i.w2) \\
\hline
$b_1 \neq 0, k = 2, t_0 \neq 0$ &  (ND), or (FP), or Normal form (v), or Normal form (i.w2) \\
\hline
\end{tabular}
\end{center}

The first (ND) case $b_1 = 0, b_0 \neq 0, k = 2$ follows
from Lemma \ref{lemma_second_y_derivatives}, (1), (2.a), and (3.c),
and the second (ND) case $b_1 \neq 0, k \geq 3$
from Lemma \ref{lemma_second_y_derivatives}, (1) and (2.a).
For the very last two cases (namely, $b_1 \neq 0, k = 2, t_0 = 0$ and $b_1 \neq 0, k = 2, t_0 \neq 0$)
we refer the reader, as usual, to Subsection \ref{subsection_forms_derivations_w}.


\subsection{The case when $\rho \neq \alpha_2$, $b_1 \neq 0$, $k = 2$}
\label{subsection_forms_derivations_w}

In this subsection we shall discuss the remaining cases where $y$ coordinates did not suffice and
all of which (as one easily sees from the tables in the previous two subsection)
satisfy $\rho \neq \alpha_2$, $b_1 \neq 0$, $k = 2$.
Here it will turn out that we are either in the (ND) case, or (FP) case,
or that we need to use the $w$ coordinates.
In this case the form of the function $\phi$ in $y$ coordinates is according to \eqref{eq_phi_basic_form_y} equal to
\begin{align*}
\phi^y(y) &=
  (v_1+y_1)^{\rho/\alpha_1} (b_0-t_0 b_1) + (v_2+y_2+\gamma v_2 v_1^{-1}y_1) \, (v_1+y_1)^{(\rho-\alpha_2)/\alpha_1} b_1 \\
  &\quad\,\, + (y_2-y_1^2 \omega(y_1))^2 r(y),
\end{align*}
where $r(0) \neq 0$, and, as noted in Remark \ref{remark_root_identically_null},
$\omega \equiv 0$ if and only if $\gamma = 1$ or $t_0 = 0$, and otherwise $\omega(0) \neq 0$.
By Lemma \ref{lemma_second_y_derivatives}, (1) and (2.a), we have
\begin{align*}
\partial_2^2 \phi^y(0) \neq 0 \qquad \text{and} \qquad \partial_1 \partial_2 \phi^y(0) \neq 0,
\end{align*}
i.e., the $y_2^2$ term and the $y_1 y_2$ term in Taylor expansion of $\phi^y$ do not vanish.
Therefore, depending on what the coefficient of the $y_1^2$ term is,
it can happen that the Hessian determinant vanishes or not.

\bigskip
\noindent
{\bf Case (ND) and the definition of $z$ coordinates.}
If the Hessian determinant does not vanish, we are in the nondegenerate case.
Otherwise, if the Hessian determinant does vanish,
then since $\partial_2^2 \phi(v) \neq 0$ (which is by definition equivalent to $k = 2$),
there is a coordinate system of the form
\begin{align*}
x_1 &= v_1 + z_1,\\
x_2 &= v_2 + z_2 + A z_1,
\end{align*}
with $A$ unique, such that $\phi^x(x) = \phi^z(z)$,
and such that the $z_1^2$ and $z_1 z_2$ terms in Taylor expansion of $\phi^z$ at the origin vanish, i.e.,
\begin{align*}
\partial_1^2 \phi^z(0) = 0 \qquad \text{and} \qquad \partial_1 \partial_2 \phi^z(0) = 0.
\end{align*}
In particular, the coordinate systems $y$ and $z$ cannot coincide since the term $y_1 y_2$ does not vanish.
This implies $B \coloneqq A - \gamma v_2 v_1^{-1} \neq 0$
(compare \eqref{eq_basic_coordinate_change_reverse} and \eqref{eq_basic_coordinate_change_wz}).

\bigskip
\noindent
{\bf Case (FP) and the reduction to $A \neq 0$.}
Let us now prove that we may reduce ourselves to the case
\begin{align*}
A \neq 0.
\end{align*}
If $t_0 = 0$ (i.e., $v_2 = 0$), then we always have $A = B \neq 0$.
The second possibility is $t_0 \neq 0$, and if in this case we would have $A = 0$,
then $z$ and $x$ coordinates would coincide (up to a translation) which implies
$\partial_{x_1}^2 \phi^x(v) = \partial_{z_1}^2 \phi^z(0) = 0$.
Thus, by flipping coordinates, we would have that the $k$ associated to the flipped coordinates is $\geq 3$,
and so we would be in the case where the $y$ coordinates associated to the flipped coordinates would suffice,
i.e., we could apply Lemma \ref{lemma_k_greater_than_3}.

This is also the reason why in the case when $\rho = 1$, $\alpha_1 = 1$, and $\alpha_2 \neq 1$,
it always sufficed to flip coordinates.
The calculation below the corresponding table in Subsection \ref{subsection_forms_table_1} shows that
$\mathcal{H}_\phi(v) = 0$ implies $\partial_1^2 \phi (v) = \partial_1 \partial_2 \phi (v) = 0$,
which in turn implies that one always has $A = 0$.

\bigskip
\noindent
{\bf The normal form in $z$ coordinates.}
Now that we may assume $A \neq 0$,
our first step is to write down the Euler equation for homogeneous functions in $z$ coordinates.
The Euler equation is
\begin{align*}
\rho \phi(x) = \alpha_1 x_1 \partial_1 \phi(x) + \alpha_2 x_2 \partial_2 \phi(x).
\end{align*}
By the definition of $z$ coordinates we have
\begin{align*}
\partial_{x_1} &= \partial_{z_1} - A \partial_{z_2}, \\
\partial_{x_2} &= \partial_{z_2}.
\end{align*}
Thus, the Euler equation in $z$ coordinates is
\begin{align}
\label{eq_Euler_z}
\rho \phi^z(z)
   &= \alpha_1 (v_1+z_1) \partial_1 \phi^z(z) - \alpha_1 A (v_1+z_1) \partial_2 \phi^z(z) + \alpha_2 (v_2 + z_2 + A z_1) \partial_2 \phi^z(z) \nonumber \\
   &= \alpha_1 (v_1+z_1) \partial_1 \phi^z(z) + \Big( -\alpha_1 v_1 B + A(-\alpha_1+\alpha_2)z_1 + \alpha_2 z_2 \Big) \partial_2 \phi^z(z).
\end{align}

We now claim that if $\partial_1^{\tau_1+1} \phi^z(0) = \partial_1^{\tau_1} \partial_2 \phi^z(0) = 0$ for all $1 \leq \tau_1 < N$ for some $N \geq 2$,
then $\partial_1^{N+1} \phi^z(0) = 0$ if and only if $\partial_1^{N} \partial_2 \phi^z(0) = 0$.
But this is almost obvious. Namely, we just take the derivative $\partial_1^N$ at $0$ in the above Euler equation and get
\begin{align*}
\rho \partial_1^N \phi^z(0)
   &= \alpha_1 v_1 \partial^{N+1}_1 \phi^z(0) + \alpha_1 N \partial^{N}_1 \phi^z(0) \\
   &\quad - \alpha_1 v_1 B \partial^{N}_1 \partial_2 \phi^z(0)
          + AN(-\alpha_1+\alpha_2) \partial^{N-1}_1 \partial_2 \phi^z(0).
\end{align*}
Using the assumption on vanishing derivatives we get
\begin{align}
\label{section_forms_normal_form_z_derivative_rel}
\partial^{N+1}_1 \phi^z(0) = B \partial^{N}_1 \partial_2 \phi^z(0).
\end{align}
As we noted above $B \neq 0$ and our claim follows.

\medskip
Now recall that $\partial^{2}_1 \phi^z(0) = 0$ and $\partial_1 \partial_2 \phi^z(0) = 0$.
Thus, the previously proved claim implies in particular by an inductive argument in $N$ that
either there is a $\tilde{k} \in \N$ such that $3 \leq \tilde{k} < \infty$, satisfying
\begin{align*}
\tilde{k}
  &= \min \{j \geq 2 : \partial^{j}_1 \phi^z(0) \neq 0 \} \\
  &= \min \{j \geq 2 : \partial^{j-1}_1 \partial_2 \phi^z(0) \neq 0 \},
\end{align*}
and
\begin{align}
\label{section_forms_normal_form_z}
\phi_v^z(z) = z_1^{\tilde{k}} r_1(z) + z_1^{\tilde{k}-1} z_2 r_2(z) + z_2^2 r_3(z),
\end{align}
where $r_i(0) \neq 0$, $i = 1,2,3$, or that
\begin{align*}
\phi_v^z(z) = z_1^N r_{N,1}(z) + z_1^{N-1} z_2 r_{N,2}(z) + z_2^2 r_3(z),
\end{align*}
for any $N \in \N$, which we shall consider as the case when $\tilde{k} = \infty$.

\bigskip
\noindent
{\bf The normal form in $w$ coordinates.}
It will be advantageous to use $w$ coordinates where
unlike in \eqref{section_forms_normal_form_z} the $w_1^{\tilde{k}-1} w_2$ term is no longer present,
i.e., that we may write:
\begin{align}
\label{eq_phi_first_w_form}
\phi_v^w(w) = w_1^{\tilde{k}} r_1(w) + w_2^2 r_2(w).
\end{align}
This fact follows directly from \eqref{section_forms_normal_form_z_derivative_rel} and from
\begin{align*}
\partial_{w_1} &= \partial_{z_1}, \\
\partial_{w_2} &= \partial_{z_2} - \frac{1}{B}\partial_{z_1},
\end{align*}
which we get from the definition of $w$ coordinates \eqref{eq_basic_coordinate_change_wz}.
Actually, we can gain more information, especially in the case when $\gamma = 1$.
To see this let us rewrite the Euler equation in $w$ coordinates by using \eqref{eq_Euler_z}:
\begin{align*}
\frac{\rho}{\alpha_1} \phi^w(w)
   &= \Big( v_1+w_1 - \frac{1}{B}w_2 \Big) \partial_1 \phi^w(w) \nonumber \\
   &\quad + \Big( - v_1 B + A(\gamma-1)(w_1 - \frac{1}{B}w_2) + \gamma w_2 \Big) (\partial_2 + \frac{1}{B} \partial_1) \phi^w(w) \nonumber \\
   &= \Big( \frac{B+A(\gamma-1)}{B} w_1 + \frac{(B-A)(\gamma-1)}{B^2} w_2 \Big) \partial_1 \phi^w(w) \nonumber \\
   &\quad + \Big( - v_1 B + A(\gamma-1)w_1  + \frac{B\gamma - A(\gamma-1)}{B} w_2 \Big) \partial_2 \phi^w(w).
\end{align*}

\medskip

{\bf Case $\gamma = 1$.}
Here the Euler equation reduces to
\begin{align}
\label{eq_Euler_w_gamma_1}
\frac{\rho}{\alpha_1} \phi^w(w)
   = w_1 \partial_1 \phi^w(w) + ( - v_1 B  + w_2) \partial_2 \phi^w(w).
\end{align}
Taking the $\partial^\tau = \partial_1^{\tau_1} \partial_2^{\tau_2}$ derivative
and evaluating at $0$ one gets
\begin{align*}
\frac{\rho}{\alpha_1} \partial^\tau \phi^w(0)
   = \tau_1 \partial^\tau \phi^w(0)
   - v_1 B \partial_1^{\tau_1} \partial_2^{\tau_2+1} \phi^w(0)
   + \tau_2 \partial^\tau \phi^w(0),
\end{align*}
which can be rewritten as
\begin{align*}
\Big(\frac{\rho}{\alpha_1}-|\tau|\Big) \partial^\tau \phi^w(0)
   =  - v_1 B \partial_1^{\tau_1} \partial_2^{\tau_2+1} \phi^w(0).
\end{align*}
From this and the fact from \eqref{eq_phi_first_w_form} that
$\partial^\tau \phi^w(0) = 0$ for all $\tau$ satisfying
$|\tau| = \tau_1+\tau_2 \geq 2$, $0 \leq \tau_1 \leq \tilde{k}-1$,  and $0 \leq \tau_2 \leq 1$,
one easily gets by induction on $\tau_2$ that
\begin{align}
\label{diff_initial_conditions}
\partial_1^{\tau_1} \partial_2^{\tau_2} \phi^w (0) = 0
\quad \text{when} \quad |\tau| = \tau_1+\tau_2 \geq 2, \,\, 1 \leq \tau_1 \leq \tilde{k}-1.
\end{align}

We may actually prove a stronger claim,
namely that
\begin{align}
\label{diff_claim}
\begin{split}
\partial_1^{\tau_1} \phi^w(0,w_2) &\equiv 0,  \qquad \quad \text{for} \,\, 2 \leq \tau_1 \leq \tilde{k}-1,\\
\partial_1 \phi^w(0,w_2) &\equiv \partial_1 \phi^w(0).
\end{split}
\end{align}
In order to obtain this we take the $\partial_1^{\tau_1}$ derivative in \eqref{eq_Euler_w_gamma_1} and
evaluate it at $(0,w_2)$ to get
\begin{align*}
\Big(\frac{\rho}{\alpha_1}- \tau_1 \Big) \partial_1^{\tau_1} \phi^w(0,w_2)
   = ( - v_1 B  + w_2) \partial_2 \partial_1^{\tau_1} \phi^w(0,w_2).
\end{align*}
We note that this is a simple ordinary differential equation in $w_2$ of first order.
It has a unique solution for $2 \leq \tau_1 \leq \tilde{k}-1$ since $- v_1 B  + w_2 \neq 0$ for small $w_2$,
and since we can take \eqref{diff_initial_conditions} as initial conditions.
The claim for $2 \leq \tau_1 \leq \tilde{k}-1$ follows
since $\partial_1^{\tau_1} \phi^w(0,w_2) \equiv 0$ is obviously a solution.
For $\tau_1 = 1$ we note that the case $\rho/\alpha_1 - \tau_1 = 0$ is trivial,
and the solution is a unique constant function (necessarily equal to $\partial_1 \phi^w(0)$).
When $\tau_1 = 1$ and $\rho/\alpha_1 - \tau_1 \neq 0$,
then the differential equation evaluated at $w_2 = 0$ gives us that
$\partial_1 \partial_2 \phi^w(0) = 0$ implies $\partial_1 \phi^w(0) = 0$,
which again means that $\partial_1 \phi^w(0,w_2) \equiv 0$ is the unique solution 
of the given differential equation.
We have thus proven \eqref{diff_claim}.

Now by using Taylor approximation in $w_1$ for a fixed $w_2$,
and the just proven fact
for the mapping $w_2 \mapsto \partial_1^{\tau_1} \phi^w(0,w_2)$ for $1 \leq \tau_1 \leq \tilde{k}-1$,
we obtain that the normal form of $\phi^w$ \eqref{eq_phi_first_w_form}
in the case $\gamma = 1$ can be rewritten as
\begin{align*}
\phi_v^w(w) = w_1^{\tilde{k}} r_1(w) + w_2^2 r_2(w_2),
\end{align*}
where $r_1(0), r_2(0) \neq 0$.
Note that now $r_2$ depends only on $w_2$.
This corresponds to Normal form (ii.w) when $\tilde{k}$ is finite and to Normal form (i.w1) otherwise.

\medskip

{\bf Case $\gamma \neq 1$.}
In this case we use our assumption that $A \neq 0$ in a critical way.
Here it will be important to know what happens with $\partial_1^{\tau_1} \partial_2^2 \phi^w(0)$ for $0 \leq \tau_1 \leq \tilde{k}-1$,
and also how one can rewrite the normal form of the Hessian determinant $\mathcal{H}_{\phi^w}$ (and in particular its root).

Let us begin by taking the $\partial_1^{\tau_1} \partial_2$ derivative of the Euler equation in $w$ coordinates
and evaluating it at $w = 0$.
One gets
\begin{align*}
\frac{\rho}{\alpha_1} \partial_1^{\tau_1} \partial_2 \phi^w(0)
   &= \tau_1 \frac{B+A(\gamma-1)}{B} \partial_1^{\tau_1} \partial_2 \phi^w(0)  + \frac{(B-A)(\gamma-1)}{B^2} \partial_1^{\tau_1+1} \phi^w(0) \\
   &\quad - v_1 B \partial_1^{\tau_1} \partial^2_2 \phi^w(0) + \tau_1 A(\gamma-1) \partial_1^{\tau_1-1} \partial_2^2 \phi^w(0) \\
   &\quad + \frac{B\gamma - A(\gamma-1)}{B} \partial_1^{\tau_1} \partial_2 \phi^w(0).
\end{align*}
Now recall again from \eqref{eq_phi_first_w_form} that
$\partial^\tau \phi^w(0) = 0$ holds for any $\tau$ satisfying
$|\tau| = \tau_1+\tau_2 \geq 2$, $0 \leq \tau_1 \leq \tilde{k}-1$,  and $0 \leq \tau_2 \leq 1$.
Thus, if $1 \leq \tau_1 \leq \tilde{k}-2$ then we get
\begin{align}
\label{section_forms_normal_form_w_info_1}
v_1 B \partial_1^{\tau_1} \partial^2_2 \phi^w(0)
   &= \tau_1 A(\gamma-1) \partial_1^{\tau_1-1} \partial_2^2 \phi^w(0),
\end{align}
and if $\tau_1 = \tilde{k}-1$, then
\begin{align*}
v_1 B \partial_1^{\tilde{k}-1} \partial^2_2 \phi^w(0)
   &= \frac{(B-A)(\gamma-1)}{B^2} \partial_1^{\tilde{k}} \phi^w(0)
    + (\tilde{k}-1) A(\gamma-1) \partial_1^{\tilde{k}-2} \partial_2^2 \phi^w(0),
\end{align*}
i.e., since $B-A = - \gamma v_2 v_1^{-1}$, we can rewrite this as
\begin{align}
\label{section_forms_normal_form_w_info_2}
v_1 B \partial_1^{\tilde{k}-1} \partial^2_2 \phi^w(0)
    + \frac{v_2 \gamma(\gamma-1)}{v_1 B^2} \partial_1^{\tilde{k}} \phi^w(0)
    = (\tilde{k}-1) A(\gamma-1) \partial_1^{\tilde{k}-2} \partial_2^2 \phi^w(0).
\end{align}
Now since $A,B,v_1 \neq 0$, and $\gamma \neq 1$,
from \eqref{section_forms_normal_form_w_info_1}
we may conclude by induction on $\tau_1$
that for $0 \leq \tau_1 \leq \tilde{k}-2$ one has
\begin{align*}
\partial_1^{\tau_1} \partial_2^2 \phi^w(0) \neq 0.
\end{align*}

In order to unravel what is happening with $\partial_1^{\tilde{k}-1} \partial_2^2 \phi^w(0)$
we need to investigate the root of $\mathcal{H}_{\phi^w}$.
For this we want to solve the equation
\begin{align*}
x_2 - t_0 x_1^\gamma
  &= y_2 - \binom{\gamma}{2} v_1^{-2} v_2 y_1^2 + \mathcal{O}(y_1^3) \\
  &= 0.
\end{align*}
in the $w$ coordinates, representing the homogeneity curve through $v$.
Recall that by \eqref{eq_basic_coordinate_change_yzw} we have $y_1 = w_1 - w_2/B$, $y_2 = B w_1$, and so we want to solve
\begin{align*}
B w_1 - \binom{\gamma}{2} v_1^{-2} v_2 (w_1 - \frac{1}{B}w_2)^2 + \mathcal{O}((w_1 - \frac{1}{B}w_2)^3) = 0
\end{align*}
for the $w_1$ variable in terms of the $w_2$ variable when $|w_1|, |w_2|$ are small numbers.
Using the above equation one gets by a simple calculation that
\begin{align}
\label{section_forms_normal_form_w_root}
w_1 &= \frac{v_2 \gamma (\gamma-1)}{2 v_1^2 B^3} w_2^2 + \mathcal{O}(w_2^3) \nonumber \\
    &= w_2^2 \tilde{\omega}(w_2),
\end{align}
and $\tilde{\omega} \equiv 0$ if and only if $v_2 = 0 = t_0$.
Note that we have the precise value of $\tilde{\omega}(0)$.
Using this we can now write down the normal form of $w$ as
\begin{align}
\label{eq_phi_last_w_form}
\phi_v^w(w)
  &= w_1^{\tilde{k}} r_1(w) + w_2^2 r_2(w) \nonumber \\
  &= (w_1 - w_2^2 \tilde{\omega}(w_2))^{\tilde{k}} r_1(w)
     + w_2^2 \Big(r_2(w) + {\tilde{k}} w_1^{\tilde{k}-1} \frac{v_2 \gamma (\gamma-1)}{2 v_1^2 B^3} r_1(w) \Big)
     + \mathcal{O}(w_2^4) \nonumber \\
  &= (w_1 - w_2^2 \tilde{\omega}(w_2))^{\tilde{k}} \tilde{r}_1(w) + w_2^2 \tilde{r}_2(w),
\end{align}
where one can easily check by using
\eqref{section_forms_normal_form_w_info_1}, \eqref{section_forms_normal_form_w_info_2},
\eqref{section_forms_normal_form_w_root}, and \eqref{eq_phi_last_w_form}
that $\partial_1^{\tau_1} \tilde{r}_2 (0) \neq 0$ for all $0 \leq \tau_1 \leq \tilde{k}-1$,
and that in fact one has the relations
\begin{align*}
v_1 B \partial_1^{\tau_1} \tilde{r}_2 (0) = \tau_1 A (\gamma-1) \partial_1^{\tau_1-1} \tilde{r}_2 (0)
\end{align*}
for $1 \leq \tau_1 \leq \tilde{k}-1$.
If $\tilde{k} = \infty$, then the above normal form in \eqref{eq_phi_last_w_form} corresponds to Normal form (i.w2).
Otherwise we have $3 \leq \tilde{k} < \infty$ and two subcases.
Namely, if $t_0 \neq 0$ (i.e., $\tilde{\omega}(0) \neq 0$),
then the above normal form corresponds to Normal form (v),
and if $t_0 = 0$ (and therefore $\tilde{\omega} \equiv 0$),
then it corresponds to Normal form (iii).

\bigskip
\noindent
{\bf Determining $\tilde{k}$ in the special case when $\rho = \pm 1$ and $\alpha_1=\alpha_2 \neq \rho$.}
According to the corresponding table for this case in Subsection \ref{subsection_forms_table_1}
here we may assume $b_0, b_1 \neq 0$, and note that here $\gamma = 1$.
We prove that the Hessian determinant of $\phi$ vanishes at $v$ if and only if
\begin{align}
\label{section_forms_special_case_Hessian_vanish}
b_2 = (1-\rho\alpha_1) \frac{b_1^2}{b_0} = \Big(1 - \frac{\alpha_1}{\rho}\Big) \frac{b_1^2}{b_0}.
\end{align}
%
In this case we furthermore have that if $\tilde{k} < \infty$, then
\begin{align}
\label{section_forms_special_case_coef}
\begin{split}
b_j &= (\rho \alpha_1)^j j! \binom{\rho/\alpha_1}{j} \frac{b_1^j}{b_0^{j-1}}, \qquad \qquad \text{for} \,\, j = 2, \ldots, \tilde{k}-1 \\
b_{\tilde{k}} &\neq (\rho \alpha_1)^{\tilde{k}} \tilde{k}! \binom{\rho/\alpha_1}{\tilde{k}} \frac{b_1^{\tilde{k}}}{b_0^{\tilde{k}-1}},
\end{split}
\end{align}
and if $\tilde{k} = \infty$, then
\begin{align*}
b_j &= (\rho \alpha_1)^j j! \binom{\rho/\alpha_1}{j} \frac{b_1^j}{b_0^{j-1}}, \qquad \qquad \text{for} \,\, j \in \{2,3,\ldots\} .
\end{align*}
These formulae have already been shown for homogeneous polynomials in \cite[Lemma 2.2]{FGU04}.
Therefore, we only sketch how one can prove them in our slightly more general case.

Recall from \eqref{eq_phi_basic_form_series_y}
that we have the formal series for $\phi$ at $y = 0$:
\begin{align*}
\phi^y(y)
  &\approx (v_1+y_1)^{\frac{\rho}{\alpha_1}} b_0 + y_2 (v_1+y_1)^{\frac{\rho}{\alpha_1} - 1} b_1
   + \frac{1}{2!} (v_1+y_1)^{\frac{\rho}{\alpha_1}-2} y_2^2 b_2 + \ldots\\
  &= \sum_{j = 0}^\infty \frac{b_j}{j!} (v_1 + y_1)^{\frac{\rho}{\alpha_1}-j} y_2^j.
\end{align*}
From this one gets
\begin{align*}
\partial_1^2 \phi^y (0) &= b_0 \frac{\rho}{\alpha_1} \Big( \frac{\rho}{\alpha_1}-1 \Big) v_1^{\frac{\rho}{\alpha_1}-2}, &
\partial_1 \partial_2 \phi^y (0) &= b_1 \Big( \frac{\rho}{\alpha_1}-1 \Big) v_1^{\frac{\rho}{\alpha_1}-2}, &
\partial_2^2 \phi^y (0) &= b_2 v_1^{\frac{\rho}{\alpha_1}-2},
\end{align*}
and \eqref{section_forms_special_case_Hessian_vanish} follows by a direct computation
(recall that $\mathcal{H}_{\phi^y}(0) = 0$ if and only if $\mathcal{H}_{\phi}(v) = 0$).
More generally, we have
\begin{align}
\label{section_forms_special_case_taylor}
\partial^\tau \phi^y(0) = \tau_1! {\frac{\rho}{\alpha_1} - \tau_2 \choose \tau_1} v_1^{\frac{\rho}{\alpha_1}-|\tau|} b_{\tau_2}.
\end{align}

Let us now determine the relation between $y$ and $z$ when the Hessian determinant vanishes.
We may write
\begin{align*}
z_1 &= y_1,         &  \partial_{z_1} &= \partial_{y_1} + B \partial_{y_2}, \\
z_2 &= y_2 - B y_1, &  \partial_{z_2} &= \partial_{y_2}.
\end{align*}
Then by \eqref{section_forms_special_case_Hessian_vanish}
one gets that $\partial_1^2 \phi^z(0) = \partial_1 \partial_2 \phi^z(0) = 0$
if and only if
\begin{align*}
B = -\frac{b_0}{b_1} \frac{\rho}{\alpha_1}.
\end{align*}
From this we can determine the constant $A$ since it is equal to $t_0+B$,
i.e., $A = v_2/v_1 - (\rho b_0)/(\alpha_1 b_1)$.

One can now directly prove \eqref{section_forms_special_case_coef}
by induction in $j$ by using \eqref{section_forms_special_case_taylor},
and the fact that
$\partial_1^{j} \phi^z(0) = 0$ for $2 \leq j < \tilde{k}$ and $\partial_1^{\tilde{k}} \phi^z(0) \neq 0$
is equivalent to
\begin{align*}
\Big(\partial_1 - \frac{b_0}{b_1} \frac{\rho}{\alpha_1} \partial_2 \Big)^j \phi^y (0) &= 0, \qquad j = 2, \ldots, \tilde{k}-1,  \\
\Big(\partial_1 - \frac{b_0}{b_1} \frac{\rho}{\alpha_1} \partial_2 \Big)^{\tilde{k}} \phi^y (0) &\neq 0.
\end{align*}
We have already checked the induction base $j = 2$.

\subsection{Order of vanishing of the Hessian determinant}
\label{subsection_forms_derivations_Hessian}

In this subsection we determine the normal forms of the Hessian determinant of $\phi$
(or more precisely, the order of vanishing of the Hessian determinant of $\phi$),
as listed in Subsection \ref{subsection_forms_considerations}.
We recall from Subsection \ref{subsection_forms_considerations} that if $v_1 > 0$,
then one can write
\begin{align*}
\mathcal{H}_\phi (x) = (x_2 - t_0 x_1^{\gamma})^N q(x),
\end{align*}
where either $q$ is flat in $v$ (which we consider as the case $N = \infty$),
or $q(v) \neq 0$ and $0 \leq N < \infty$.
It remains to determine $N$ from the information provided by the normal forms of $\phi$.
We note that
\begin{align*}
N = \min\{j \geq 0 : (\partial_2^j \mathcal{H})(v) \neq 0 \}.
\end{align*}

\bigskip
\noindent
{\bf Normal form (i.y1).}
First we note by the normal form tables above
that this normal form appears only in cases when either $\gamma = 1$ or $t_0 = v_2 = 0$,
and so we have $\omega \equiv 0$.
Thus, by \eqref{eq_phi_basic_form_series_y} the function $\phi^y_v$ has the formal expansion:
\begin{align}
\label{subsection_forms_derivations_phi_expansion}
\phi_v^y(y)
  &= \frac{1}{k!} y_2^k \, (y_1+v_1)^{\rho/\alpha_1-k\gamma} \, g_k(y_2 (y_1+v_1)^{-1} +t_0) \nonumber \\
  &\approx \sum_{j = k}^\infty \frac{b_j}{j!} y_2^j \, (y_1+v_1)^{\rho/\alpha_1 - j\gamma},
\end{align}
and the Hessian determinant vanishes along $y_2 = 0$,
which means we need to determine what is the least $N$ such that $(\partial_2^N \mathcal{H}_{\phi^y}) (0) \neq 0$.
From the above expansion one obtains
\begin{align}
\label{subsection_forms_derivations_taylor_coeff}
\begin{split}
\partial_1^{\tau_1} \partial_2^{\tau_2} \phi^y (0)
  &= 0, \qquad |\tau| = \tau_1 + \tau_2 \geq 2, \,\, 0 \leq \tau_2 \leq k-1, \\
\partial^\tau \phi^y(0)
  &= \tau_1! {\frac{\rho}{\alpha_1} - \gamma \tau_2 \choose \tau_1} v_1^{\frac{\rho}{\alpha_1}- \tau_1 - \gamma \tau_2}
     b_{\tau_2}, \qquad \tau_2 \geq k.
\end{split}
\end{align}
By applying the general Leibniz rule to the definition of the Hessian determinant we get
\begin{align}
\label{subsection_forms_derivations_Leibniz}
\partial_2^N \mathcal{H}_{\phi^y}
  &= \partial_2^N (\partial_1^2 \phi^y \partial_2^2 \phi^y - (\partial_1 \partial_2 \phi^y)^2) \nonumber \\
  &= \sum_{n=0}^N {N \choose n} \Big(\partial_1^2 \partial_2^{n} \phi^y \, \partial_2^{N+2-n} \phi^y - \partial_1 \partial_2^{n+1} \phi^y \, \partial_1 \partial_2^{N+1-n} \phi^y \Big),
\end{align}
and one can easily check by using \eqref{subsection_forms_derivations_taylor_coeff}
that $\partial_2^N \mathcal{H}_{\phi^y}(0) = 0$ for $N < 2k-2$.
For $N = 2k-2$ we get
\begin{align*}
\partial_2^{2k-2} \mathcal{H}_{\phi^y} (0)
  &= {2k-2 \choose k}  \partial_1^2 \partial_2^k \phi^y(0) \partial_2^{k} \phi^y (0) - {2k-2 \choose k-1} (\partial_1 \partial_2^k \phi^y)^2 (0) \\
  &= \Bigg[ {2k-2 \choose k} \Big( \frac{\rho}{\alpha_1} - k\gamma \Big) \Big( \frac{\rho}{\alpha_1} - k\gamma - 1 \Big)
   - {2k-2 \choose k-1} \Big(\frac{\rho}{\alpha_1} - k\gamma \Big)^2 \Bigg]
   b_k^2 v_1^{2 \frac{\rho}{\alpha_1} - 2k\gamma - 2} \\
  &= \Bigg[ \frac{k-1}{k} \Big( \frac{\rho}{\alpha_1} - k\gamma - 1 \Big)
   - \Big(\frac{\rho}{\alpha_1} - k\gamma \Big) \Bigg]
   {2k-2 \choose k-1} \Big( \frac{\rho}{\alpha_1} - k \gamma \Big) b_k^2 v_1^{2 \frac{\rho}{\alpha_1} - 2k\gamma - 2}.
\end{align*}
Thus, $\partial_2^{2k-2} \mathcal{H}_{\phi^y} (0) \neq 0$ if and only if
\begin{align}
\label{subsection_forms_derivations_conditions_1}
\frac{\rho}{\alpha_1} \notin \Big\{ k\gamma, \, k \gamma + 1-k \Big\}.
\end{align}
Let us now additionally assume that
\begin{align*}
b_{k+j} &= 0, \qquad 0 < j < \tilde{k}, \\
b_{k+\tilde{k}} &\neq 0,
\end{align*}
for some $\tilde{k} \geq 1$.

\medskip

{\bf Case when $\frac{\rho}{\alpha_1} = k \gamma$.}
By examining the term $j = k$ in \eqref{subsection_forms_derivations_phi_expansion}
we note that in this case we additionally have
\begin{align*}
\partial_2^{k} \phi^y (0) &\neq 0 \\
\partial_1^{\tau_1} \partial_2^{k} \phi^y (0) &= 0, \qquad \tau_1 \geq 1.
\end{align*}
Now by using the information in \eqref{subsection_forms_derivations_taylor_coeff},
the above additional assumption that $b_{k+j} = 0$ for $0 < j < \tilde{k}$, $b_{k+\tilde{k}} \neq 0$,
and the Leibniz formula \eqref{subsection_forms_derivations_Leibniz}
a straightforward calculation yields that
$\partial_2^N \mathcal{H}_{\phi^y}(0) = 0$ for $N < 2k+\tilde{k}-2$ and $\partial_2^{2k+\tilde{k}-2} \mathcal{H}_{\phi^y}(0) \neq 0$,
i.e., we have the precise order of vanishing of the Hessian determinant.

\medskip

{\bf Case when $\frac{\rho}{\alpha_1} = k \gamma + 1-k$.}
Again, by a straightforward calculation using the Leibniz formula one gets that
$\partial_2^N \mathcal{H}_{\phi^y}(0) = 0$ for $N < 2k+\tilde{k}-2$ and
we have for $N = 2k+\tilde{k}-2$:
\begin{align*}
\partial_2^{2k+\tilde{k}-2} \mathcal{H}_{\phi^y}(0)
  &=        {2k+\tilde{k}-2 \choose k}  \partial_1^2 \partial_2^k \phi^y(0) \partial_2^{k+\tilde{k}} \phi^y (0) +
            {2k+\tilde{k}-2 \choose k-2}  \partial_1^2 \partial_2^{k+\tilde{k}} \phi^y(0) \partial_2^{k} \phi^y (0) \\
  &\qquad - 2{2k+\tilde{k}-2 \choose k-1} \partial_1 \partial_2^k \phi^y (0) \, \partial_1 \partial_2^{k+\tilde{k}} \phi^y (0).
\end{align*}
Thus
\begin{align*}
\Bigg({2k+\tilde{k}-2 \choose k-2} \Bigg)^{-1} \partial_2^{2k+\tilde{k}-2} \mathcal{H}_{\phi^y}(0)
  &=        \frac{(k+\tilde{k})(k+\tilde{k}-1)}{(k-1)k}  \partial_1^2 \partial_2^k \phi^y(0) \partial_2^{k+\tilde{k}} \phi^y (0) +
            \partial_1^2 \partial_2^{k+\tilde{k}} \phi^y (0) \partial_2^{k} \phi^y (0) \\
  &\qquad - \frac{2(k+\tilde{k})}{k-1} \partial_1 \partial_2^k \phi^y (0) \, \partial_1 \partial_2^{k+\tilde{k}} \phi^y (0).
\end{align*}
This is equal to zero when the expression
\begin{align*}
&(k+\tilde{k})(k+\tilde{k}-1) \partial_1^2 \partial_2^k \phi^y (0) \partial_2^{k+\tilde{k}} \phi^y (0) \\
 &+ (k-1)k \partial_1^2 \partial_2^{k+\tilde{k}} \phi^y (0) \partial_2^{k} \phi^y (0)
 - 2k(k+\tilde{k})\partial_1 \partial_2^k \phi^y (0) \, \partial_1 \partial_2^{k+\tilde{k}} \phi^y (0)
\end{align*}
equals zero.
Plugging in the values of the derivatives from \eqref{subsection_forms_derivations_taylor_coeff}
one obtains that the above expression equals to
\begin{align*}
(k+\tilde{k})(k+\tilde{k}-1) (1-k)(-k)
 + (k-1)k (1-k-\tilde{k}\gamma) (-k-\tilde{k}\gamma)
 - 2k(k+\tilde{k}) (1-k) (1-k-\tilde{k}\gamma),
\end{align*}
up to a nonzero constant factor.
Factoring out $(1-k)(-k)$ we get
\begin{align*}
(k+\tilde{k})(k+\tilde{k}-1)
 + (k+\tilde{k}\gamma-1) (k+\tilde{k}\gamma)
 - 2(k+\tilde{k}) (k+\tilde{k}\gamma-1)
\end{align*}
and this equals zero if and only if $\gamma \in \{1, (\tilde{k}+1)/\tilde{k} \}$.

The condition $\frac{\rho}{\alpha_1} = k \gamma + 1 - k$ tells us that if $\gamma = 1$ then $\rho = \alpha_1 = \alpha_2 = 1$,
and from the normal form tables we see that this is precisely when the Hessian determinant vanishes of infinite order.

In the case $\gamma = (\tilde{k}+1)/\tilde{k}$
we get that $\rho = 1$, $\alpha_1 = \tilde{k}/(k+\tilde{k})$ and $\alpha_2 = (\tilde{k}+1)/(k+\tilde{k})$.
It seems that in this case the order of vanishing of the Hessian determinant
depends explicitly on the values $b_j$,  and so, in contrast to the previous cases,
one cannot relate in an easy way the order of vanishing of the Hessian determinant and
the form of $\phi$ in \eqref{subsection_forms_derivations_phi_expansion}.
As we shall not need the precise order of vanishing of the Hessian determinant in this case,
we do not pursue this question further.


\bigskip
\noindent
{\bf Other normal forms.}
First we recall that Normal form (i.y2) was dealt with in Subsection \ref{subsection_forms_flat},
and there it was already determined that the Hessian vanishes of infinite order (i.e., it is flat).

In all the remaining normal forms we use either $y$ or $w$ coordinates,
and so (as already noted in Subsection \ref{subsection_forms_considerations})
the Hessian determinant in these coordinates has the normal form
\begin{align*}
\mathcal{H}_{\phi^u}(u) = (u_2-u_1^2 \psi(u_1))^{N} q(u),
\end{align*}
where $u$ can represent either $y$ or $w$ coordinates,
and where either $N$ is finite and $q(0) \neq 0$,
or the Hessian determinant is flat (in which case we consider $N$ to be infinite).
The function $\psi$ is equal to either $\omega$ or $\tilde{\omega}$.
Our goal is to determine $N = \min \{ j \geq 0 : (\partial_2^j \mathcal{H}_{\phi^u}) (0) \neq 0 \}$.

We first note that we can rewrite all the remaining normal forms as either
\begin{align}
\label{eq_nf_case_1}
\phi_v^u(u) = (u_2-u_1^2 \psi(u_1))^{k_0} r(u)
\end{align}
or
\begin{align}
\label{eq_nf_case_2}
\phi_v^u(u) = u_1^2 r_1(u) + u_2^{k_0} r_2(u),
\end{align}
where $r(0), \psi(0), r_1(0), r_2(0) \neq 0$, and $k_0 \geq 2$ in the first case and $k_0 \geq 3$ in the second.
In the second case $k_0 = \infty$ is allowed with an obvious interpretation.
Note that the second case \eqref{eq_nf_case_2} includes Normal forms (ii), (iii), (iv), (v),
and also subcases of (i) where the $w$ coordinates are used.

For both cases \eqref{eq_nf_case_1} and \eqref{eq_nf_case_2}
one can use the Leibniz rule \eqref{subsection_forms_derivations_Leibniz}
and the information on the Taylor series of $\phi_v^u$ gained from these normal forms
to obtain the order of vanishing of the Hessian determinant (in the $\partial_{u_2}$ direction) by a direct calculation.
In the first case \eqref{eq_nf_case_1} one gets that the order of vanishing is $N = 2k_0-3$ and
in the second case \eqref{eq_nf_case_2} one gets that $N = k_0-2$
(or that the Hessian determinant is flat if $k_0 = \infty$).



\section{Fourier restriction when a mitigating factor is present}
\label{section_strichartz_mitigating}

In this section we prove Theorem \ref{theorem_Strichartz_mitigating}, i.e., the Fourier restriction estimate
\begin{align*}
\Vert \widehat{f} \Vert_{L^2(\mathrm{d}\mu)} \leq C \Vert f\Vert_{L^{\mathfrak{p}_3}_{x_3}(L^{\mathfrak{p}_1}_{(x_1,x_2)})},
\end{align*}
where $\mu$ is the surface measure
\begin{align*}
\langle \mu, f \rangle = \int_{\R^2\setminus\{0\}} f(x,\phi(x)) \, |\mathcal{H}_\phi(x)|^{\sigma} \, \mathrm{d}x
\end{align*}
and the exponents are
\begin{align*}
\Big(\frac{1}{\mathfrak{p}_1'}, \frac{1}{\mathfrak{p}_3'}\Big) = \Big(\frac{1}{2} - \sigma, \sigma \Big).
\end{align*}
We assume $0 \leq \sigma < 1/2$ when only adapted normal forms appear,
and $0 \leq \sigma \leq 1/3$ if a non-adapted normal form appears.
Since the case $\sigma = 0$ follows directly by Plancherel, we may assume $\sigma > 0$.

Our assumptions in this case are that the Hessian determinant $\mathcal{H}_\phi$
does not vanish of infinite order anywhere (i.e., condition (H2) is satisfied).
According to Subsection \ref{subsection_prelim_further} we may restrict our attention to the localized measure
\begin{align*}
\langle \mu_{0,v}, f \rangle = \int_{\R^2\setminus\{0\}} f(x,\phi_v(x-v)) \, \eta_v(x) \, |\mathcal{H}_\phi(x)|^{\sigma} \, \mathrm{d}x,
\end{align*}
where $v = (v_1, v_2)$ satisfies $v_1 \sim 1$, and either $v_2 = 0$ or $v_2 \sim 1$,
and where $\eta_v$ is a smooth nonnegative function with support in a small neighbourhood of $v$.

After changing to $y$ or $w$ coordinates from Section \ref{section_forms} we get that $\mu_{0,v}$ can be rewritten as
\begin{align*}
\langle \nu, f \rangle = \int f(x,\phi_{\text{loc}}(x)) \, a(x) \, |\mathcal{H}_{\phi_{\text{loc}}}(x)|^{\sigma} \, \mathrm{d}x,
\end{align*}
where now $a$ is smooth,
nonnegative, and supported in a small neighbourhood of the origin,
and where we have for $\phi_{\text{loc}}$ the normal form cases (i)-(vi) from Proposition \ref{proposition_normal_forms}.
Recall that since we assume (H2), in case (i) of Proposition \ref{proposition_normal_forms}
the function $\varphi$ vanishes identically.

\medskip

The strategy will be to appropriately localize and rescale the problem,
and then to use the associated ``$R^* R$'' operator.
Let us begin by proving modifications of two essentially known results.
\begin{lemma}
\label{lemma_simple_stationary}
Let $\phi : \Omega \to \R$ be a smooth function
on an open set $\Omega \subseteq \R^2$ contained in a ball of radius $\lesssim 1$,
and let $\mathcal{H}_\phi = \partial_1^2 \phi \partial_2^2 \phi - (\partial_1 \partial_2 \phi)^2$
denote the Hessian determinant of $\phi$.
We consider the measure defined by
\begin{align*}
\langle \mu, f\rangle \coloneqq \int f(x_1,x_2,\phi(x)) a(x) \mathrm{d}x,
\end{align*}
where $a \in C_c^\infty(\Omega)$ satisfies
$\Vert \partial^\tau a \Vert_{L^\infty(\Omega)} \lesssim_{\tau} 1$ for all multiindices $\tau$.
If we assume that on $\Omega$ we have
$|\partial_1^2 \phi| \sim 1$,
$|\partial^\tau \phi| \lesssim_{\tau} 1$ for all multiindices $\tau$,
and that $|\mathcal{H}_\phi| \sim \varepsilon$ for a bounded,
strictly positive (but possibly small) constant $\varepsilon$, then
\begin{align*}
|\widehat{\mu}(\xi)| \lesssim \varepsilon^{-1/2} (1+|\xi|)^{-1}.
\end{align*}
The claim also holds if $\phi$ and $a$ depend on $\varepsilon$,
assuming that the implicit constants appearing in the lemma can be taken to be independent of $\varepsilon$.
\end{lemma}

\begin{proof}
By compactness and translating we may assume that $a$ is supported on a small neighbourhood of the origin.
We also assume for simplicity that $|\partial_1 \phi| \sim 1$, which can be achieved by applying a linear transformation to $\mu$.
The Fourier transform of $\mu$ is by definition
\begin{align*}
\widehat{\mu}(\xi) = \int e^{-i \Phi(x, \xi)} a(x) \mathrm{d}x,
\end{align*}
where the phase function is of the form
\begin{align*}
\Phi(x,\xi) = x_1 \xi_1 + x_2 \xi_2 + \phi(x) \xi_3,
\end{align*}
from which one easily sees that unless $|\xi_1| \sim |\xi_3| \gtrsim |\xi_2|$, we have a very fast decay.
Let us denote
\begin{align*}
s_1 = \frac{\xi_1}{\xi_3}, \quad s_2 = \frac{\xi_2}{\xi_3}, \quad \lambda = \xi_3,
\end{align*}
and rewrite the phase as
\begin{align*}
\Phi(x,\xi) = \lambda(s_1 x_1 + s_2 x_2 + \phi(x)),
\end{align*}
where now $|s_1| \sim 1$ and $|s_2| \lesssim 1$.

Now either the $x_1$ derivative of $\Phi$ has no zeros on the domain of integration
(e.g. when $s_1$ and $\partial_1 \phi(0)$ are of the same sign),
in which case we get a fast decay by integrating by parts,
or there is a unique zero $x_1^c = x_1^c(x_2; s_1,s_2)$, depending smoothly on $(x_2; s_1,s_2)$ by the implicit function theorem,
i.e., we have the relation
\begin{align}
\label{lemma_simple_stationary_point}
s_1 + (\partial_1 \phi)(x_1^c,x_2) = 0.
\end{align}
In this case we apply the stationary phase method and get that
\begin{align*}
\widehat{\mu}(\xi)
 = \lambda^{-1/2} \int e^{-i \lambda \Psi(x_2;s_1,s_2)} a(x_2, s_1, s_2; \lambda) \mathrm{d}x_2,
\end{align*}
where $a$ is a smooth function in $(x_2,s_1,s_2)$ and a classical symbol of order $0$ in $\lambda$,
and where $\Psi(x_2; s_1,s_2) \coloneqq s_1 x_1^c + s_2 x_2 + \phi(x_1^c,x_2) = \lambda^{-1} \Phi(x_1^c,x_2,\xi)$.

Taking the $x_2$ derivative of \eqref{lemma_simple_stationary_point} we get that
\begin{align*}
\partial_{x_2} x_1^c (x_2;s_1,s_2) = - \frac{\partial_1 \partial_2 \phi(x_1^c,x_2)}{\partial_1^2 \phi(x_1^c,x_2)},
\end{align*}
and the $x_2$ derivative of the new phase is by \eqref{lemma_simple_stationary_point}:
\begin{align*}
\lambda \partial_{x_2} \Psi(x_2; s_1, s_2)
  &= \lambda(s_1 \partial_{x_2} x_1^c + s_2 + \partial_{x_2} x_1^c \, \partial_1 \phi(x_1^c,x_2) + \partial_2 \phi(x_1^c, x_2)) \\
  &= \lambda(s_2 + \partial_2 \phi(x_1^c, x_2)).
\end{align*}
From this and the expression for $(x_1^c)'$ it follows that
\begin{align*}
\lambda \partial_{x_2}^2 \Psi(x_2; s_1, s_2)
  = \lambda \frac{\mathcal{H}_\phi (x_1^c,x_2)}{\partial_1^2 \phi(x_1^c,x_2)} \sim \lambda \varepsilon.
\end{align*}
Thus, we may apply the van der Corput lemma, which then delivers the claim of the lemma.
\end{proof}

The following lemma for obtaining mixed norm Fourier restriction estimates
goes back essentially to Ginibre and Velo \cite{GV92} (see also \cite{KT98}).
\begin{lemma}
\label{lemma_simple_estimates}
Assume that we are given a bounded open set $\Omega$ and functions
$\Phi \in C^\infty(\Omega;\R^2)$, $\phi \in C^\infty(\Omega; \R)$, $a \in L^\infty(\Omega)$.
Let us consider the measure
\begin{align*}
\langle \mu, f \rangle \coloneqq \int f(\Phi(x),\phi(x)) a(x) \mathrm{d}x
\end{align*}
and the operator $T : f \mapsto f * \widehat{\mu}$.
If $\Phi$ is injective and its Jacobian is of size $|J_\Phi| \sim A_1$,
then the operator $L^1_{x_3}(\R;L^2_{(x_1,x_2)}(\R^2)) \to L^\infty_{x_3}(\R;L^2_{(x_1,x_2)}(\R^2))$ norm of $T$
is bounded (up to a universal constant) by $A_1^{-1} \Vert a \Vert_{L^\infty}$.
If one has furthermore the estimate
\begin{align*}
|\widehat{\mu}(\xi)| \leq A_2 (1+|\xi_3|)^{-1},
\end{align*}
then for any $\sigma \in [0,1/2)$ and $(1/p_1',1/p_3') = (1/2-\sigma, \sigma)$
the operator $L^{p_3}_{x_3}(\R;L^{p_1}_{(x_1,x_2)}(\R^2)) \to L^{p_3'}_{x_3}(\R;L^{p_1'}_{(x_1,x_2)}(\R^2))$ norm of $T$
is bounded (up to a constant depending on $\sigma$) by $(A_1^{-1} \Vert a \Vert_{L^\infty})^{1-2\sigma} A_2^{2\sigma}$.
\end{lemma}

\begin{proof}
For functions on $\R^3$ let us denote by $\FT'$ the inverse Fourier transform in the first two variables. 
Then it suffices for the first claim to prove that
the $L^\infty$ norm of $\FT' \widehat{\mu}$ is bounded by $A_1^{-1} \Vert a \Vert_{L^\infty}$.
But this follows immediately since $\FT' \widehat{\mu}$ is equal by Fourier inversion
to the Fourier transform of $\mu$ in the third coordinate,
which is easily seen to be the function (up to a universal constant)
\begin{align*}
(x,\xi_3) \mapsto e^{-i \xi_3 \phi \circ \Phi^{-1}(x)} \, a \circ \Phi^{-1}(x) |J_\Phi(x)|^{-1}.
\end{align*}

For the second claim we introduce the operator $T_{\xi_3}g \coloneqq g * \widehat{\mu}(\cdot,\xi_3)$
defined for functions $g$ on $\R^2$ and a fixed $\xi_3 \in \R$.
Then the $L^1(\R^2) \to L^\infty(\R^2)$ norm of $T_{\xi_3}$ is bounded by $A_2 (1+|\xi_3|)^{-1}$,
and the $L^2(\R^2) \to L^2(\R^2)$ norm is bounded up to a universal constant by $A_1^{-1} \Vert a \Vert_{L^\infty}$.
Interpolating one gets that the $L^{p_1}(\R^2) \to L^{p_1'}(\R^2)$ norm is bounded by
\begin{align*}
(A_1^{-1} \Vert a \Vert_{L^\infty})^{1-2\sigma} A_2^{2\sigma} (1+|\xi_3|)^{-2\sigma}
\end{align*}
for $p_1'=(1/2 - \sigma)$ and $\sigma \in [0,1/2]$.
If one now writes a function $f$ on $\R^3$ as $f(\xi_1,\xi_2,\xi_3) = f(\xi', \xi_3) = f_{\xi_3}(\xi')$, then
\begin{align*}
T f (\xi',\xi_3)
   = \int (f_{\eta_3-\xi_3} * \widehat{\mu}(\cdot,\eta_3)) (\xi') \mathrm{d}\eta_3
   = \int T_{\eta_3} f_{\eta_3-\xi_3} \mathrm{d}\eta_3,
\end{align*}
and so the claim follows by the (weak) Young inequality for $\sigma < 1/2$.
\end{proof}


\subsection{Normal form (i)}
\label{subsection_strichartz_mitigating_i}

In this case the local form of the phase is
\begin{align*}
\phi_{\text{loc}}(x) = x_2^k r(x),
\end{align*}
where $r(0) \neq 0$ and the Hessian determinant vanishes of order $2k+k_0-2$ for some $k_0 \geq 0$,
i.e., it has the normal form
\begin{align*}
\mathcal{H}_{\phi_{\text{loc}}}(x) = x_2^{2k+k_0-2} q(x)
\end{align*}
for some smooth function $q$ satisfying $q(0) \neq 0$.

We begin by a dyadic decomposition $\nu = \sum_{j \gg 1} \nu_j$ in $x_2$
followed by scaling $x_2 \mapsto 2^{-j} x_2$.
Namely, for a $j \gg 1$ we define
\begin{align*}
\langle \nu_j, f \rangle = \int f(x,\phi_{\text{loc}}(x)) \, a(x) \, \chi_1(2^j x_2) \, |\mathcal{H}_{\phi_{\text{loc}}}(x)|^{\sigma} \, \mathrm{d}x,
\end{align*}
where $\chi_1(x_2)$ is supported where $|x_2| \sim 1$ and is such that $\sum_{j \in \Z} \chi_1(x_2) = 1$.
Thus, by a Littlewood-Paley argument it suffices to prove
\begin{align*}
\Vert \widehat{f} \Vert^2_{L^2(\mathrm{d}\nu_j)} \lesssim \Vert f\Vert^2_{L^{\mathfrak{p}_3}_{x_3}(L^{\mathfrak{p}_1}_{(x_1,x_2)})},
\end{align*}
with the implicit constant independent of $j$.
Rescaling, this is equivalent to
\begin{align}
\label{section_strichartz_mitigating_i_FRP}
\Vert \widehat{f} \Vert^2_{L^2(\mathrm{d}\tilde{\nu}_j)} \lesssim 2^{\sigma j k_0} \Vert f\Vert^2_{L^{\mathfrak{p}_3}_{x_3}(L^{\mathfrak{p}_1}_{(x_1,x_2)})},
\end{align}
where now
\begin{align*}
\langle \tilde{\nu}_j, f \rangle = \int f(x,\tilde{\phi}(x,2^{-j})) \, a(x,2^{-j}) \mathrm{d}x.
\end{align*}
The amplitude $a(x, 2^{-j})$ is now supported so that $|x_1| \ll 1$ and $|x_2| \sim 1$,
and it is $C^{\infty}$ having derivatives uniformly bounded.
The phase is
\begin{align*}
\tilde{\phi}(x,2^{-j})
   &= 2^{jk} \phi_{\text{loc}}(x_1, 2^{-j}x_2) \\
   &= x_2^k r(x_1, 2^{-j} x_2).
\end{align*}
From this we have $|\partial_2 \tilde{\phi}| \sim 1 \sim |\partial_2^2 \tilde{\phi}|$ and
one easily gets by using the definition of the Hessian determinant that
\begin{align*}
\mathcal{H}_{\tilde{\phi}}(x, 2^{-j})
  &= 2^{j(2k-2)} \mathcal{H}_{\phi_{\text{loc}}}(x_1, 2^{-j} x_2) \\
  &= 2^{-jk_0} x_2^{2k+k_0-2} q(x_1, 2^{-j}x_2).
\end{align*}
Thus $|\mathcal{H}_{\tilde{\phi}}(x, 2^{-j})| \sim 2^{-j k_0}$, from which
the estimate \eqref{section_strichartz_mitigating_i_FRP} follows by
an application of Lemma \ref{lemma_simple_stationary} and subsequently Lemma \ref{lemma_simple_estimates}.


\subsection{Preliminary rescaling for cases (ii)-(vi)}
\label{subsection_strichartz_mitigating_scaling}

In normal form cases (ii)-(vi) the principal face of $\mathcal{N}(\phi_{\text{loc}})$ is compact
(for the definition of the Newton polyhedron $\mathcal{N}(\phi_{\text{loc}})$
of a smooth phase function $\phi_{\text{loc}}$ see for example \cite{IM16} or \cite{Pa19}),
and so we use the scaling associated to it:
\begin{align*}
\delta_r^{\kappa} (x) = (r^{\kappa_1} x_1, r^{\kappa_2} x_2),
\end{align*}
where in cases (ii)-(v) we have $\kappa = (1/2,1/k)$ and in case (vi) we have $\kappa = (1/(2k),1/k)$.
In particular, for $j \gg 1$ we define
\begin{align*}
\langle \nu_j, f \rangle = \int f(x,\phi_{\text{loc}}(x)) \, a(x) \, \eta(\delta_{2^j}^\kappa x) \, |\mathcal{H}_{\phi_{\text{loc}}}(x)|^{\sigma} \, \mathrm{d}x,
\end{align*}
where $\eta$ is supported on an annulus and is such that $\sum_{j \in \Z} \eta(\delta_{2^j}^\kappa x) = 1$.
By using Littlewood-Paley theory we get that it is sufficient to prove
\begin{align*}
\Vert \widehat{f} \Vert^2_{L^2(\mathrm{d}\nu_j)} \lesssim \Vert f\Vert^2_{L^{\mathfrak{p}_3}_{x_3}(L^{\mathfrak{p}_1}_{(x_1,x_2)})}.
\end{align*}
Rescaling, the above estimate is equivalent to
\begin{align}
\label{section_strichartz_mitigating_scaling_FRP}
\Vert \widehat{f} \Vert^2_{L^2(\mathrm{d}\tilde{\nu}_j)} \lesssim \Vert f\Vert^2_{L^{\mathfrak{p}_3}_{x_3}(L^{\mathfrak{p}_1}_{(x_1,x_2)})},
\end{align}
where
\begin{align}
\label{section_strichartz_mitigating_scaling_measure}
\langle \tilde{\nu}_j, f \rangle = \int f(x,\tilde{\phi}(x,\delta)) \, |\mathcal{H}_{\tilde{\phi}}(x, \delta)|^{\sigma} \, a(x, \delta) \mathrm{d}x.
\end{align}
Here the amplitude $a(x, \delta)$ is supported on a fixed annulus around the origin,
\begin{align}
\label{section_strichartz_mitigating_delta_def_ii_iii_iv_v}
\delta = (\delta_0, \delta_1, \delta_2) \coloneqq (2^{-j\frac{k-1}{k}}, 2^{-j/2}, 2^{-j/k})
\end{align}
in cases (ii)-(v), and
\begin{align*}
\delta = (\delta_1, \delta_2) \coloneqq (2^{-j/(2k)}, 2^{-j/k})
\end{align*}
in case (vi).
The phase which one obtains in \eqref{section_strichartz_mitigating_scaling_measure} is
\begin{align*}
\tilde{\phi}(x,\delta)
  \coloneqq 2^{j} \phi_{\text{loc}}(\delta_1 x_1, \delta_2 x_2).
\end{align*}
The quantity $\delta_0$ will be appear only later when we use the explicit normal forms.
From the above phase form it follows that
\begin{align*}
\mathcal{H}_{\tilde{\phi}}(x, \delta) = 2^{j(k-2)/k} \mathcal{H}_{\phi_{\text{loc}}}(\delta_1 x_1, \delta_2 x_2)
\end{align*}
in cases (ii)-(v), and
\begin{align*}
\mathcal{H}_{\tilde{\phi}}(x, \delta) = 2^{j(2k-3)/k} \mathcal{H}_{\phi_{\text{loc}}}(\delta_1 x_1, \delta_2 x_2)
\end{align*}
in case (vi).


\subsection{Normal forms (ii) and (iii)}
\label{subsection_strichartz_mitigating_ii}

Using the normal forms for $\phi_{\text{loc}}$ one gets in these cases
\begin{align*}
\tilde{\phi}(x,\delta) &= x_1^2 r_1(\delta_1 x_1, \delta_2 x_2) + x_2^k r_2(\delta_1 x_1, \delta_2 x_2), \\
\mathcal{H}_{\tilde{\phi}}(x, \delta) &= x_2^{k-2} q(\delta_1 x_1, \delta_2 x_2),
\end{align*}
where $r_1(0), r_2(0), q(0) \neq 0$, and $k \geq 3$.
Hence, for the part where $|x_2| \gtrsim 1$ in \eqref{section_strichartz_mitigating_scaling_measure}
the Hessian is nondegenerate,
and so we may localize to $|x_1| \sim 1$ and $|x_2| \ll 1$,
and subsequently perform a dyadic decomposition in the $x_2$ coordinate, i.e., we define
\begin{align*}
\langle \nu_l, f \rangle
  &\coloneqq \int f(x,\tilde{\phi}(x,\delta)) \, |x_2|^{\sigma(k-2)} \, \chi_1(2^l x_2) \, a(x, \delta) \mathrm{d}x\\
  &= 2^{-l- l \sigma (k-2)} \int f(x_1,2^{-l}x_2,\tilde{\phi}(x_1,2^{-l}x_2,\delta)) \, a(x, \delta, 2^{-l}) \mathrm{d}x,
\end{align*}
where now the amplitude is supported in a domain where $|x_1| \sim 1 \sim |x_2|$,
and has uniformly bounded $C^N$ norm for any $N$.
Applying the Littlewood-Paley theorem again and rescaling, it is sufficient for us to prove
\begin{align}
\label{section_strichartz_mitigating_ii_scaling_FRP_final}
\Vert \widehat{f} \Vert^2_{L^2(\mathrm{d} \tilde{\nu}_{j,l})}
   \lesssim 2^{kl\sigma} \Vert f\Vert^2_{L^{\mathfrak{p}_3}_{x_3}(L^{\mathfrak{p}_1}_{(x_1,x_2)})},
\end{align}
where the rescaled measure is
\begin{align*}
\langle \tilde{\nu}_{j,l}, f \rangle = \int f(x,\tilde{\phi}(x_1, 2^{-l} x_2,\delta)) \, a(x, \delta, 2^{-l}) \mathrm{d}x.
\end{align*}
The phase has now the form
\begin{align}
\label{subsection_strichartz_mitigating_ii_scaling_phase}
x_1^2 r_1(\delta_1 x_1, \delta_2 x_2) + 2^{-kl} x_2^k r_2(\delta_1 x_1, 2^{-l} \delta_2 x_2)
\end{align}
on the domain $|x_1| \sim 1$ and $|x_2| \sim 1$, and its Hessian determinant is of size $2^{-kl}$.
By Lemma \ref{lemma_simple_stationary} we have
\begin{align*}
|\widehat{\tilde{\nu}}_{j,l}(\xi)| \lesssim 2^{kl/2} (1+|\xi|)^{-1}.
\end{align*}
And so the estimate \eqref{section_strichartz_mitigating_ii_scaling_FRP_final} follows by Lemma \ref{lemma_simple_estimates}.


\subsection{Normal form (iv)}
\label{subsection_strichartz_mitigating_iv}

In this case we get
\begin{align*}
\tilde{\phi}(x,\delta) &= x_1^2 r_1(\delta_1 x_1) + (x_2 - \delta_0 x_1^2 \psi(\delta_1 x_1))^k r_2(\delta_1 x_1, \delta_2 x_2), \\
\mathcal{H}_{\tilde{\phi}}(x, \delta) &= (x_2 - \delta_0 x_1^2 \psi(\delta_1 x_1))^{k-2} q(\delta_1 x_1, \delta_2 x_2),
\end{align*}
where $r_1(0), r_2(0), q(0), \psi(0) \neq 0$, and $k \geq 3$.
Therefore again, if $|x_2| \gtrsim 1$ the Hessian is nondegenerate
and therefore we may concentrate on $|x_1| \sim 1$ and $|x_2| \ll 1$ in \eqref{section_strichartz_mitigating_scaling_measure}.
We perform a dyadic decomposition,
though this time depending on how close we are to the root of the Hessian determinant, i.e., we define
\begin{align*}
\langle \nu_{l}, f \rangle
  &\coloneqq \int f(x,\tilde{\phi}(x,\delta)) \, |x_2 - \delta_0 x_1^2 \psi(\delta_1 x_1)|^{\sigma(k-2)} \,
             \chi_1(2^l (x_2 - \delta_0 x_1^2 \psi(\delta_1 x_1))) \, a(x, \delta) \mathrm{d}x.
\end{align*}
Next, after changing coordinates from $x_2$ to $x_2 + \delta_0 x_1^2 \psi(\delta_1 x_1)$ we may write
\begin{align}
\label{section_strichartz_mitigating_iv_measure_l}
\langle \nu_{l}, f \rangle
  = \int f(x_1, x_2 + \delta_0 x_1^2 \psi(\delta_1 x_1),\phi_{1}(x,\delta)) \, |x_2|^{\sigma(k-2)} \, \chi_1(2^l x_2) \, a_1(x, \delta) \, \mathrm{d}x,
\end{align}
where
\begin{align*}
\phi_{1}(x,\delta)
  &= x_1^2 r_1(\delta_1 x_1) + x_2^k r_2(\delta_1 x_1, \delta_2 x_2 + \delta_0 \delta_2 x_1^2 \psi(\delta_1 x_1, \delta_2 x_2)) \\
  &= x_1^2 r_1(\delta_1 x_1) + x_2^k r_2(\delta_1 x_1, \delta_2 x_2 + (\delta_1 x_1)^2 \psi(\delta_1 x_1, \delta_2 x_2)) \\
  &= x_1^2 r_1(\delta_1 x_1) + x_2^k \tilde{r}_2(\delta_1 x_1, \delta_2 x_2).
\end{align*}
The function $\tilde{r}_2$ is a smooth and nonzero at the origin.
Finally, we rescale in $x_2$ as $x_2 \mapsto 2^{-l} x_2$ and may write
\begin{align}
\label{section_strichartz_mitigating_iv_measure_l_rewritten}
\langle \nu_{l}, f \rangle
  = 2^{-l - l\sigma(k-2)} \int f(x_1, 2^{-l} x_2 + \delta_0 x_1^2 \psi(\delta_1 x_1),\phi_{j, l}(x,\delta,2^{-l}))
     \, \chi_1(x_1) \, \chi_1(x_2) \, a(x, \delta, 2^{-l}) \, \mathrm{d}x,
\end{align}
where the amplitude is a smooth function and the phase is
\begin{align*}
\phi_{j, l}(x,\delta)
  &= x_1^2 r_1(\delta_1 x_1) + 2^{-kl} x_2^k \tilde{r}_2(\delta_1 x_1, 2^{-l} \delta_2 x_2).
\end{align*}

\medskip

In order to obtain the estimate \eqref{section_strichartz_mitigating_scaling_FRP}
we shall need essentially a variant of Lemma \ref{lemma_simple_estimates}.
Namely, we shall consider the analytic family of operators $T_\zeta$
defined by convolution against the Fourier transform of the measure
\begin{align}
\label{section_strichartz_mitigating_iv_complex_measure}
\mu_\zeta \coloneqq \sum_{2^l \gg 1} 2^{l\sigma(k-2)} 2^{-l\zeta (k-2)} \nu_{l},
\end{align}
where $\zeta$ has real part between $0$ and $1/2$, and in particular, for a fixed $\xi_3 \in \R^3$,
we shall consider the operator $T_\zeta^{\xi_3} : f \mapsto f * \widehat{\mu}_\zeta(\cdot, \xi_3)$.
Note that we are interested in $\mu_{\sigma}$ since this is precisely the sum of measures $\nu_{l}$.

When the real part of $\zeta$ is $0$ (i.e., $\zeta = it$, $t \in \R$)
one considers the $L^2(\R^2) \to L^2(\R^2)$ estimate for which we use the equation \eqref{section_strichartz_mitigating_iv_measure_l}.
In \eqref{section_strichartz_mitigating_iv_measure_l} we see that the amplitude is of size $2^{-l \sigma (k-2)}$,
which is precisely what we need in \eqref{section_strichartz_mitigating_iv_complex_measure}.
Since the supports are disjoint when varying $l$, we get by a similar argumentation as in Lemma \ref{lemma_simple_estimates}
that the operator $L^2(\R^2) \to L^2(\R^2)$ norm of $T_{it}^{\xi_3}$ is $\lesssim 1$ (uniform in $\xi_3$ and $t$).

When the real part of $\zeta$ is $1/2$ we need to prove
\begin{align}
\label{section_strichartz_mitigating_iv_FT_estimate}
|\widehat{\mu}_{1/2+it} (\xi)| \lesssim (1+|\xi_3|)^{-1}
\end{align}
with implicit constant independent of $t$ and $\xi_3$,
since this would give us that
the operator norm of $T_{1/2+it}^{\xi_3}$ for mapping $L^1(\R^2) \to L^\infty(\R^2)$ is bounded by $(1+|\xi_3|)^{-1}$.

Thus, under the assumption that we have the estimate \eqref{section_strichartz_mitigating_iv_FT_estimate}
we may apply complex interpolation for each fixed $\xi_3$ to the analytic family of operators $T_\zeta^{\xi_3}$
and obtain that the operator norm of $T_{\sigma}^{\xi_3}$
between spaces $L^{\mathfrak{p}_1}(\R^2) \to L^{\mathfrak{p}_1'}(\R^2)$ is $\lesssim (1+|\xi_3|)^{-2\sigma}$,
and so in the same way as in the proof of Lemma \ref{lemma_simple_estimates}
the (weak) Young inequality in the $x_3$ direction implies \eqref{section_strichartz_mitigating_scaling_FRP}.

In proving \eqref{section_strichartz_mitigating_iv_FT_estimate} it suffices to show that
\begin{align*}
\sum_{2^l \gg 1} 2^{-l(1/2-\sigma)(k-2)} |\widehat{\nu}_{l} (\xi)| \lesssim (1+|\xi_3|)^{-1}
\end{align*}
for all $\xi \in \R^3$.
By \eqref{section_strichartz_mitigating_iv_measure_l_rewritten} the Fourier transform of a summand is
\begin{align*}
2^{-l(1/2-\sigma)(k-2)} \widehat{\nu}_{l} (\xi)
   = 2^{-kl/2} \int e^{-i \Phi (x, \xi, \delta, 2^{-l})} \, \chi_1(x_1) \, \chi_1(x_2) \, a(x, \delta, 2^{-l}) \mx,
\end{align*}
where the phase function is
\begin{align*}
\Phi (x, \xi, \delta, 2^{-l})
  &\coloneqq  \xi_1 x_1 + \xi_2 \delta_0 x_1^2 \psi(\delta_1 x_1) + \xi_3 x_1^2 r_1(\delta_1 x_1) \\
  &\quad + 2^{-l} \xi_2 x_2 + 2^{-kl} \xi_3 x_2^k \tilde{r}_2(\delta_1 x_1, 2^{-l} \delta_2 x_2).
\end{align*}

We see that when either $|\xi_1| \gg \max\{|\xi_2|, |\xi_3|\}$ or $|\xi_3| \gg \max\{|\xi_1|, |\xi_2|\}$
we can use integration by parts in the $x_1$ variable and get a very fast decay.
This is also the case when $|\xi_1| \sim |\xi_2|$ are much greater than $|\xi_3|$,
or when $|\xi_2| \sim |\xi_3|$ are much greater than $|\xi_1|$.
If we have $|\xi_2| \gtrsim |\xi_3|$, then we may use integration by parts in $x_2$ and get
\begin{align*}
|2^{-l(1/2-\sigma)(k-2)} \widehat{\nu}_{l} (\xi)|
   \lesssim 2^{-kl/2} (1+2^{-l}|\xi_2|)^{-1}
   \lesssim 2^{-kl/2} (1+2^{-l}|\xi_3|)^{-1},
\end{align*}
from which \eqref{section_strichartz_mitigating_iv_FT_estimate} follows since $k \geq 3$.
We are thus left with the case when $|\xi_1| \sim |\xi_3| \gg |\xi_2|$.

\medskip
\noindent
{\bf Case 1. $2^{-kl} |\xi_3| \lesssim 1$.}
Here we use the van der Corput lemma in $x_1$ only and get
\begin{align*}
|2^{-l(1/2-\sigma)(k-2)} \widehat{\nu}_{l} (\xi)| \lesssim 2^{-kl/2} |\xi_3|^{-1/2}.
\end{align*}
Summation in $l$ then gives precisely \eqref{section_strichartz_mitigating_iv_FT_estimate}.

\medskip
\noindent
{\bf Case 2. $2^{-l} |\xi_2| \nsim 2^{-kl} |\xi_3|$ and $2^{-kl} |\xi_3| \gg 1$.}
We may use in this case integration by parts in $x_2$ and then
the van der Corput lemma in $x_1$ and get
\begin{align*}
|2^{-l(1/2-\sigma)(k-2)} \widehat{\nu}_{l} (\xi)|
   &\lesssim 2^{-kl/2} |\xi_3|^{-1/2} \, (2^{-kl} |\xi_3|)^{-1} \\
   &\lesssim 2^{kl/2} |\xi_3|^{-3/2}.
\end{align*}
We may now sum in $l$.

\medskip
\noindent
{\bf Case 3. $2^{-l} |\xi_2| \sim 2^{-kl} |\xi_3| \gg 1$.}
Here we have by iterative stationary phase (first in $x_2$ and then in $x_1$) that
\begin{align*}
|2^{-l(1/2-\sigma)(k-2)} \widehat{\nu}_{l} (\xi)|
  \lesssim 2^{-kl/2} |\xi_3|^{-1/2} \, (2^{-kl} |\xi_3|)^{-1/2}
  = |\xi_3|^{-1}.
\end{align*}
Here we note that $2^{l(k-1)} \sim |\xi_3| \, |\xi_2|^{-1}$, and so we sum only over finitely many (i.e., $\mathcal{O}(1)$) $l$ for each fixed $\xi$.
Thus, here we also have the estimate \eqref{section_strichartz_mitigating_iv_FT_estimate}.


\subsection{Normal form (v)}
\label{subsection_strichartz_mitigating_v}

Recall that here
\begin{align*}
\phi_{\text{loc}}(x) &= x_1^2 r_1(x) + (x_2- x_1^2 \psi(x_1))^{k} r_2(x), \\
\mathcal{H}_{\phi_{\text{loc}}}(x) &= (x_2- x_1^2 \psi(x_1))^{k-2} q(x),
\end{align*}
where we know that $k \geq 3$, $r_1(0), r_2(0), q(0), \psi(0) \neq 0$.
Furthermore, recall that this corresponded to the $w$ coordinates when deriving the normal forms,
and we have shown that we additionally have in this case:
\begin{align*}
\partial_2^{\tau_2} r_1(0) \neq 0, \qquad \text{for all} \,\,\, \tau_2 \in \{0,1,\ldots,k-1\}.
\end{align*}
In fact, one has the relationship:
\begin{align*}
c \, \tau_2 \partial_2^{\tau_2-1} r_1(0) = \partial_2^{\tau_2} r_1(0), \qquad \text{for all} \,\,\, \tau_2 \in \{1,\ldots,k-1\},
\end{align*}
where $c$ is some fixed nonzero constant (see Subsection \ref{subsection_forms_derivations_w}).
This implies for example the relation:
\begin{align}
\label{section_strichartz_mitigating_v_coef_relation}
r_1(0) \, \partial_2^2 r_1(0) - 2 (\partial_2 r_1)^2(0) = 0.
\end{align}

From the above normal form we have
\begin{align*}
\tilde{\phi}(x,\delta) &= x_1^2 r_1(\delta_1 x_1, \delta_2 x_2) + (x_2 - \delta_0 x_1^2 \psi(\delta_1 x_1))^k r_2(\delta_1 x_1, \delta_2 x_2), \\
\mathcal{H}_{\tilde{\phi}}(x, \delta) &= (x_2 - \delta_0 x_1^2 \psi(\delta_1 x_1))^{k-2} q(\delta_1 x_1, \delta_2 x_2).
\end{align*}
We may as usual localize to $|x_1| \sim 1$ and $|x_2| \ll 1$.
We shall abuse the notation a bit and denote this localized measure again by $\tilde{\nu}_j$.
After changing coordinates from $x_2$ to $x_2 + \delta_0 x_1^2 \psi(\delta_1 x_1)$ we may write
\begin{align*}
\langle \tilde{\nu}_j, f \rangle
  = \int f(x_1, x_2 + \delta_0 x_1^2 \psi(\delta_1 x_1),\phi_1(x,\delta))
  \, |x_2|^{\sigma(k-2)} \, a_1(x,\delta) \, \chi_1(x_1) \, \chi_0(x_2) \, \mathrm{d}x
\end{align*}
with the phase being
\begin{align*}
\phi_1(x,\delta)
  &= x_1^2 \tilde{r}_1(\delta_1 x_1, \delta_2 x_2) + x_2^k \tilde{r}_2(\delta_1 x_1, \delta_2 x_2),
\end{align*}
where $\tilde{r}_1, \tilde{r}_2$ are smooth functions, nonzero at the origin,
and satisfy the same properties and relations as $r_1$ and $r_2$ mentioned at the beginning of this subsection.
As in the case (iv), we also decompose the measure $\tilde{\nu}_j$ as $\tilde{\nu}_j = \sum_l \nu_l$, where
\begin{align*}
\langle \nu_{l}, f \rangle
  = \int f(x_1,x_2 + \delta_0 x_1^2 \psi(\delta_1 x_1),\phi_1(x,\delta))
  \, |x_2|^{\sigma(k-2)} \, a_1(x,\delta) \, \chi_1(x_1)  \, \chi_1(2^l x_2)\,  \mathrm{d}x.
\end{align*}

Next, we shall be interested in the rescaled phase:
\begin{align*}
\phi_l(x,\delta,2^{-l}) = \phi_1(x_1, 2^{-l}x_2, \delta) = \tilde{\phi}(x_1, 2^{-l} x_2 + \delta_0 x_1^2 \psi(\delta_1 x_1), \delta).
\end{align*}
Now we need a relation between the Hessian determinant of $\phi_l$ and the Hessian determinant of $\tilde{\phi}$.
For this let us denote for simplicity
\begin{align*}
\varphi(x_1,\delta_1) \coloneqq \delta_1^2 x_1^2 \psi(\delta_1 x_1).
\end{align*}
The reason why we have not included the factor $\delta_2^{-1}$ will be clear later
(recall from \eqref{section_strichartz_mitigating_delta_def_ii_iii_iv_v} that $\delta_0 = \delta_1^2 \delta_2^{-1}$).
A direct calculation shows then
\begin{align}
\label{section_strichartz_mitigating_v_Hessian_nonlinear}
\mathcal{H}_{\phi_l} = 2^{-2l} \mathcal{H}_{\tilde{\phi}} + \delta_2^{-1} 2^{l} \partial_1^2 \varphi \partial_2 \phi_l \partial_2^2 \phi_l,
\end{align}
and due to our localization we have $|\mathcal{H}_{\tilde{\phi}}| \sim 2^{-l(k-2)}$.

\medskip

We use the same complex interpolation idea as in (iv) according to which it suffices to prove
\begin{align*}
\sum_{2^l \gg 1} 2^{-l(1/2-\sigma)(k-2)} |\widehat{\nu}_{l} (\xi)| \lesssim (1+|\xi_3|)^{-1},
\end{align*}
where after rescaling $x_2 \mapsto 2^{-l}x_2$ we have
\begin{align*}
2^{-l(1/2-\sigma)(k-2)} \widehat{\nu}_{l} (\xi)
   = 2^{-kl/2} \int e^{i \Phi_0(x, \xi, \delta, 2^{-l})} a(x,\delta,2^{-l}) \mathrm{d}x,
\end{align*}
where the phase function for the Fourier transform of $\nu_l$ is
\begin{align*}
\Phi_0 (x, \xi, \delta, 2^{-l})
  &\coloneqq  \xi_1 x_1 + \xi_2 \delta_0 x_1^2 \psi(\delta_1 x_1) + \xi_3 x_1^2 \tilde{r}_1(\delta_1 x_1, 2^{-l} \delta_2 x_2) \\
  &\quad + \xi_2 2^{-l} x_2 + \xi_3 2^{-kl} x_2^k \tilde{r}_2(\delta_1 x_1, 2^{-l} \delta_2 x_2) \\
  &= \xi_1 x_1 + \xi_2 \delta_2^{-1} \varphi(x_1, \delta_1) + \xi_2 2^{-l} x_2 + \xi_3 \phi_l(x, \delta, 2^{-l}).
\end{align*}
The amplitude localizes the integration to $|x_1| \sim 1 \sim |x_2|$.

\medskip

Using the same argumentation as in the case (iv) we can reduce ourselves to the case
when $|\xi_1|\sim|\xi_3|$, $|\xi_2| \ll |\xi_3|$, and $|\xi_3| 2^{-kl} \gg 1$ are satisfied.

Now let us make some further reductions using the fact that $\partial_2 \tilde{r}_1(0), \partial_2^2 \tilde{r}_1(0) \neq 0$.
The $x_2$ derivative of the phase $\Phi_0$ contains three terms of respective sizes:
$\sim |2^{-l} \delta_2 \xi_3|$, $\sim |2^{-l} \xi_2|$, and $\sim |2^{-kl} \xi_3|$.
If we may integrate by parts in $x_2$ (i.e., if one of the above terms is much larger than the other two),
we can get an admissible estimate and sum in $l$.
If $|2^{-kl} \xi_3|$ is comparable to the larger of the other two terms,
then one easily sees that the second derivative in $x_2$ is necessarily of size $|2^{-kl} \xi_3|$,
and so in this case we get by iterative stationary phase the estimate
\begin{align*}
2^{-l(1/2-\sigma)(k-2)} |\widehat{\nu}_{l} (\xi)| \lesssim (1+|\xi_3|)^{-1}.
\end{align*}
Note that we do not need to sum in $l$ since there are only finitely many $l$ satisfying one of the relations
$|2^{-kl} \xi_3| \sim |2^{-l} \delta_2 \xi_3|$ or $|2^{-kl} \xi_3| \sim |2^{-l} \xi_2|$.

\medskip

We are thus now reduced to the case when
\begin{align*}
|2^{-l} \xi_2| \sim |2^{-l} \delta_2 \xi_3| \gg |2^{-kl} \xi_3|, \qquad |\xi_1| \sim |\xi_3|, \qquad \text{and} \qquad |\xi_3| 2^{-kl} \gg 1.
\end{align*}
At this point we introduce some further notation:
\begin{align*}
\lambda \coloneqq \xi_3, \quad
s_1 \coloneqq \frac{\xi_1}{\xi_3}, \quad
s_2 \coloneqq \frac{\xi_2}{\delta_2 \xi_3}, \quad
\varepsilon \coloneqq 2^{-l} \delta_2,
\end{align*}
and so we have $|s_1| \sim 1 \sim |s_2|$, $\lambda 2^{-kl} \gg 1$, and $\varepsilon \gg 2^{-kl}$.
The phase $\Phi_0$ can now be rewritten as $\lambda \Phi$, where $\Phi$ is
\begin{align*}
\Phi (x, s_1,s_2, \delta_1, \varepsilon, 2^{-kl})
  &= s_1 x_1 + s_2 \delta_1^2 x_1^2 \psi(\delta_1 x_1) + s_2 \varepsilon x_2 + \phi_l(x, \delta, 2^{-l}),
\end{align*}
since we note from the form of $\phi_l$ that $\phi_l$
can also be taken to depend on $(x_1,x_2,\delta_1,\varepsilon,2^{-kl})$.

\medskip

Let us now apply the stationary phase method in $x_1$.
We may rewrite the phase as
\begin{align*}
\Phi (x, s_1,s_2, \delta_1, \varepsilon, 2^{-kl})
  = s_1 x_1 + s_2 \varphi + s_2 \varepsilon x_2 + \phi_l,
\end{align*}
where we recall that $\varphi(x_1, \delta) = \delta_1^2 x_1^2 \psi(\delta_1 x_1)$.
We may assume that there is a stationary point for the $x_1$ derivative
since $|\partial_1^2 \phi_l| \sim 1$ and $|s_1| \sim 1$,
and as otherwise we may use integration by parts.

We denote by $x_1^c = x_1^c(x_2,s_1,s_2,\delta_1,\varepsilon, 2^{-kl})$ the function such that
\begin{align}
\label{subsubsection_stationary_1_critical}
(\partial_1 \Phi) (x_1^c, x_2,s_1,s_2,\delta_1,\varepsilon, 2^{-kl})
  = s_1 + s_2 \partial_1 \varphi + \partial_1 \phi_l
  = 0.
\end{align}
Taking the $x_2$ derivative we get
\begin{align}
\label{subsubsection_stationary_1_critical_derivative}
s_2 (x_1^c)' \partial_1^2 \varphi + (x_1^c)' \partial_1^2 \phi_l + \partial_1 \partial_2 \phi_l = 0.
\end{align}
After applying the stationary phase method in $x_1$ we gain a decay factor of $\lambda^{-1/2}$, i.e.,
we have
\begin{align*}
2^{-l(1/2-\sigma)(k-2)} \widehat{\nu}_{l} (\xi)
   = \lambda^{-1/2} 2^{-kl/2} \int e^{-i \lambda \tilde{\Phi} (x_2, s_1,s_2,\delta_1,\varepsilon, 2^{-kl})} a(x_2,s_1,s_2,\delta,2^{-l}; \lambda) \mathrm{d}x_2,
\end{align*}
where the new phase is
\begin{align*}
\tilde{\Phi} (x_2, s_1,s_2,\delta_1,\varepsilon, 2^{-kl})
  = s_1 x_1^c + s_2 \varphi(x_1^c,\delta_1) + s_2 \varepsilon x_2 + \phi_l(x_1^c,x_2,\delta,2^{-l}),
\end{align*}
and the amplitude $a$ is a classical symbol in $\lambda$ of order $0$.

Taking the $x_2$ derivative of the expression for the new phase $\tilde{\Phi}$
and using the equation \eqref{subsubsection_stationary_1_critical} we get
\begin{align}
\label{subsubsection_stationary_1_phase_first_derivative}
\tilde{\Phi}'
  = s_2 \varepsilon + \partial_2 \phi_l.
\end{align}
Therefore, the second derivative of the new phase is
\begin{align}
\label{subsubsection_stationary_1_phase_second_derivative}
\tilde{\Phi}''
  &= (\partial_2 \phi_l)' \nonumber \\
  &= \partial_2^2 \phi_l + (x_1^c)' \partial_1 \partial_2 \phi_l.
\end{align}
Now using in order
\eqref{subsubsection_stationary_1_critical_derivative},
the definition of $\mathcal{H}_{\phi_l}$
\eqref{section_strichartz_mitigating_v_Hessian_nonlinear},
\eqref{subsubsection_stationary_1_phase_first_derivative},
and \eqref{subsubsection_stationary_1_phase_second_derivative},
we obtain
\begin{align*}
(\partial_1^2 \phi_l) \tilde{\Phi}''
  &= \partial_1^2 \phi_l \, \partial_2^2 \phi_l + \partial_1 \partial_2 \phi_l (- \partial_1 \partial_2 \phi_l - s_2 (x_1^c)' \partial_1^2 \varphi) \\
  &= \mathcal{H}_{\phi_l} - s_2 (x_1^c)' \partial_1^2 \varphi \, \partial_1 \partial_2 \phi_l \\
  &= 2^{-2l} \mathcal{H}_{\tilde{\phi}} + \delta_2^{-1} 2^{l} \partial_1^2 \varphi \partial_2 \phi_l \partial_2^2 \phi_l
     - s_2 (x_1^c)' \partial_1^2 \varphi \, \partial_1 \partial_2 \phi_l \\
  &= 2^{-2l} \mathcal{H}_{\tilde{\phi}} + \varepsilon^{-1} \partial_1^2 \varphi \, \partial_2^2 \phi_l ( \tilde{\Phi}' - \varepsilon s_2)
     - s_2 (x_1^c)' \partial_1^2 \varphi \, \partial_1 \partial_2 \phi_l \\
  &= 2^{-2l} \mathcal{H}_{\tilde{\phi}} - s_2 \partial_1^2 \varphi \, \partial_2^2 \phi_l
     - s_2 (x_1^c)' \partial_1^2 \varphi \, \partial_1 \partial_2 \phi_l + \varepsilon^{-1} \partial_1^2 \varphi \, \partial_2^2 \phi_l \,\tilde{\Phi}' \\
  &= 2^{-2l} \mathcal{H}_{\tilde{\phi}} - s_2 \partial_1^2 \varphi \, \tilde{\Phi}'' + \varepsilon^{-1} \partial_1^2 \varphi \, \partial_2^2 \phi_l \,\tilde{\Phi}'.
\end{align*}
Thus, we get
\begin{align}
\label{subsubsection_stationary_1_phase_ODE}
(s_2 \partial_1^2 \varphi + \partial_1^2 \phi_l) \tilde{\Phi}''
  &= 2^{-2l} \mathcal{H}_{\tilde{\phi}} + \varepsilon^{-1} \partial_1^2 \varphi \, \partial_2^2 \phi_l \,\tilde{\Phi}'.
\end{align}
Note that we have $|\varepsilon^{-1} \partial_1^2 \varphi \, \partial_2^2 \phi_l| \ll \delta_1^2 \ll 1$
and $|s_2 \partial_1^2 \varphi + \partial_1^2 \phi_l| \sim 1$,
and recall that $|2^{-2l} \mathcal{H}_{\tilde{\phi}}| \sim 2^{-kl}$.
We claim that either
$|\tilde{\Phi}'| \lesssim 2^{-kl}$ on the whole domain of integration (i.e., for $|x_2| \sim 1$), or that
$|\tilde{\Phi}'| \gtrsim 2^{-kl}$ on the whole domain of integration.
This can be shown
by using the formula for the solution of a linear first order ODE (considering $\tilde{\Phi}'$ as the unknown),
or by arguing by contradiction.

Let us argue by contradiction in the following way.
Let us assume that there exists a point $|x_2^0| \sim 1$ such that $|\tilde{\Phi}'(x_2^0)| \leq 2^{-kl}$.
Furthermore, let us assume that there exists a point $|x_2^1| \sim 1$
where one has $|\tilde{\Phi}'| = C_1 2^{-kl}$ for some sufficiently large $C_1$,
and let us assume that $x_2^1$ is the closest point to $x_2^0$ satisfying this condition
in the sense that $|\tilde{\Phi}'| < C_1 2^{-kl}$ between $x_2^0$ and $x_2^1$.
Then the mean value theorem implies that there is a point between $x_2^0$ and $x_2^1$ where we have $|\tilde{\Phi}''| \geq C_2 2^{-kl}$,
where $C_2$ can be taken to tend to $\infty$ as $C_1$ tends to $\infty$.
On the other hand, the equation \eqref{subsubsection_stationary_1_phase_ODE} implies that
on the interval between $x_2^0$ and $x_2^1$ we have $|\tilde{\Phi}''| \leq C_3 2^{-kl}$,
where we can take $C_3$ to be a fixed constant if $\delta_1$ is taken to be sufficiently small when $C_1$ and $C_2$ are large
(we can always take say $C_1$ of size $\delta_1^{-1}$).
This is a contradiction, i.e., the point $x_2^1$ where one has $|\tilde{\Phi}'| \geq C_1 2^{-kl}$ for a too large $C_1$
cannot exist within the integration domain.

Now in the case that $|\tilde{\Phi}'| \gtrsim 2^{-kl}$ we may apply integration by parts and get an estimate summable in $l$.
Let us therefore assume $|\tilde{\Phi}'| \lesssim 2^{-kl}$,
in which case we have $|\tilde{\Phi}''| \sim 2^{-kl}$ by \eqref{subsubsection_stationary_1_phase_ODE}.
Then the van der Corput lemma implies that
\begin{align*}
2^{-l(1/2-\sigma)(k-2)} |\widehat{\nu}_{l} (\xi)| \lesssim (1+|\xi_3|)^{-1}.
\end{align*}

The problem is now that a priori we may not sum this estimate in $l$.
Luckily, it turns out that one can pin down the size of $2^{-l}$,
which in turn will pin down the number $l$ to a finite set of size $\mathcal{O}(1)$.
In order to prove this we use the expression \eqref{subsubsection_stationary_1_phase_first_derivative}
and the normal form of $\phi_l$:
\begin{align*}
\phi_l(x,\delta,2^{-l}) =
x_1^2 \tilde{r}_1(\delta_1 x_1, \varepsilon x_2) + 2^{-kl} x_2^k \tilde{r}_2(\delta_1 x_1, \varepsilon x_2),
\end{align*}
from which one has
\begin{align}
\label{subsubsection_eq_partial_2_phi_l}
(\partial_2 \phi_l)(x,\delta,2^{-l}) =
  \varepsilon x_1^2  (\partial_2 \tilde{r}_1)(\delta_1 x_1, \varepsilon x_2)
+ 2^{-kl} x_2^{k-1} \tilde{r}_3(\delta_1 x_1, \varepsilon x_2),
\end{align}
where $\tilde{r}_3(0) \neq 0$ is a smooth function.

The idea is as follows.
First, by compactness we may assume that we integrate in $x_2$
over a sufficiently small neighbourhood of a point $x_2^0$ satisfying $|x_2^0| \sim 1$.
In particular, we may write
\begin{align*}
\tilde{\Phi}'(x_2, s_1,s_2,\delta_1,\varepsilon, 2^{-kl})
  &= \tilde{\Phi}'(x_2^0, s_1,s_2,\delta_1,\varepsilon, 2^{-kl}) + \smallO (|\tilde{\Phi}''|) \\
  &= \tilde{\Phi}'(x_2^0, s_1,s_2,\delta_1,\varepsilon, 2^{-kl}) + \smallO (2^{-kl}).
\end{align*}
Thus, it suffices to prove that
\begin{align*}
|\tilde{\Phi}'(x_2^0, s_1,s_2,\delta_1,\varepsilon, 2^{-kl})| = |s_2 \varepsilon + \partial_2 \phi_l(x_1^c, x_2^0, s_1,s_2,\delta_1,\varepsilon, 2^{-kl})| \ll 2^{-kl}
\end{align*}
can happen only for finitely many $l$.
If the above inequality does not hold, then we may simply integrate by parts and are able to simply sum in $l$ afterwards.

If we now develop both terms in $\partial_2 \phi_l$ in the $\varepsilon$ and $2^{-kl}$ variables
(recall that $x_1^c$ depends on both $\varepsilon$ and $2^{-kl}$),
then one gets that the expression for $\tilde{\Phi}'$ is of the form
\begin{align*}
s_2 \varepsilon
+ \sum_{i = 1}^{k-1} \varepsilon^i f_i(x_2^0, s_1, s_2, \delta_1)
+ 2^{-kl} g_0(x_2^0, s_1, s_2, \delta_1)
+ \smallO (2^{-kl}),
\end{align*}
where we used the fact that $\varepsilon^k = (\delta_2 2^{-l})^k \ll 2^{-kl}$.
Note that we have $|g_0| \sim 1$ by \eqref{subsubsection_eq_partial_2_phi_l}
(and also $|f_1| \sim 1$, but this is not important).
We have to find out how many $l$'s satisfy
\begin{align*}
\Big| \tilde{f}_1(x_2^0, s_1, s_2, \delta_1)
+ \sum_{i = 2}^{k-1} \varepsilon^{i-1} f_i(x_2^0, s_1, s_2, \delta_1)
+ \varepsilon^{-1} 2^{-kl} g_0(x_2^0, s_1, s_2, \delta_1)
+ \smallO (\varepsilon^{-1} 2^{-kl}) \Big|
\ll \varepsilon^{-1} 2^{-kl},
\end{align*}
where $\tilde{f}_1(x_2^0, s_1, s_2, \delta_1) \coloneqq s_2 + f_1(x_2^0, s_1, s_2, \delta_1)$.
But now one easily shows that this inequality is possible
only if at least two of the terms are comparable in size (precisely because $|g_0| \sim 1$).
This implies in particular that we can determine $l$ in terms of $(x_2^0, s_1, s_2, \delta_1)$,
which finishes the proof.

We mention that, interestingly, one can prove that $f_2(x_2^0,s_1,s_2,0) = 0$,
a consequence of the relation \eqref{section_strichartz_mitigating_v_coef_relation}.


\subsection{Normal form (vi)}
\label{subsection_strichartz_mitigating_vi}
Here we obtain
\begin{align*}
\tilde{\phi}(x,\delta) &= (x_2 - x_1^2 \psi(\delta_1 x_1))^{k} r(\delta_1 x_1, \delta_2 x_2), \\
\mathcal{H}_{\tilde{\phi}}(x, \delta) &= (x_2 - x_1^2 \psi(\delta_1 x_1))^{2k-3} q(\delta_1 x_1, \delta_2 x_2),
\end{align*}
where $r(0), q(0), \psi(0) \neq 0$.
Thus, we may localize to the part where $|x_2 - x_1^2 \psi(\delta_1 x_1)| \ll 1$, i.e.,
it is sufficient to consider the measure
\begin{align*}
f \mapsto \int f(x,\tilde{\phi}(x,\delta)) \, |(x_2 - x_1^2 \psi(\delta_1 x_1))^{2k-3}q(\delta_1 x_1, \delta_2 x_2)|^{\sigma} \, \chi_0(\tilde{\phi}(x,\delta)) \, a(x, \delta) \mathrm{d}x
\end{align*}
since $|\tilde{\phi}(x,\delta)| \sim |x_2 - x_1^2 \psi(\delta_1 x_1)|^k$.
Note that here we have $|x_1| \sim 1 \sim |x_2|$.

Now, the next idea is to use, as in \cite{IM16}, a Littlewood-Paley decomposition in the $x_3$ direction
(for the mixed norm Littlewood-Paley theory see \cite{Liz70})
and reduce ourselves to proving the Fourier restriction estimate for the measure piece
\begin{align*}
\langle \nu_{l}, f \rangle
  = \int f(x,\tilde{\phi}(x,\delta)) \, |(x_2 - x_1^2 \psi(\delta_1 x_1))^{2k-3}q(\delta_1 x_1, \delta_2 x_2)|^{\sigma} \,
    \chi_1(2^{kl}(\tilde{\phi}(x,\delta))) \, a(x, \delta) \mathrm{d}x.
\end{align*}
Using the coordinate transformation $x_2 \mapsto x_2 + x_1^2 \psi(\delta_1 x_1)$ we may write
\begin{align*}
\langle \nu_{l}, f \rangle
  = \int &f(x_1, x_2 + x_1^2 \psi(\delta_1 x_1), x_2^k \tilde{r}(\delta_1 x_1, \delta_2 x_2)) \\
    &\times |x_2^{2k-3} \tilde{q}(\delta_1 x_1, \delta_2 x_2)|^{\sigma} \, \chi_1(2^{kl}x_2^k \tilde{r}(\delta_1 x_1, \delta_2 x_2)) \, \tilde{a}(x, \delta) \mathrm{d}x,
\end{align*}
where $|\tilde{r}| \sim 1$ is a smooth function.
Finally, we use the coordinate transformation $x_2 \mapsto 2^{-l} x_2$ and rescale $f$ in the third coordinate.
Then we are reduced to proving the Fourier restriction estimate
\begin{align}
\label{subsection_strichartz_mitigating_vi_FRP}
\Vert \widehat{f} \Vert^2_{L^2(\mathrm{d}\tilde{\nu}_{j,l})}
   \leq C 2^{l(1-3\sigma)} \Vert f\Vert^2_{L^{\mathfrak{p}_3}_{x_3}(L^{\mathfrak{p}_1}_{(x_1,x_2)})},
\end{align}
for the measure
\begin{align}
\label{subsection_strichartz_mitigating_vi_final_measure}
\langle \tilde{\nu}_{j,l}, f \rangle
  = \int &f(x_1, 2^{-l}x_2 + x_1^2 \psi(\delta_1 x_1), x_2^k \tilde{r}(\delta_1 x_1, 2^{-l}\delta_2 x_2)) a(x, \delta, 2^{-l}) \mathrm{d}x
\end{align}
where $a$ is supported so that $|x_1| \sim 1$ and $|x_2| \sim 1$.
Now we note that the estimate for $\sigma = 0$ follows by Plancherel,
while the estimate for $\sigma = 1/3$ is going to be shown in Section \ref{section_strichartz_mixed}
since the form of the measure $\tilde{\nu}_{j,l}$ coincides with the form in \eqref{section_strichartz_mixed_measure_rescaled} below.
Interpolating, we obtain the estimate for all $0 \leq \sigma \leq 1/3$.

Note that when $1/p_1' = 1/p_3' = 1/4$,
then one can simplify the proof by a modification of Lemma \ref{lemma_simple_estimates}, i.e.,
by using the Fourier decay of $\tilde{\nu}_{j,l}$,
which is easily seen to be
\begin{align*}
|\widehat{\tilde{\nu}}_{j,l}(\xi)| \lesssim 2^{l/2} (1+|\xi|)^{-1},
\end{align*}
and by using the Plancherel theorem, but this time in the $(x_1,x_3)$-plane
(which is why it works only for $1/p_1' = 1/p_3'$)
since the mapping $(x_1,x_2) \mapsto (x_1, x_2^k \tilde{r}(\delta_1 x_1, 2^{-l}\delta_2 x_2))$ has Jacobian of size $\sim 1$.
In fact, in Section \ref{section_strichartz_mixed} we shall combine
this idea of using Lemma \ref{lemma_simple_estimates} together with the methods used in \cite{IM16} (and \cite{Pa19}).

\subsubsection{A Knapp-type example}
\label{subsubsection_strichartz_mitigating_vi_Knapp}

Let us now show by using a Knapp-type example that
one cannot get the estimate \eqref{subsection_strichartz_mitigating_vi_FRP} for $\sigma > 1/3$.
Let us consider the function $\varphi_\epsilon$ defined by
\begin{align*}
\widehat{\varphi_\epsilon}(x)
  = \chi_1 \Big(\frac{x_1}{\epsilon^\delta}\Big) \chi_1 \Big(\frac{x_2}{\epsilon^{2\delta}}\Big) \chi_1 \Big(\frac{x_3}{\epsilon} \Big)
\end{align*}
for some small $\epsilon$ and $\delta$.
Its mixed $L^p$ norm is
\begin{align*}
\Vert \varphi_\epsilon\Vert_{L^p(\R^3)} \sim \epsilon^{\frac{3\delta}{p_1'} + \frac{1}{p_3'}}.
\end{align*}
Now, in the integral
\begin{align*}
\int |\widehat{\varphi_\epsilon}|^2 \mathrm{d}\nu
  = \int |\widehat{\varphi_\epsilon}|^2(x, \phi_{\text{loc}}(x)) a(x) |\mathcal{H}_{\phi_\text{loc}}(x)|^\sigma \mathrm{d}x
\end{align*}
we integrate over the set
\begin{align*}
D_\epsilon^0
\coloneqq \{ x \in \R^2 :
|x_1| \lesssim \epsilon^\delta, \quad
|x_2| \lesssim \epsilon^{2\delta}, \quad
|\phi_{\text{loc}}(x)| \sim |x_2-x_1^2 \psi(x_1)|^k \lesssim \epsilon \}
\end{align*}
by definition of $\varphi_\epsilon$.
If $\delta$ is sufficiently small, $D_\epsilon^0$ contains the set
\begin{align*}
 D_\epsilon \coloneqq \{ x \in \R^2 :
|x_1| \lesssim \epsilon^\delta, \quad
|\phi_{\text{loc}}(x)| \sim |x_2-x_1^2 \psi(x_1)| \lesssim \epsilon^{1/k} \},
\end{align*}
and so if the Fourier restriction estimate holds, one has
\begin{align*}
\epsilon^{\frac{6\delta}{p_1'} + \frac{2}{p_3'}} \sim \Vert \varphi_\epsilon \Vert^2_{L^p(\R^3)} \gtrsim
\int |\widehat{\varphi_\epsilon}|^2 \mathrm{d}\nu
  &\gtrsim \int_{D_\epsilon} |x_2-x_1^2 \psi(x_1)|^{\sigma(2k-3)} \mathrm{d}x \\
  &\sim \epsilon^\delta \int_{|y| \lesssim \epsilon^{1/k}} |y|^{\sigma(2k-3)} \mathrm{d}y \\
  &\sim \epsilon^{\delta+(\sigma(2k-3)+1)/k}.
\end{align*}
Letting $\epsilon$ and then $\delta$ tend to $0$ we obtain the condition
\begin{align*}
\frac{1}{p_3'} \leq \frac{\sigma(2k-3)+1}{2k} = \sigma + \frac{1-3\sigma}{2k}.
\end{align*}
Since we are interested in $1/p_1' = 1/2-\sigma$ and $1/p_3' = \sigma$,
the above inequality implies precisely $\sigma \leq 1/3$.



\section{Fourier restriction without a mitigating factor}
\label{section_strichartz_mixed}

Here we prove Theorem \ref{theorem_Strichartz}, i.e., the estimate
\begin{align*}
\Vert \widehat{f} \Vert_{L^2(\mathrm{d}\mu)} \leq C \Vert f\Vert_{L^{p_3}_{x_3}(L^{p_1}_{(x_1,x_2)})}
\end{align*}
for $\mu$ the surface measure of the form
\begin{align*}
\langle \mu, f \rangle = \int_{\R^2\setminus\{0\}} f(x,\phi(x)) \, |x|_{\alpha}^{2\vartheta} \, \mathrm{d}x,
\end{align*}
where
\begin{align*}
\vartheta = \frac{|\alpha|}{p_1'}+\frac{\rho}{p_3'} - \frac{|\alpha|}{2}.
\end{align*}
Recall that this $\vartheta$ is chosen (depending on $(p_1,p_3) \in (1,2]^2$) precisely so that
the above restriction estimate is equivalent to the local estimate
\begin{align*}
\Vert \widehat{f} \Vert_{L^2(\mathrm{d}\mu_0)} \leq C \Vert f\Vert_{L^{p_3}_{x_3}(L^{p_1}_{(x_1,x_2)})},
\end{align*}
where $\mu_0$ is the surface measure
\begin{align}
\label{section_strichartz_mixed_original_measure_compact}
\langle \mu_0, f \rangle = \int_{\R^2\setminus\{0\}} f(x,\phi(x)) \, \eta(x) \, |x|_{\alpha}^{2\vartheta} \, \mathrm{d}x
\end{align}
for $\eta \in C_c^\infty(\R^2\setminus\{0\})$ identically equal to $1$ in an annulus.

Note that $|x|_{\alpha}^{2\vartheta}$ is not smooth near the axes.
Luckily, we shall be able to circumvent this problem by using the Littlewood-Paley theorem to localize away from the axes,
as was done in the case with the mitigating factor.

Now we recall the necessary conditions from \cite[Proposition 2.1]{Pa19} obtained through the Knapp-type examples.
Let us fix a point $v$ such that $\eta(v) \neq 0$ and let $\eta_v$ be a smooth cutoff function
identically equal to $\eta$ on a small neighbourhood of $v$.
It suffices to consider the measure
\begin{align}
\label{section_strichartz_mixed_original_measure_loc}
\langle \mu_{0,v}, f \rangle = \int_{\R^2\setminus\{0\}} f(x,\phi_v(x-v)) \, \eta_v(x) \, |x|_{\alpha}^{2\vartheta} \, \mathrm{d}x,
\end{align}
where we recall from the introduction that
\begin{align*}
\phi_v(x) = \phi(x+v) - \phi(v) - x \cdot \nabla \phi(v).
\end{align*}
We recall also that $h_{\text{lin}}(\phi,v)$ is the linear height of $\phi_v$ at its origin,
and that $h(\phi,v)$ is its Newton height.

If $\phi$ satisfies (LA) at $v$, then $h_{\text{lin}}(\phi,v) = h(\phi,v)$,
and according to \cite[Proposition 2.1]{Pa19} the only necessary condition is
\begin{align}
\label{section_strichartz_mixed_Knapp_LA}
\frac{1}{p_1'} + \frac{h_{\text{lin}}(\phi,v)}{p_3'} \leq \frac{1}{2}.
\end{align}
If $\phi$ does not satisfy (LA) at $v$, then from Proposition \ref{proposition_normal_forms} we know that
this is only possible for the normal form
\begin{align*}
\phi_{v,y}(y) \coloneqq (y_2 - y_1^2 \psi(y_1))^k r(y),
\end{align*}
where $r(0) \neq 0$, $\psi(0) \neq 0$, and $2 \leq k < \infty$,
since this is the only non-adapted normal form.
Again, \cite[Proposition 2.1]{Pa19} implies that in this case we have two necessary conditions, namely
\begin{align}
\label{section_strichartz_mixed_Knapp_NLA}
\frac{1}{p_1'} + \frac{h_{\text{lin}}(\phi,v)}{p_3'} \leq \frac{1}{2} \qquad \text{and} \qquad
\frac{h(\phi,v)}{p_3'} \leq \frac{1}{2},
\end{align}
where $h(\phi,v) = k$ and $h_{\text{lin}}(\phi,v) = 2k/3$.
Note that in the case $h_{\text{lin}}(\phi,v) = h(\phi,v)$ the second condition in \eqref{section_strichartz_mixed_Knapp_NLA} would be redundant.
Thus, if we now vary $v$ over the points where $\eta(v) \neq 0$, then we obtain the conditions
\begin{align*}
\frac{1}{p_1'} + \frac{h_{\text{lin}}(\phi)}{p_3'} \leq \frac{1}{2} \qquad \text{and} \qquad
\frac{h(\phi)}{p_3'} \leq \frac{1}{2},
\end{align*}
where we remind that $h_{\text{lin}}(\phi)$ and $h(\phi)$ are respectively global linear height and global Newton height
defined as in \eqref{subsection_intro_results_global_heights}.

At all points $v$ where (LA) is satisfied and where $|x|_{\alpha}^{2\vartheta}$ is smooth (i.e., $v$ is not on an axis)
we get the local Fourier restriction estimate in the range \eqref{section_strichartz_mixed_Knapp_LA}
directly from \cite[Proposition 4.2]{Pa19}.
We shall briefly touch upon what happens in the case when $v$ is situated on the axis in Subsection \ref{subsection_strichartz_mixed_non_smooth}.
In this case one has to only slightly adjust the proofs in Section \ref{section_strichartz_mitigating}.

In the case when (LA) is not satisfied at $v$ let us call the pair $(\mathfrak{p}_1, \mathfrak{p}_3) = (\mathfrak{p}_1(v), \mathfrak{p}_3(v))$ given by
\begin{align*}
\Big( \frac{1}{\mathfrak{p}_1'}, \frac{1}{\mathfrak{p}_3'} \Big) = \Big( \frac{1}{2} - \frac{h_{\text{lin}}(\phi,v)}{2 h(\phi,v)}, \frac{1}{2 h(\phi,v)} \Big)
\end{align*}
the critical exponent of $\phi$ at $v$.
It is obtained as the intersection of the lines
\begin{align*}
\frac{1}{p_1'} + \frac{h_{\text{lin}}(\phi,v)}{p_3'} &= \frac{1}{2} \qquad\qquad \text{and} \qquad\qquad
\frac{1}{p_3'} = \frac{1}{2 h(\phi,v)}
\end{align*}
in the $(1/p_1',1/p_3')$ plane.
Thus, for the local estimate in this case it suffices to prove the inequality
\begin{align*}
\Vert \widehat{f} \Vert_{L^2(\mathrm{d}\mu_{0,v})} \leq C \Vert f\Vert_{L^{p_3}_{x_3}(L^{p_1}_{(x_1,x_2)})},
\end{align*}
where
\begin{align*}
\langle \mu_{0,v}, f \rangle = \int f(x_1,x_2,\phi_v(x_1,x_2)) \, \eta_v(x_1,x_2) \, \mathrm{d}x_1 \mathrm{d}x_2
\end{align*}
and
\begin{align*}
\Big( \frac{1}{p_1'}, \frac{1}{p_3'}) \in
   \Bigg\{ \Big( 0, \frac{1}{2 h(\phi)} \Big),
           \Big( \frac{1}{2}, 0\Big),
           \Big( \frac{1}{\mathfrak{p}_1'(v)}, \frac{1}{\mathfrak{p}_3'(v)} \Big) \Bigg\},
\end{align*}
since then we get the full range from the necessary conditions by interpolation.
We shall only give a sketch of the proof in this case too
in Subsections \ref{subsection_strichartz_mixed_FRP_prelim} and \ref{subsection_strichartz_mixed_FRP_spectral},
since it is almost identical to a type of singularity considered in \cite[Subsection 5.5]{Pa19}.


\subsection{Fourier restriction for the adapted case}
\label{subsection_strichartz_mixed_non_smooth}

As mentioned, in the adapted case one needs to prove the Fourier restriction estimate for
$(p_1,p_3) \in (1,2)^2$ satisfying
\begin{align*}
\frac{1}{p_1'} + \frac{h_{\text{lin}}(\phi,v)}{p_3'} = \frac{1}{2},
\end{align*}
and the part of the measure where the amplitude in \eqref{section_strichartz_mixed_original_measure_compact} is smooth
the restriction estimate is already proven in \cite{Pa19}.

Now the amplitude in \eqref{section_strichartz_mixed_original_measure_compact}
(in particular the function $x \mapsto |x|_\alpha^{2\vartheta}$) is
in general not smooth along the axes $x_1=0$ and $x_2=0$.
Namely, on the $x_1=0$ axis one can take only the derivatives (of the amplitude) in the $x_2$ direction, and analogously
on the $x_2=0$ axis one can take only derivatives in the $x_1$ direction.
Note that the only possible non-adapted normal form appears only away from the axes.

Let us consider without loss of generality what happens for the point $v = (v_1,0)$ on the axis $x_2 = 0$
and its associated measure $\mu_{0,v}$ defined in \eqref{section_strichartz_mixed_original_measure_loc}.
We shall only briefly sketch what one needs to do in order to prove the Fourier restriction estimate when
the amplitude is not smooth in the $x_2$ direction at $v$.
Since we are dealing only with adapted normal forms,
it suffices to obtain an appropriate estimate on the Fourier transform,
after which one can apply Lemma \ref{lemma_simple_estimates}
or its modification such as \cite[Lemma 3.8]{Pa19}.
Often we shall need to use the Littlewood-Paley theorem in order to localize away from the axis.

According to the normal forms listed at the end of Subsection \ref{subsection_forms_considerations},
and under the condition (H1), we have the following cases.

\medskip

{\bf Case 1.}
If (under the notation of Section \ref{section_forms}) we have $k = \infty$, then
by the considerations from Subsection \ref{subsection_forms_flat} the phase at $v$ is
\begin{align*}
\phi_v(x-v) = (x_1-v_1)^{\tilde{k}} + \varphi(x_1,x_2),
\end{align*}
where $2 \leq \tilde{k} < \infty$ and $\varphi$ is a flat function at $v$.
This corresponds to Normal form (i.y2) and we have $h_{\text{lin}}(\phi,v) = \tilde{k}$.
Since $|x|_\alpha^{2\vartheta}$ is still smooth in the $x_1$ direction,
one can use the van der Corput lemma in the $x_1$ direction and
get that the decay of the Fourier transform of $\mu_{0,v}$ is $(1+|\xi|)^{-1/\tilde{k}}$.
This now implies the desired estimate
(see the result \cite[Theorem 1.2]{KT98} or the results in \cite[Subsection 3.3]{Pa19},
or apply an appropriate modification of Lemma \ref{lemma_simple_estimates}).

\medskip

If $2 \leq k < \infty$, then we have three further cases.

{\bf Case 2.}
Let us consider the phase
\begin{align*}
\phi_v(x) = x_2^k r(x),
\end{align*}
where $r(v) \neq 0$ and $k \geq 2$. In this case the linear height is $h_{\text{lin}}(\phi,v) = k$.
Here the idea is to apply the Littlewood-Paley theorem in order to localize away from the axis $x_2 = 0$, and rescale afterwards.
Since the essentially same thing was done in Section \ref{section_strichartz_mitigating} for this type of singularity
(see the proof for Normal form (i) in Subsection \ref{subsection_strichartz_mitigating_i}),
let us just briefly mention the main differences compared to there.
Obviously, one scales differently the measure pieces away from the axis obtained by applying the Littlewood-Paley theorem
since here we consider different exponents $(p_1,p_3)$.
The main difference is that we do not use the Hessian determinant to obtain a decay on the Fourier transformation of the rescaled measure piece
(since the Hessian determinant may vanish of infinite order as only (H1) is assumed and not the stronger condition (H2)),
but rather directly from the form of the phase above.
This we may now do since the new amplitude for the rescaled measure pieces is now smooth.

{\bf Case 3.}
Let us now consider the case when the phase is nondegenerate, i.e.,
the Hessian determinant does not vanish at $v$ (and in particular $h_{\text{lin}}(\phi,v)$ = 1).
Here we use the Littlewood-Paley theorem as in Case 2,
but after rescaling we use the size of the Hessian determinant of the new phase
to get a decay on the Fourier transform of the measure
(as was done in Section \ref{section_strichartz_mitigating} for Normal forms (i), (ii), and (iii)).

{\bf Case 4.}
The final case is when
(after an affine change to $y$ or $w$ coordinates from Section \ref{section_forms}) we have
\begin{align*}
\phi_{v,u}(u) = u_1^2 r_1(u) + u_2^{k_0} r_2(u),
\end{align*}
where $3 \leq k_0 \leq \infty$, $r_1(0) \neq 0$, and in case when $k_0 < \infty$
then $r_2(0) \neq 0$ and $h_{\text{lin}}(\phi,v) = 2 k_0 / (2+k_0)$.
If $k_0 = \infty$ then $h_{\text{lin}}(\phi,v) = 2$, and
the above equality holds in the sense that we can take any $k_0 \geq 0$ and $r_2$ flat.
Inspecting the $y$ and $w$ coordinates from Section \ref{section_forms} we see that the $x_2=0$ axis corresponds to the $u_2=0$ axis.

If $k_0 = \infty$, we can argue in the same way as in the case $k = \infty$ above
(here it is critical that $\partial_{u_1} = c \partial_{x_1}, c \neq 0$, in order to be able to apply the van der Corput lemma in the smooth direction).

Otherwise, if $k_0$ is finite, we proceed again with a Littlewood-Paley decomposition in the $u_2$ direction
(as was done in Subsection \ref{subsection_strichartz_mitigating_ii} for Normal forms (ii) and (iii))
in order to get a smooth amplitude.
At this point one gets that the estimate on the decay of the Fourier transform is $2^{k_0 l/2}(1+|\xi|)^{-1}$
by using the size of the Hessian determinant.
Since the new rescaled phase is (compare with \eqref{subsection_strichartz_mitigating_ii_scaling_phase})
\begin{align*}
u_1^2 r_1(u_1, 2^{-l}u_2) + 2^{-k_0 l} u_2^{k_0} r_2(u_1, 2^{-l}u_2),
\end{align*}
by applying the van der Corput lemma in $u_1$ we also have the decay estimate $(1+|\xi|)^{-1/2}$.
Interpolating these two estimates gives the decay $2^l (1+|\xi|)^{-(2+k_0)/(2 k_0)}$,
which turns out to be precisely what one needs when interpolating with the Plancherel estimate.

\subsection{Fourier restriction for the non-adapted case: preliminaries}
\label{subsection_strichartz_mixed_FRP_prelim}

Let us fix a phase function $\phi_\text{loc}$ of the form
\begin{align*}
\phi_\text{loc}(x) = (x_2 - x_1^2 \psi(x_1))^k r(x),
\end{align*}
where $\psi(0), r(0) \neq 0$ and $k \in \N$, $k \geq 2$.
The adapted coordinates are obtained by the smooth transformation $y_1 = x_1$, $y_2 = x_2 - x_1^2 \psi(x_1)$:
\begin{align*}
\phi_\text{loc}^a(y) \coloneq y_2^k r^a(y),
\end{align*}
where $r^a(0) \neq 0$.
Thus, the Newton height of $\phi_\text{loc}$ is $k$ and the Newton distance is $d \coloneq 2 k /3$
(which coincides with the linear height $h_{\text{lin}}$).
The Varchenko exponent is $0$ since in adapted coordinates the principal face is noncompact.
Then from e.g. \cite[Subsection 3.3]{Pa19} we know that we automatically have the Fourier restriction estimate
\begin{align}
\label{section_strichartz_mixed_FRP}
\Vert \FT f \Vert_{L^2(\mathrm{d}\nu)} \lesssim \Vert f \Vert_{L^{p_3}_{x_3}(L^{p_1}_{(x_1,x_2)})}, \qquad f \in \mathcal{S}(\R^3),
\end{align}
for the exponents
\begin{align*}
\Big(\frac{1}{p_1'}, \frac{1}{p_3'}\Big) = \Big(0, \frac{1}{2 k}\Big) \qquad \text{and} \qquad
\Big(\frac{1}{p_1'}, \frac{1}{p_3'}\Big) = \Big(\frac{1}{2}, 0\Big),
\end{align*}
and where the measure $\nu$ is defined through
\begin{align}
\label{section_strichartz_mixed_measure}
\langle \nu, f \rangle = \int f(x_1,x_2,\phi_\text{loc}(x_1,x_2)) \, a(x_1,x_2) \, \mathrm{d}x_1 \mathrm{d}x_2,
\end{align}
where $a \in C_c^\infty(\R^2)$ is a nonnegative function supported in a small neighbourhood of the origin.
It remains to obtain the Fourier restriction estimate for the critical exponent, which in this case is
\begin{align}
\label{section_strichartz_mixed_critical_exponent}
\Big(\frac{1}{\mathfrak{p}_1'}, \frac{1}{\mathfrak{p}_3'}\Big) = \Big(\frac{1}{6}, \frac{1}{2 k}\Big).
\end{align}
The case $k = 2$ has been solved in \cite{Pa19}.
In the case $k = 3$ the critical exponent lies on the diagonal and
so this case has already been solved in \cite{IM16}.

In the case $k \geq 4$ we have $1/\mathfrak{p}_1' > 1/\mathfrak{p}_3'$ and
so one would need to slightly modify the methods used in \cite{Pa19}
(i.e., the methods for the case $h_{\text{lin}}(\phi) < 2$)
since there one interpolated between two points of the form
\begin{align*}
\Big(\frac{1}{p_1'}, \frac{1}{p_3'}\Big) = \Big(0, \frac{s}{2}\Big) \qquad \text{and} \qquad
\Big(\frac{1}{p_1'}, \frac{1}{p_3'}\Big) = \Big(\frac{1}{2}, \frac{1}{2}\Big),
\end{align*}
for some $0 < s < 1/k$.
In the case $1/\mathfrak{p}_1' > 1/\mathfrak{p}_3'$ in general one would need to interpolate between three points
\begin{align*}
\Big(\frac{1}{p_1'}, \frac{1}{p_3'}\Big) = \Big(0, 0\Big), \qquad
\Big(\frac{1}{p_1'}, \frac{1}{p_3'}\Big) = \Big(\frac{1}{2}, \frac{1}{2} \Big), \qquad \text{and} \qquad
\Big(\frac{1}{p_1'}, \frac{1}{p_3'}\Big) = \Big(\frac{1}{2}, 0 \Big).
\end{align*}
In  particular, if one has an operator $T : L^p \to L^{p'}$,
where here we denote $L^p = L^{p_3}_{x_3}(L^{p_1}_{(x_1,x_2)})$ for $p = (p_1, p_3)$,
satisfying the estimates
\begin{align}
\label{section_strichartz_mixed_basic_estimates}
\Vert T \Vert_{L^p \to L^{p'}} &\lesssim A_{1} \qquad \text{for} \quad \Big(\frac{1}{p_1'}, \frac{1}{p_3'}\Big) = \Big(0, 0\Big), \nonumber \\
\Vert T \Vert_{L^p \to L^{p'}} &\lesssim A_{2} \qquad \text{for} \quad \Big(\frac{1}{p_1'}, \frac{1}{p_3'}\Big) = \Big(\frac{1}{2}, \frac{1}{2}\Big), \\
\Vert T \Vert_{L^p \to L^{p'}} &\lesssim A_{3} \qquad \text{for} \quad \Big(\frac{1}{p_1'}, \frac{1}{p_3'}\Big) = \Big(\frac{1}{2}, 0\Big), \nonumber
\end{align}
then by interpolation one has the estimate
\begin{align*}
\Vert T \Vert_{L^p \to L^{p'}} &\lesssim A_{1}^{2/3} \, A_{2}^{1/k} \, A_{3}^{(k-3)/(3 k)}
\qquad \text{for} \quad \Big(\frac{1}{p_1'}, \frac{1}{p_3'}\Big) = \Big(\frac{1}{6}, \frac{1}{2 k}\Big).
\end{align*}

In our special case we shall not use the above general approach
since we recall that when we considered the case when the mitigating factor was present
(to be more precise, the case of Normal form (vi) considered in Subsection \ref{subsection_strichartz_mitigating_vi}),
after performing some decompositions and rescalings one got measure pieces
for which one needed to prove the Fourier restriction estimate for the exponent
\begin{align}
\label{section_strichartz_mixed_critical_exponent_stronger}
\Big(\frac{1}{p_1'}, \frac{1}{p_3'}\Big) = \Big(\frac{1}{6}, \frac{1}{3}\Big).
\end{align}
In the current case without the mitigating factor it turns out that we shall get the same measure pieces,
but for which we need to prove the Fourier restriction estimate for the exponent \eqref{section_strichartz_mixed_critical_exponent}.
Thus, if we have the Fourier restriction estimate for the exponent \eqref{section_strichartz_mixed_critical_exponent_stronger},
then the Fourier restriction for \eqref{section_strichartz_mixed_critical_exponent} is obtained by interpolating with the result for
\begin{align*}
\Big(\frac{1}{p_1'}, \frac{1}{p_3'}\Big) = \Big(\frac{1}{6}, 0\Big),
\end{align*}
which one can obtain by applying the 2-dimensional Fourier restriction result for curves with nonvanishing curvature.

These stronger estimates for the rescaled measure pieces do not contradict
the necessary conditions obtained by Knapp-type examples in \cite{Pa19}
since the information on the exponents and the Newton height of $\phi$ is consumed in the rescaling procedure
(which is different in this section and in Subsection \ref{subsection_strichartz_mitigating_vi}).

\medskip

Let us begin with some preliminary reductions.
By the results from \cite[Section 4.2]{Pa19}, instead of considering the whole measure \eqref{section_strichartz_mixed_measure},
we may reduce ourselves to considering the part near the principal root jet in the half plane $\{(x_1, x_2) \in \R^2 : x_1 \geq 0\}$:
\begin{align*}
\langle \nu^{\rho_1}, f \rangle = \int_{x_1 \geq 0} f(x,\phi_\text{loc}(x)) \, a(x) \, \rho_1(x) \mathrm{d}x,
\end{align*}
where
\begin{align*}
\rho_1(x) = \chi_0  \Big(\frac{x_2- \psi(0) x_1^2}{\varepsilon x_1^2}\Big)
\end{align*}
for an $\varepsilon$ which we can take to be as small as we want.

The next step is to use a Littlewood-Paley argument in the $(x_1,x_2)$-plane and the scaling by $\kappa$ dilations
\begin{align*}
\delta_r^{\kappa}(x) = (r^{\kappa_1} x_1, r^{\kappa_2} x_2)
\end{align*}
where $\kappa \coloneq (1/(2 k),1/k)$ is the weight associated to the principal face of $\phi_\text{loc}$.
Then one is reduced to proving \eqref{section_strichartz_mixed_FRP} for the measures
\begin{align*}
\langle \nu_j, f \rangle = \int f(x,\phi(x,\delta)) \, a(x,\delta) \, \mathrm{d}x,
\end{align*}
uniformly in $j$, where the function $\phi(x,\delta)$ has the form
\begin{align*}
\phi(x,\delta) \coloneqq (x_2 - x_1^2 \psi(\delta_1 x_1))^k r(\delta_1 x_1, \delta_2 x_2),
\end{align*}
where
\begin{align*}
\delta = (\delta_1,\delta_2) \coloneqq (2^{-\kappa_1 j}, 2^{-\kappa_2 j}).
\end{align*}
Note that we can take $|\delta| \ll 1$.
The amplitude $a(x,\delta) \geq 0$ is a smooth function of $(x,\delta)$ supported where
\begin{align*}
x_1 \sim 1 \sim |x_2|.
\end{align*}
We may additionally assume $|x_2 - x_1^2 \psi(0)| \ll 1$ due to $\rho_1$,
and by compactness we may in fact reduce ourselves to assuming $|(x_1,x_2) - (v_1^0, v_2^0)| \ll 1$ for some $(v_1^0, v_2^0) \in \R^2$ with $v_1^0 \sim 1$.

The following step is to again apply the Littlewood-Paley theorem,
but this time in the $x_3$ direction
(again, for the mixed norm Littlewood-Paley theory see \cite{Liz70}),
and reduce the Fourier restriction problem for $\nu_j$ to the Fourier restriction for the measures
\begin{align*}
\langle \nu_{\delta, l}, f \rangle = \int f(x,\phi(x,\delta)) \, \chi_1(2^{k l} \phi(x,\delta)) \, a(x,\delta) \, \mathrm{d}x,
\end{align*}
i.e., we need to prove
\begin{align*}
\Vert \FT f \Vert_{L^2(\mathrm{d} \nu_{\delta, l})} \lesssim \Vert f \Vert_{L^{\mathfrak{p}_3}_{x_3}(L^{\mathfrak{p}_1}_{(x_1,x_2)})}, \qquad f \in \mathcal{S}(\R^3),
\end{align*}
uniformly in $l$ and $\delta$, where $l \gg 1$ and $|\delta| \ll 1$.

Finally, we perform a change of coordinates and a rescaling.
Namely, after substituting $(x_1,x_2) \mapsto (x_1, 2^{-l}x_2+x_1^2 \psi(\delta_1 x_1))$ we get
\begin{align*}
\langle \nu_{\delta, l}, f \rangle =
  2^{-l} \int f(x_1, 2^{-l}x_2+x_1^2 \psi(\delta_1 x_1), 2^{-k l} \phi^a(x, \delta, l)) \, a(x, \delta, l) \, \mathrm{d}x,
\end{align*}
where
\begin{align*}
a(x, \delta, l) &\coloneq \chi_1(\phi^a(x,\delta,l)) \, a(x_1, 2^{-l}x_2+x_1^2 \psi(\delta_1 x_1), \delta), \\
\phi^a(x, \delta, l) &\coloneqq x_2^k r(\delta_1 x_1, \delta_2 (2^{-l}x_2+x_1^2 \psi(\delta_1 x_1))).
\end{align*}
Note that $a(x, \delta, l)$ is again supported in a domain where $x_1 \sim 1 \sim |x_2|$.
Rescaling we obtain that the Fourier restriction estimate for $\nu_{\delta, l}$ is equivalent to the estimate
\begin{align*}
\Vert \FT f \Vert_{L^2(\mathrm{d} \tilde{\nu}_{\delta, l})} \lesssim \Vert f \Vert_{L^{\mathfrak{p}_3}_{x_3}(L^{\mathfrak{p}_1}_{(x_1,x_2)})}, \qquad f \in \mathcal{S}(\R^3),
\end{align*}
for the measure
\begin{align}
\label{section_strichartz_mixed_measure_rescaled}
\langle \tilde{\nu}_{\delta, l}, f \rangle =
  \int f(x_1, 2^{-l}x_2+x_1^2 \psi(\delta_1 x_1), \phi^a(x, \delta, l)) \, a(x, \delta, l) \, \mathrm{d}x.
\end{align}

As mentioned, since this measure is of the same form as \eqref{subsection_strichartz_mitigating_vi_final_measure},
we are interested in proving the stronger estimate
\begin{align*}
\Vert \FT f \Vert_{L^2(\mathrm{d} \tilde{\nu}_{\delta, l})}
   \lesssim \Vert f \Vert_{L^{\tilde{\mathfrak{p}}_3}_{x_3}(L^{\tilde{\mathfrak{p}}_1}_{(x_1,x_2)})}, \qquad f \in \mathcal{S}(\R^3),
\end{align*}
where
\begin{align*}
\Big(\frac{1}{\tilde{\mathfrak{p}}_1'}, \frac{1}{\tilde{\mathfrak{p}}_3'}\Big) \coloneq \Big(\frac{1}{6}, \frac{1}{3}\Big).
\end{align*}
Note that we automatically have the estimate for
\begin{align*}
\Big(\frac{1}{p_1'}, \frac{1}{p_3'}\Big) = \Big(\frac{1}{6}, 0 \Big)
\end{align*}
by a classical result of Zygmund \cite{Zyg74},
since $x_1 \mapsto (x_1, 2^{-l}x_2 + x_1^2 \psi(\delta_1 x_1))$ is a curve
with curvature bounded from below uniformly in $|x_2| \sim 1$, $2^{-l} \ll 1$, and $\delta_1 \ll 1$.

\subsection{Fourier restriction for the non-adapted case: spectral decomposition}
\label{subsection_strichartz_mixed_FRP_spectral}

We begin by performing a spectral decomposition of the measure $\tilde{\nu}_{\delta, l}$.
For $(\lambda_1, \lambda_2, \lambda_3)$ dyadic numbers with $\lambda_i \geq 1$, $i = 1,2,3$,
we consider localized measures $\nu_l^\lambda$ defined through
\begin{align}
\begin{split}
\label{subsection_strichartz_mixed_FRP_spectral_measure_piece}
\widehat{\nu_l^\lambda} (\xi) =
   &\chi_1\Big(\frac{\xi_1}{\lambda_1}\Big) \, \chi_1\Big(\frac{\xi_2}{\lambda_2}\Big) \, \chi_1\Big(\frac{\xi_3}{\lambda_3}\Big)\\
   &\times \int e^{-i \Phi(x, \delta, l, \xi)} \, a(x,\delta, l) \, \chi_1(x_1) \, \chi_1(x_2) \, \mx,
\end{split}
\end{align}
where the phase function is
\begin{align}
\label{subsection_strichartz_mixed_FRP_spectral_phase}
\Phi(x, \delta, l, \xi) \coloneqq
\xi_3 \phi^a(x, \delta, l) + 2^{-l} \xi_2 x_2 + \xi_2 x_1^2 \psi(\delta_1 x_1) + \xi_1 x_1.
\end{align}
By an abuse of notation,
above whenever $\lambda_i = 1$, we consider the cutoff function $\chi_1(\xi_i/\lambda_i)$
to be actually $\chi_0(\xi_1/\lambda_1)$, i.e., it localizes so that $|\xi_i| \lesssim 1$.

Let us introduce the convolution operators $\tilde{T}_{\delta, l} f \coloneqq f * \widehat{\tilde{\nu}}_{\delta, l}$
and $T^\lambda_l f \coloneqq f * \widehat{\nu^\lambda_l}$.
Then we need to show
\begin{align*}
\Vert \tilde{T}_{\delta, l} \Vert_{L^{\tilde{\mathfrak{p}}} \to L^{\tilde{\mathfrak{p}}'}} \lesssim 1,
\end{align*}
since $\tilde{T}_{\delta, l}$ is the ``$R^* R$'' operator,
i.e., one has $\tilde{T}_{\delta_l} = (\tilde{R}_{\delta, l})^* \tilde{R}_{\delta, l}$
if $\tilde{R}_{\delta, l}$ denotes the Fourier restriction operator with respect to the surface measure $\tilde{\nu}_{\delta, l}$.
Therefore, the boundedness of $\tilde{T}_{\delta, l}$ is equivalent to the boundedness of $\tilde{R}_{\delta, l}$ by H\"older's inequality.

Our first step shall be to reduce the problem to the case when $\lambda_2 \ll 2^l$.
In order to achieve this we split the Fourier transform of $\tilde{\nu}_{\delta, l}$ as
\begin{align}
\label{subsection_strichartz_mixed_FRP_split}
\widehat{\tilde{\nu}}_{\delta, l}
= (1-\chi_0(2^{-l} \xi_2 )) \widehat{\tilde{\nu}}_{\delta, l}
+ \chi_0(2^{-l} \xi_2 )\widehat{\tilde{\nu}}_{\delta, l},
\end{align}
where we assume that $\chi_0$ is supported in a sufficiently small neighbourhood of the origin,
and we denote the respective operators for the respective terms by $T_I$ and $T_{II}$.

For the first term in \eqref{subsection_strichartz_mixed_FRP_split}
and its operator $T_I$ one uses Lemma \ref{lemma_simple_estimates} above, though with a slight modification.
First, since on the support of $(1-\chi_0(2^{-l} \xi_2 )) \widehat{\tilde{\nu}}_{\delta, l}$
we have $|\xi_2| \gtrsim 2^l$,
one can easily show by using \eqref{subsection_strichartz_mixed_FRP_spectral_phase} that now
\begin{align*}
|(1-\chi_0(2^{-l} \xi_2 )) \widehat{\tilde{\nu}}_{\delta, l}| \lesssim 2^{-l/2} (1+|\xi_3|)^{-1},
\end{align*}
as the ``worst case'' is when $|\xi_1| \sim |\xi_2|$ and $|\xi_3| \sim |2^{-l} \xi_2|$,
in which case we use stationary phase in both $x_1$ and $x_2$
(and in other cases we get a better decay by integrating by parts).
In order to obtain the Plancherel estimate $L^1(\R;L^2(\R^2)) \to L^\infty(\R;L^2(\R^2))$ in Lemma \ref{lemma_simple_estimates}
for $T_I$ it suffices to prove it for $T_{II}$ and $\tilde{T}_{\delta, l}$
(formally, one needs to actually consider the $L^2(\R^2) \to L^2(\R^2)$ estimate for a fixed $\xi_3$).
For the operator $\tilde{T}_{\delta, l}$ we get the bound $2^l$ in the same way as in Lemma \ref{lemma_simple_estimates}.
The main fact to notice is that in \eqref{section_strichartz_mixed_measure_rescaled}
the Jacobian of $(x_1,x_2) \mapsto (x_1, 2^{-l}x_2+x_1^2 \psi(\delta_1 x_1))$ is of size $2^{-l}$.
One now gets the same estimate automatically for $T_{II}$ since the $L^1$ norm of the Fourier transform
of the cutoff function $\chi_0(2^{-l} \xi_2)$ is of size $\sim 1$.
The estimate $L^{\tilde{\mathfrak{p}}} \to L^{\tilde{\mathfrak{p}}'}$ estimate for $T_{I}$ follows
with constant of size $\sim 1 = (2^{-l/2})^{2/3} (2^{l})^{1/3}$.

For the operator $T_{II}$ we shall use the spectral decomposition
\eqref{subsection_strichartz_mixed_FRP_spectral_measure_piece}
where we may now assume $\lambda_2 \ll 2^l$.
Recall that for an operator of the form $Tf = f * \widehat{g}$
the $A_1$ constant from \eqref{section_strichartz_mixed_basic_estimates} is bounded by the $L^\infty$ norm of $\widehat{g}$,
and the $A_2$ constant is bounded by the $L^\infty$ norm of $g$.
If we now furthermore have that $\widehat{g}$ has its support in the $\xi_3$ coordinate localized at $|\xi_3| \lesssim \lambda_3$,
then by \cite[Lemma 3.9]{Pa19} we have the estimate
\begin{align*}
\Vert T \Vert_{L^p \to L^{p'}} &\lesssim A_{1} \lambda_3^{1/2}, \qquad \text{for} \,\, \Big(\frac{1}{p_1'}, \frac{1}{p_3'}\Big) = \Big(0, \frac{1}{4}\Big),
\end{align*}
and so by interpolation we get
\begin{align}
\label{section_strichartz_mixed_basic_interpol}
\Vert T \Vert_{L^{\tilde{\mathfrak{p}}} \to L^{\tilde{\mathfrak{p}}'}} &\lesssim A_{1}^{2/3} A_2^{1/3} \lambda_3^{1/3}.
\end{align}

The inverse Fourier transform of \eqref{subsection_strichartz_mixed_FRP_spectral_measure_piece} is
\begin{align}
\begin{split}
\label{subsection_strichartz_mixed_FRP_spectral_space_form}
\nu_l^\lambda (x) = \lambda_1 \lambda_2 \lambda_3 \int &\widecheck{\chi}_1 (\lambda_1 (x_1-y_1)) \, \widecheck{\chi}_1 (\lambda_2 (x_2-2^{-l}y_2-y_1^2\psi(\delta_1 y_1)))\\
   &\times \widecheck{\chi}_1(\lambda_3 (x_3-\phi^a(y,\delta,l))) \, a(y, \delta, l) \, \chi_1(y_1) \, \chi_1(y_2) \, \my.
\end{split}
\end{align}
One can consider either the substitution $(z_1, z_2) = (\lambda_1 y_1, \lambda_2 2^{-l} y_2)$,
or the substitution $(z_1, z_2) = (\lambda_1 y_1, \lambda_3 \phi^a(y, \delta, l))$
(in order to carry this out one needs to consider the cases $y_2 \sim 1$ and $y_2 \sim -1$ separately),
and get
\begin{align*}
\Vert \nu_j^\lambda \Vert_{L^\infty} \lesssim \min\{2^{l} \lambda_3, \lambda_2 \}.
\end{align*}
But now since $\lambda_2 \ll 2^l$ we may take $A_2 \coloneqq \lambda_2$.

It remains to calculate the $L^\infty$ bound for the $\widehat{\nu_l^\lambda}$ function.
This we can do by estimating the oscillatory integral in \eqref{subsection_strichartz_mixed_FRP_spectral_measure_piece}.
As the calculations for the oscillatory integral in this case are almost identical to the ones in \cite[Subsection 5.5]{Pa19},
we shall only briefly explain
the case when $\lambda_1 \sim \lambda_2$, $2^{-l} \lambda_2 \ll \lambda_3 \ll \lambda_2$,
corresponding to Case 6 in \cite[Subsection 5.5]{Pa19}.
In all the other cases one gets that one can sum absolutely in the operator norm the operator pieces $T_l^\lambda$.

Let us remark that since $\lambda_2 \ll 2^l$, the case when $\lambda_1 \sim \lambda_2$, $2^{-l} \lambda_2 \sim \lambda_3$,
corresponding to Case 4 in \cite[Subsection 5.5]{Pa19}, does not appear anymore.
This is critical since in this case one would not have absolute summability,
nor would the complex interpolation method developed in \cite{IM16} work.
This is the reason why we needed to consider $T_{I}$ and $T_{II}$ separately.

\medskip

{\bf Case $\lambda_1 \sim \lambda_2$ and $2^{-l} \lambda_2 \ll \lambda_3 \ll \lambda_2$.}
As was obtained in \cite[Subsection 5.5]{Pa19}, we have
\begin{align}
\label{subsection_strichartz_mixed_FRP_spectral_complex_Fourier_estimate_2}
\Vert \widehat{\nu_l^\lambda} \Vert_{L^\infty} \lesssim \lambda_1^{-1/2} \lambda_3^{-N}
\end{align}
for any $N > 0$, that is, we have $A_1 = \lambda_1^{-1/2} \lambda_3^{-N}$, and recall that $A_2 = \lambda_2$,
Therefore \eqref{section_strichartz_mixed_basic_interpol} gives
\begin{align*}
\Vert T_l^\lambda \Vert_{L^{\tilde{\mathfrak{p}}} \to L^{\tilde{\mathfrak{p}}'}} &\lesssim \lambda_3^{-N}.
\end{align*}
In order to be able to sum in $\lambda_1 \sim \lambda_2$ we need to use the complex interpolation method from \cite{IM16}.
For a fixed $\lambda_3$ and $\zeta$ a complex number we define the measure $\mu_\zeta^{\lambda_3}$ by
\begin{align*}
\mu_\zeta^{\lambda_3} \coloneq \gamma(\zeta) \sum_{\lambda_1, \lambda_2} \lambda_1^{(1-3\zeta)/2} \nu_l^\lambda,
\end{align*}
where the sum is over $\lambda_3 \ll \lambda_2 \ll 2^l$ and $\lambda_1 \sim \lambda_2$,
and where $\gamma(\zeta) = 2^{-3(\zeta-1)/2} - 1$.
We denote the associated convolution operator by $T_\zeta^{\lambda_3}$
and we recover with $\zeta = 1/3$ the operator we want to estimate.

By a complex interpolation argument it suffices to show that
\begin{align*}
\Vert T_{it}^{\lambda_3} \Vert_{L^p \to L^{p'}} &\lesssim \lambda_3^{-N}, & \text{for}& \,\, \Big(\frac{1}{p_1'}, \frac{1}{p_3'}\Big) = \Big(0, \frac{1}{4}\Big), \nonumber \\
\Vert T_{1+it}^{\lambda_3} \Vert_{L^p \to L^{p'}} &\lesssim 1, & \text{for}& \,\, \Big(\frac{1}{p_1'}, \frac{1}{p_3'}\Big) = \Big(\frac{1}{2}, \frac{1}{2}\Big).
\end{align*}
for some $N > 0$, with constants uniform in $t \in \R$.
The first estimate follows directly from the fact that
$\widehat{\nu_l^\lambda}$ have essentially disjoint supports with respect to $\lambda$
and the estimate \eqref{subsection_strichartz_mixed_FRP_spectral_complex_Fourier_estimate_2}
(see Lemma 3.8, (i), in \cite{Pa19}),
and for the other bound we need to estimate the $L^\infty$ norm of the corresponding sum of the expressions
\eqref{subsection_strichartz_mixed_FRP_spectral_space_form}.
The proof is the same as in \cite[Subsection 5.5, Case 6]{Pa19},
up to the formal difference in the function $\phi^a$ which here behaves like $y_2^k$, and there like $y_2^2$.
Since the domain of integration in \eqref{subsection_strichartz_mixed_FRP_spectral_space_form}
is $|y_2| \sim 1$, this is not essential.
This finishes (the sketch of) the proof of the Fourier restriction for the non-adapted case,
and also the proof of Theorem \ref{theorem_Strichartz}.



\begin{appendices}
\section{Application of the Christ-Kiselev lemma}
\label{appendix_CKlemma}

Recall that we consider the nonhomogeneous initial problem
\begin{align*}
\begin{cases}
(\partial_t - i \phi(D))u(x,t) &= F(x,t), \quad (x,t) \in \R^2 \times (0,\infty),\\
\qquad \qquad \,\,\quad u(x,0) &= G(x), \quad \quad\,x \in \R^2,
\end{cases}
\end{align*}
for $F \in \mathcal{S}(\R^3)$, $G \in \mathcal{S}(\R^2)$,
where $\phi$, $\mathcal{W}$, and $(p_1,p_3) \in (1,2)^2$ are either
as in Theorem \ref{theorem_Strichartz_mitigating} or Theorem \ref{theorem_Strichartz},
and where we additionally assume $\rho \in \{0,1\}$.
Note that according to Remark \ref{remark_integrability} the weight $\mathcal{W}$ is locally integrable in $\R^2$.
The formula for the solution of the above equation is obtained through the Duhamel principle:
\begin{align}
\label{appendix_Duhamel}
u(x,t) = (e^{i\phi(D)t}G)(x) + \int_0^t (e^{i\phi(D)(t-s)}F(\cdot,s))(x,s) \mathrm{d}s.
\end{align}
Note that $u \in C^\infty(\R^2 \times \R)$.

We consider the following two surface measures (the second defined as in \eqref{section_intro_restriction_measure}):
\begin{align*}
\langle \mu_\phi, f \rangle &= \int_{\R^2\setminus\{0\}} f(x_1,x_2,\phi(x_1,x_2)) \, \mathrm{d}x, \\
\langle \mu, f \rangle &= \int_{\R^2\setminus\{0\}} f(x_1,x_2,\phi(x_1,x_2)) \, \mathcal{W}(x_1,x_2) \, \mathrm{d}x,
\end{align*}
and we assume that the Fourier restriction estimate \eqref{section_intro_restriction_estimate} for $\mu$ holds true for $(p_1,p_3) \in (1,2)^2$.
One can easily check that
\begin{align*}
(e^{i\phi(D)t}G)(x) = \IFT((\FT G) \mathrm{d}\mu_\phi) (x,t) = \IFT(\mathcal{W}^{-1} (\FT G) \mathrm{d}\mu) (x,t),
\end{align*}
and so this is precisely the Fourier extension operator of $\mu$ applied to the function $\mathcal{W}^{-1} \FT G$.
We can therefore bound the $L_{t}^{p_3'}(L_{(x_1,x_2)}^{p_1'})$ norm of this expression
by the $L^2(\mathrm{d}\mu)$ norm of $\mathcal{W}^{-1} \FT G$.

It remains to estimate the $L_{t}^{p_3'}(L_{(x_1,x_2)}^{p_1'})$ norm of the second term in \eqref{appendix_Duhamel}.
It turns out that the operator associated to this second term is closely related to the operator $f \mapsto f * \IFT \mu$
(which we know is bounded from $L_{t}^{p_3}(L_{(x_1,x_2)}^{p_1})$ to $L_{t}^{p_3'}(L_{(x_1,x_2)}^{p_1'})$
since this is the corresponding $R^* R$ operator).
Namely, one can check that
\begin{align*}
\int_0^\infty (e^{i\phi(D)(t-s)}F(\cdot,s))(x,s) \mathrm{d}s
   &= \Big( (F \, \chi_{(0,\infty)}(s)) * (\IFT \mu_\phi) \Big) (x,t),
\end{align*}
and therefore it remains to pass from $\mu_\phi$ to $\mu$ and
to pass from integrating over $(0,\infty)$ in $s$ to integrating over $(0,t)$ in $s$.

In order to do this, our first step is to use the Littlewood-Paley theorem in the $x$-direction
so that our problem is reduced to proving the boundedness of the operator
\begin{align}
\label{appendix_LP_operator}
\int_0^t (e^{i\phi(D)(t-s)} \, \eta_j(D) \,F(\cdot,s))(x,s) \mathrm{d}s
\end{align}
where $(\eta_j)_{j \in \Z}$, $\eta_j = \eta \circ \delta_{2^{-j}}$,
constitutes a partition of unity in $\R^2\setminus\{0\}$ (as in \eqref{section_prelim_partition} in Subsection \ref{subsection_prelim_scaling})
respecting the $\alpha$-mixed homogeneous dilation $\delta_{2^{-j}}$ defined in \eqref{section_intro_alpha_dilations}.
By unwinding the definition of the operator in \eqref{appendix_LP_operator} and inserting the $\mathcal{W}$ factor,
one obtains the expression (up to a universal constant)
\begin{align}
\label{appendix_LP_operator_expanded}
\int_0^t \int_{\R^2} \Bigg( \int_{\R^2} e^{i(x-y)\cdot \xi+i(t-s)\phi(\xi)} \eta_j(\xi)
\mathcal{W}(\xi) \mathrm{d}\xi \Bigg)  F_{\mathcal{W}^{-1}}(y,s) \mathrm{d}y \mathrm{d}s,
\end{align}
where $F_{\mathcal{W}^{-1}} = \IFT_{(x_1,x_2)} (\mathcal{W}^{-1} \FT_{(x_1,x_2)} F)$.
The expression within the brackets defines a convolution kernel $K_j(t-s; x-y)$ whose associated operator $T_j(t-s)$ in the $x$ variable
is a bounded mapping from $L^q(\R^2)$ to $L^{q'}(\R^2)$ for any $q \in [1,2]$ (since the integrand in the brackets is an $L^\infty_c(\R^2)$ function).
Using the dominated convergence theorem one can get strong continuity of the operator valued function $T_j : \R \to \mathcal{L}(L^{q}(\R^2); L^{q'}(\R^2))$
(which in turn, by the uniform boundedness principle, implies joint continuity $T_j : \R \times L^{q}(\R^2) \to L^{q'}(\R^2)$).

We may now apply the Christ-Kiselev lemma (for a proof of this variant see e.g. \cite[Chapter IV, Lemma 2.1]{So08}):
\begin{lemma}
Let $Y$ and $Z$ be separable Banach spaces and let $K : \R \to \mathcal{L}(Y,Z)$
be a continuous function from the real numbers
to the space of bounded linear mappings $Y \to Z$ equipped with the strong operator topology.
If the operator defined by
\begin{align*}
(T f) (t) \coloneqq \int_{\R} K(t-s) f(s) \mathrm{d}s
\end{align*}
is a bounded mapping from $L^{q}(\R,Y)$ to $L^{q'}(\R,Z)$ for some $q \in (1,2)$,
then the operator defined by
\begin{align*}
(W f) (t) \coloneqq \int_{-\infty}^{t} K(t-s) f(s) \mathrm{d}s
\end{align*}
is also a bounded mapping from $L^{q}(\R,Y)$ to $L^{q'}(\R,Z)$, and in particular
\begin{align*}
\Vert W \Vert_{L^{q}(\R,Y) \to L^{q'}(\R,Z)} \lesssim_{q} \Vert T \Vert_{L^{q}(\R,Y) \to L^{q'}(\R,Z)}.
\end{align*}
\end{lemma}
Then we get that
the $L_{t}^{p_3}(L_{(x_1,x_2)}^{p_1}) \to L_{t}^{p_3'}(L_{(x_1,x_2)}^{p_1'})$ boundedness of the operator in \eqref{appendix_LP_operator_expanded}
(acting on $F_{\mathcal{W}^{-1}}$)
is implied by the $L_{t}^{p_3}(L_{(x_1,x_2)}^{p_1}) \to L_{t}^{p_3'}(L_{(x_1,x_2)}^{p_1'})$ boundedness of the operator
\begin{align*}
\int_0^\infty \int_{\R^2} \Bigg( \int_{\R^2} e^{i(x-y)\cdot \xi+i(t-s)\phi(\xi)} \eta_j(\xi)
\mathcal{W}(\xi) \mathrm{d}\xi \Bigg) &F_{\mathcal{W}^{-1}}(y,s) \mathrm{d}y \mathrm{d}s \\
   &= \Big( (F_{\mathcal{W}^{-1}} \, \chi_{(0,\infty)}(s)) * (\IFT \mu_j) \Big) (x,t),
\end{align*}
with essentially the same operator constant bound (up to a multiplicative factor which depends only on $p_3 \in (1,2)$).
Here $\mu_j$ is the localized measure defined in the same way as in \eqref{section_prelim_local_measure},
and recall that this convolution operator is bounded (uniformly in $j$).
This finishes the proof of Corollary \ref{corollary_Strichartz_equation}.

\end{appendices}


\end{document}